\documentclass{amsart}
\usepackage[hidelinks]{hyperref} 
\usepackage[utf8]{inputenc} 
\usepackage[T1]{fontenc} 
\usepackage[english]{babel}
\usepackage{tikz} 
\usepackage{cite}
\usepackage{pgfplots} 
\pgfplotsset{compat=1.17}
\usepgfplotslibrary{fillbetween}
\usepackage{mathrsfs}
\usepackage{color}
\newcommand{\CB}[1]{{\color{blue}#1}}
\newcommand{\CR}[1]{{\color{red}#1}}
\newcommand{\CG}[1]{{\color{green}#1}}
\newcommand{\CBR}[1]{{\color{brown}#1}}
\newcommand{\CV}[1]{{\color{violet}#1}}
\newcommand{\CT}[1]{{\color{teal}#1}}

\usepackage{tikz} 
\usepackage{pgfplots} 
\pgfplotsset{compat=1.16}
\usepackage{tikz-cd} 
\tikzcdset{arrow style=tikz, arrows={thick}, diagrams={>=stealth}}
\numberwithin{equation}{section} 
\numberwithin{figure}{section} 
\numberwithin{table}{section} 
\theoremstyle{plain} 
\newtheorem{theorem}{Theorem}[section]
\newtheorem{lemma}[theorem]{Lemma}
\newtheorem{corollary}[theorem]{Corollary}

\newtheorem*{theorem*}{Theorem}

\theoremstyle{remark}

\newcommand{\dx}{\, dx}

\newcommand{\bb}[1]{\mathbb{#1}}

\author{Sarah May Instanes} 
\address{Department of Mathematical Sciences, Norwegian University of Science and Technology (NTNU), NO-7491 Trondheim, Norway} 
\email{sarah.m.instanes@ntnu.no}
\date{\today}
\title{An optimization problem and point-evaluation in Paley--Wiener spaces}
\begin{document} 
	\begin{abstract}
		We study the constant $\mathscr{C}_p$ defined as the smallest constant $C$ such that $|f(0)|^p \leq C\|f\|_p^p$ holds for every function $f$ in the Paley--Wiener space $PW^p$. Brevig, Chirre, Ortega-Cerdà, and Seip have recently shown that $\mathscr{C}_p<p/2$ for all $p>2$. We improve this bound for $2<p \leq 5$ by solving an optimization problem. 
	\end{abstract}
	
	\maketitle
	\section{introduction}\label{sec:intro}
	Let $PW^p$ denote the Paley--Wiener space consisting of all functions $f$ in $L^p(\bb{R})$ where $f$ extends to an entire function of exponential type at most $\pi$. In this paper we study the constant $\mathscr{C}_p$ defined as the smallest possible constant $C$ such that 
	\[|f(0)|^p \leq C\|f\|_p^p\]
	holds for every function $f$ in  $PW^p$. The only known value of $\mathscr{C}_p$ is $\mathscr{C}_2=1$.
	
	The case $p=1$ has applications in number theory and is well studied. The best numerical results in this case are due to Hörmander and Bernhardsson \cite{Hormander} who found that 
	\[0.5409288219 \leq \mathscr{C}_1 \leq 0.5409288220.\] Bondarenko, Ortega-Cerdà, Radchenko, and Seip \cite{Seip} have recently given new results on the zeroes of the corresponding extremal function. For general $p$, the problem of determining or estimating $\mathscr{C}_p$ and solving the related extremal problem
	\begin{equation}\label{eq:extprob}
		\frac{1}{\mathscr{C}_p} = \inf_{f \in PW^p} \{\|f\|_p^p : f(0)=1 \}.
	\end{equation}
	is stated by Lubinsky in \cite[Problem 4, Problem 5]{Lubinsky2}. The constant $\mathscr{C}_p$ also appears in a problem concerning Christoffel functions studied by Levin and Lubinsky in \cite{Lubinsky}. Using the power trick and the fact $\mathscr{C}_2=1$ it is easy to show that $\mathscr{C}_p \leq \lceil p/2 \rceil$ for $p >2$, see for instance \cite{Gorbachev, Ibragimov60, Korevaar49}. This was the best known upper bound on $\mathscr{C}_p$ for $p>2$ until a recent in-depth study of the extremal problem \eqref{eq:extprob} by Brevig, Chirre, Ortega-Cerdà, and Seip \cite{Brevig} led to the improved bound $\mathscr{C}_p < p/2$. 
	
	We will solve an optimization problem which gives a further improvement on the upper bound for $\mathscr{C}_p$ in the range $2 \leq p \leq 5$. Given a sequence of strictly increasing positive numbers $\tau=(\tau_{n})_{n \geq 0}$ and $0<p<\infty$ we define
	\[E_p(\tau) = \int_{0}^{\infty}\left(\max\left(0, K_p(\tau; x) \right)\right)^2\, dx,\]
	where 
	\[K_p(\tau; x)= \sum_{n=0}^{\infty} \chi_{(\tau_n,\tau_{n+1})}(x)\frac{\sin{\frac{p}{2}\pi (x-n)}}{\pi x},\]
	and $\chi_{(a,b)}$ denotes the characteristic function on the interval $(a,b)$. 
	The optimization problem we consider is then given as
	\begin{equation}\label{eq:optprob}
		\mathscr{E}_p(\delta_1, \delta_2)= \sup_{\tau \in T(\delta_1, \delta_2)} E_p(\tau),
	\end{equation}
	where $T(\delta_1, \delta_2)$ is a set of sequences subject to two separation conditions $\delta_1, \delta_2 >0$, namely
	\[T(\delta_1, \delta_2) = \{\tau=(\tau_n)_{n \geq 0}: \tau_0=0 \text{, } \tau_1 \geq \delta_1 \text{ and } \tau_{n+1}-\tau_n \geq \delta_2 \text{ for all } n \geq 1\}.\]
	
	The problems \eqref{eq:extprob} and \eqref{eq:optprob} are closely connected. For $1 \leq p <\infty$ solutions to the extremal problem \eqref{eq:extprob} exist and are unique, and thus necessarily even (see e.g. \cite[Section 3]{Brevig}). We will refer to this solution as $\varphi_p$, and denote by $(t_n)_{n \geq 1}$ the positive zeroes of $\varphi_p$ in increasing order. In \cite[Section 5]{Brevig} the separation of the zeroes of $\varphi_p$ is studied, and it is shown that for $2 \leq p \leq 4$, we have $t_1 \geq 2/\pi$ and $t_{n+1} \geq t_n + 2/3$ for all $n \geq 1$. For $p>4$ it is shown that $t_1 \geq 1/2$ and $t_{n+1} \geq t_n + 3/5$ for all $n \geq 1$. The upper bound on $\mathscr{C}_p$ for $2<p\leq 4$ obtained in \cite{Brevig} is a consequence of the following result, which connects the separation of the zeroes of $\varphi_p$ to the separation conditions $\delta_1$ and $\delta_2$ in \eqref{eq:optprob}.
	\begin{theorem}[Brevig--Chirre--Ortega-Cerdà--Seip]\label{thm:BCOS}
		\hfill
		\begin{itemize}
			\item[(a)] If $2 \leq p \leq 4$, then $\mathscr{C}_p \leq 2 \mathscr{E}_p(2/\pi, 2/3)$.
			\item[(b)] If $4<p<\infty$, then $\mathscr{C}_p \leq 2 \mathscr{E}_p(1/2, 3/5)$. 
		\end{itemize}
	\end{theorem}
	
	If one were to show that $t_1 \geq \delta_1$ and $t_{n+1} - t_n \geq \delta_2$ for all $n \geq 1$ one could replace Theorem \ref{thm:BCOS} with $\mathscr{C}_p \leq 2 \mathscr{E}_p(\delta_1, \delta_2)$. This could theoretically give a better bound for $\mathscr{C}_p$. 
	However, we know that $\delta_2 \leq 1$, as the extremal functions solving \eqref{eq:extprob} necessarily have type exactly $\pi$. Moreover, as mentioned in \cite[Section 9]{Brevig}, it seems reasonable to expect $\delta_1 \leq 1$ for $p \geq 2$. It follows that the best bound one can hope to attain using the method from \cite{Brevig} is $\mathscr{C}_p \leq 2 \mathscr{E}_p(1,1)$. This method uses only the separation conditions $\delta_1$ and $\delta_2$. However, if we had more information on (some of) the zeroes of the extremal function, we could easily adapt the method to take this into account by reducing the set of sequences over which we take the supremum in \eqref{eq:optprob}. In particular, we note that for $2 \leq p \leq 4$ if we assume that $n-1/2+1/p \leq t_n \leq n$ (as suggested in \cite[Section 9]{Brevig}), we still get the same upper bound for $\mathscr{C}_p$. To the contrary, for $4<p<5$ assuming $n-1/2+1/p \leq t_n \leq n$ leads to the better upper bound $\mathscr{C}_p \leq 2 E_p((n)_{n \geq 0})$. This illustrates that the method developed in \cite{Brevig} is better suited for $2 \leq p \leq 4$ than for $4<p<5$. 
	
	Motivated by Theorem \ref{thm:BCOS} we study the problem of determining $\mathscr{E}_p(\delta_1, \delta_2)$ under certain related separation conditions $\delta_1$ and $\delta_2$. We obtain the following results.
	
	\begin{theorem}\label{thm:new:Ep2to4}
		Fix $2 \leq p \leq 4$ and let $\min(2/\pi, 2/p) \leq \delta_1 \leq 1$ and $2/3 \leq \delta_2 \leq 1$. Then \[\mathscr{E}_p(\delta_1, \delta_2)=\sum_{n=0}^{\infty} \int_{n}^{n+2/p} \frac{\sin^2\frac{p}{2}\pi(x-n)}{\pi^2 x^2} \dx.\]
		For $2 \leq p <4$ the supremum is attained if and only if $\tau$ in $T(\delta_1, \delta_2)$ is a sequence such that $\tau_n$ is in the interval $[n-1+2/p,n]$ for all $n \geq 1$. If $p=4$, then the supremum is attained for all sequences $\tau$ in $T(\delta_1, \delta_2)$. 
	\end{theorem}
	
	\begin{figure}
		\centering
		\begin{tikzpicture}[scale=1.5]
			\begin{axis}
				[axis equal image,
				axis lines=middle,
				axis line style=thin,
				xmin=2-2*0.05,
				xmax=5+4*0.05,
				xtick={2,3,4,5},
				xticklabels={$\scriptstyle 2$,$\scriptstyle 3$,$\scriptstyle 4$,$\scriptstyle 5$},
				ymin=0.458,
				ymax=0.505,
				ytick={0.46,0.47,0.48,0.49,0.5},
				yticklabels={$\scriptstyle 0.46$,$\scriptstyle 0.47$, $\scriptstyle 0.48$,$\scriptstyle 0.49$, $\scriptstyle 0.5$},
				every axis x label/.style={ at={(ticklabel* cs:0.1)}, anchor=west,},
				every axis y label/.style={ at={(ticklabel* cs:0.1)}, anchor=south,},
				axis line style={->}, x post scale=0.5,
				axis line style={->}, y post scale =30.0]
				\addplot[thick, color=blue]
				coordinates{
					( 2 , 0.5 )
					( 2.01 , 0.4996163 )
					( 2.02 , 0.4992372 )
					( 2.03 , 0.4988627 )
					( 2.04 , 0.4984926 )
					( 2.05 , 0.498127 )
					( 2.06 , 0.4977657 )
					( 2.07 , 0.4974087 )
					( 2.08 , 0.4970559 )
					( 2.09 , 0.4967072 )
					( 2.1 , 0.4963626 )
					( 2.11 , 0.496022 )
					( 2.12 , 0.4956853 )
					( 2.13 , 0.4953525 )
					( 2.14 , 0.4950236 )
					( 2.15 , 0.4946984 )
					( 2.16 , 0.4943769 )
					( 2.17 , 0.494059 )
					( 2.18 , 0.4937448 )
					( 2.19 , 0.493434 )
					( 2.2 , 0.4931268 )
					( 2.21 , 0.492823 )
					( 2.22 , 0.4925225 )
					( 2.23 , 0.4922254 )
					( 2.24 , 0.4919316 )
					( 2.25 , 0.491641 )
					( 2.26 , 0.4913536 )
					( 2.27 , 0.4910693 )
					( 2.28 , 0.4907881 )
					( 2.29 , 0.4905099 )
					( 2.3 , 0.4902347 )
					( 2.31 , 0.4899625 )
					( 2.32 , 0.4896932 )
					( 2.33 , 0.4894267 )
					( 2.34 , 0.4891631 )
					( 2.35 , 0.4889023 )
					( 2.36 , 0.4886442 )
					( 2.37 , 0.4883888 )
					( 2.38 , 0.4881361 )
					( 2.39 , 0.487886 )
					( 2.4 , 0.4876385 )
					( 2.41 , 0.4873936 )
					( 2.42 , 0.4871512 )
					( 2.43 , 0.4869113 )
					( 2.44 , 0.4866738 )
					( 2.45 , 0.4864387 )
					( 2.46 , 0.4862061 )
					( 2.47 , 0.4859757 )
					( 2.48 , 0.4857478 )
					( 2.49 , 0.485522 )
					( 2.5 , 0.4852986 )
					( 2.51 , 0.4850774 )
					( 2.52 , 0.4848583 )
					( 2.53 , 0.4846415 )
					( 2.54 , 0.4844268 )
					( 2.55 , 0.4842142 )
					( 2.56 , 0.4840036 )
					( 2.57 , 0.4837952 )
					( 2.58 , 0.4835887 )
					( 2.59 , 0.4833843 )
					( 2.6 , 0.4831818 )
					( 2.61 , 0.4829813 )
					( 2.62 , 0.4827827 )
					( 2.63 , 0.482586 )
					( 2.64 , 0.4823912 )
					( 2.65 , 0.4821982 )
					( 2.66 , 0.482007 )
					( 2.67 , 0.4818177 )
					( 2.68 , 0.4816301 )
					( 2.69 , 0.4814443 )
					( 2.7 , 0.4812603 )
					( 2.71 , 0.4810779 )
					( 2.72 , 0.4808972 )
					( 2.73 , 0.4807183 )
					( 2.74 , 0.4805409 )
					( 2.75 , 0.4803652 )
					( 2.76 , 0.4801911 )
					( 2.77 , 0.4800186 )
					( 2.78 , 0.4798476 )
					( 2.79 , 0.4796782 )
					( 2.8 , 0.4795104 )
					( 2.81 , 0.479344 )
					( 2.82 , 0.4791792 )
					( 2.83 , 0.4790158 )
					( 2.84 , 0.4788539 )
					( 2.85 , 0.4786934 )
					( 2.86 , 0.4785343 )
					( 2.87 , 0.4783767 )
					( 2.88 , 0.4782204 )
					( 2.89 , 0.4780655 )
					( 2.9 , 0.477912 )
					( 2.91 , 0.4777598 )
					( 2.92 , 0.477609 )
					( 2.93 , 0.4774594 )
					( 2.94 , 0.4773112 )
					( 2.95 , 0.4771642 )
					( 2.96 , 0.4770185 )
					( 2.97 , 0.476874 )
					( 2.98 , 0.4767308 )
					( 2.99 , 0.4765888 )
					( 3.0 , 0.476448 )
					( 3.01 , 0.4763084 )
					( 3.02 , 0.47617 )
					( 3.03 , 0.4760327 )
					( 3.04 , 0.4758966 )
					( 3.05 , 0.4757617 )
					( 3.06 , 0.4756278 )
					( 3.07 , 0.4754951 )
					( 3.08 , 0.4753635 )
					( 3.09 , 0.475233 )
					( 3.1 , 0.4751035 )
					( 3.11 , 0.4749751 )
					( 3.12 , 0.4748478 )
					( 3.13 , 0.4747215 )
					( 3.14 , 0.4745962 )
					( 3.15 , 0.474472 )
					( 3.16 , 0.4743488 )
					( 3.17 , 0.4742265 )
					( 3.18 , 0.4741053 )
					( 3.19 , 0.473985 )
					( 3.2 , 0.4738657 )
					( 3.21 , 0.4737474 )
					( 3.22 , 0.47363 )
					( 3.23 , 0.4735135 )
					( 3.24 , 0.4733979 )
					( 3.25 , 0.4732833 )
					( 3.26 , 0.4731696 )
					( 3.27 , 0.4730567 )
					( 3.28 , 0.4729448 )
					( 3.29 , 0.4728337 )
					( 3.3 , 0.4727235 )
					( 3.31 , 0.4726142 )
					( 3.32 , 0.4725057 )
					( 3.33 , 0.472398 )
					( 3.34 , 0.4722912 )
					( 3.35 , 0.4721852 )
					( 3.36 , 0.47208 )
					( 3.37 , 0.4719756 )
					( 3.38 , 0.471872 )
					( 3.39 , 0.4717692 )
					( 3.4 , 0.4716672 )
					( 3.41 , 0.471566 )
					( 3.42 , 0.4714655 )
					( 3.43 , 0.4713658 )
					( 3.44 , 0.4712668 )
					( 3.45 , 0.4711686 )
					( 3.46 , 0.4710711 )
					( 3.47 , 0.4709744 )
					( 3.48 , 0.4708783 )
					( 3.49 , 0.470783 )
					( 3.5 , 0.4706884 )
					( 3.51 , 0.4705944 )
					( 3.52 , 0.4705012 )
					( 3.53 , 0.4704087 )
					( 3.54 , 0.4703168 )
					( 3.55 , 0.4702256 )
					( 3.56 , 0.4701351 )
					( 3.57 , 0.4700452 )
					( 3.58 , 0.469956 )
					( 3.59 , 0.4698674 )
					( 3.6 , 0.4697794 )
					( 3.61 , 0.4696921 )
					( 3.62 , 0.4696055 )
					( 3.63 , 0.4695194 )
					( 3.64 , 0.469434 )
					( 3.65 , 0.4693491 )
					( 3.66 , 0.4692649 )
					( 3.67 , 0.4691813 )
					( 3.68 , 0.4690983 )
					( 3.69 , 0.4690158 )
					( 3.7 , 0.4689339 )
					( 3.71 , 0.4688526 )
					( 3.72 , 0.4687719 )
					( 3.73 , 0.4686918 )
					( 3.74 , 0.4686122 )
					( 3.75 , 0.4685331 )
					( 3.76 , 0.4684546 )
					( 3.77 , 0.4683767 )
					( 3.78 , 0.4682992 )
					( 3.79 , 0.4682224 )
					( 3.8 , 0.468146 )
					( 3.81 , 0.4680702 )
					( 3.82 , 0.4679949 )
					( 3.83 , 0.4679201 )
					( 3.84 , 0.4678458 )
					( 3.85 , 0.467772 )
					( 3.86 , 0.4676987 )
					( 3.87 , 0.4676259 )
					( 3.88 , 0.4675536 )
					( 3.89 , 0.4674818 )
					( 3.9 , 0.4674105 )
					( 3.91 , 0.4673396 )
					( 3.92 , 0.4672693 )
					( 3.93 , 0.4671994 )
					( 3.94 , 0.4671299 )
					( 3.95 , 0.467061 )
					( 3.96 , 0.4669924 )
					( 3.97 , 0.4669244 )
					( 3.98 , 0.4668567 )
					( 3.99 , 0.4667896 )
					( 4 , 0.4667228 )
					( 4.01 , 0.4668326 )
					( 4.02 , 0.4669417 )
					( 4.03 , 0.4670502 )
					( 4.04 , 0.4671579 )
					( 4.05 , 0.467265 )
					( 4.06 , 0.4673714 )
					( 4.07 , 0.4674771 )
					( 4.08 , 0.467582 )
					( 4.09 , 0.4676862 )
					( 4.1 , 0.4677897 )
					( 4.11 , 0.4678924 )
					( 4.12 , 0.4679943 )
					( 4.13 , 0.4680954 )
					( 4.14 , 0.4681958 )
					( 4.15 , 0.4682953 )
					( 4.16 , 0.468394 )
					( 4.17 , 0.4684919 )
					( 4.18 , 0.4685889 )
					( 4.19 , 0.4686851 )
					( 4.2 , 0.4687805 )
					( 4.21 , 0.4688749 )
					( 4.22 , 0.4689685 )
					( 4.23 , 0.4690612 )
					( 4.24 , 0.469153 )
					( 4.25 , 0.4692439 )
					( 4.26 , 0.4693338 )
					( 4.27 , 0.4694228 )
					( 4.28 , 0.4695109 )
					( 4.29 , 0.4695981 )
					( 4.3 , 0.4696843 )
					( 4.31 , 0.4697695 )
					( 4.32 , 0.4698538 )
					( 4.33 , 0.4699371 )
					( 4.34 , 0.4700194 )
					( 4.35 , 0.4701008 )
					( 4.36 , 0.4701811 )
					( 4.37 , 0.4702604 )
					( 4.38 , 0.4703388 )
					( 4.39 , 0.4704161 )
					( 4.4 , 0.4704924 )
					( 4.41 , 0.4705676 )
					( 4.42 , 0.4706419 )
					( 4.43 , 0.470715 )
					( 4.44 , 0.4707872 )
					( 4.45 , 0.4708583 )
					( 4.46 , 0.4709284 )
					( 4.47 , 0.4709974 )
					( 4.48 , 0.4710653 )
					( 4.49 , 0.4711322 )
					( 4.5 , 0.471198 )
					( 4.51 , 0.4712627 )
					( 4.52 , 0.4713264 )
					( 4.53 , 0.471389 )
					( 4.54 , 0.4714505 )
					( 4.55 , 0.4715109 )
					( 4.56 , 0.4715703 )
					( 4.57 , 0.4716285 )
					( 4.58 , 0.4716857 )
					( 4.59 , 0.4717418 )
					( 4.6 , 0.4717968 )
					( 4.61 , 0.4718507 )
					( 4.62 , 0.4719035 )
					( 4.63 , 0.4719552 )
					( 4.64 , 0.4720059 )
					( 4.65 , 0.4720554 )
					( 4.66 , 0.4721038 )
					( 4.67 , 0.4721512 )
					( 4.68 , 0.4721974 )
					( 4.69 , 0.4722426 )
					( 4.7 , 0.4722866 )
					( 4.71 , 0.4723296 )
					( 4.72 , 0.4723715 )
					( 4.73 , 0.4724123 )
					( 4.74 , 0.472452 )
					( 4.75 , 0.4724906 )
					( 4.76 , 0.4725282 )
					( 4.77 , 0.4725646 )
					( 4.78 , 0.4726 )
					( 4.79 , 0.4726343 )
					( 4.8 , 0.4726676 )
					( 4.81 , 0.4726997 )
					( 4.82 , 0.4727309 )
					( 4.83 , 0.4727609 )
					( 4.84 , 0.4727899 )
					( 4.85 , 0.4728178 )
					( 4.86 , 0.4728447 )
					( 4.87 , 0.4728705 )
					( 4.88 , 0.4728953 )
					( 4.89 , 0.4729191 )
					( 4.9 , 0.4729418 )
					( 4.91 , 0.4729636 )
					( 4.92 , 0.4729842 )
					( 4.93 , 0.4730039 )
					( 4.94 , 0.4730226 )
					( 4.95 , 0.4730402 )
					( 4.96 , 0.4730569 )
					( 4.97 , 0.4730726 )
					( 4.98 , 0.4730872 )
					( 4.99 , 0.473101 )
					( 5.0 , 0.4731137 )
				};
				\addplot[thick, color=red]
				coordinates{
					( 2 , 0.5 )
					( 2.01 , 0.4998856 )
					( 2.02 , 0.4997772 )
					( 2.03 , 0.499665 )
					( 2.04 , 0.4995588 )
					( 2.05 , 0.4994488 )
					( 2.06 , 0.4993398 )
					( 2.07 , 0.4992367 )
					( 2.08 , 0.4991298 )
					( 2.09 , 0.4990239 )
					( 2.1 , 0.498919 )
					( 2.11 , 0.4988199 )
					( 2.12 , 0.498717 )
					( 2.13 , 0.498615 )
					( 2.14 , 0.498514 )
					( 2.15 , 0.498414 )
					( 2.16 , 0.4983194 )
					( 2.17 , 0.4982212 )
					( 2.18 , 0.4981239 )
					( 2.19 , 0.4980274 )
					( 2.2 , 0.4979318 )
					( 2.21 , 0.4978371 )
					( 2.22 , 0.4977432 )
					( 2.23 , 0.4976502 )
					( 2.24 , 0.497558 )
					( 2.25 , 0.4974667 )
					( 2.26 , 0.4973761 )
					( 2.27 , 0.4972907 )
					( 2.28 , 0.4971974 )
					( 2.29 , 0.4971092 )
					( 2.3 , 0.4970217 )
					( 2.31 , 0.4969351 )
					( 2.32 , 0.4968491 )
					( 2.33 , 0.4967639 )
					( 2.34 , 0.4966795 )
					( 2.35 , 0.4965957 )
					( 2.36 , 0.4965127 )
					( 2.37 , 0.4964304 )
					( 2.38 , 0.4963487 )
					( 2.39 , 0.4962636 )
					( 2.4 , 0.4961833 )
					( 2.41 , 0.4961037 )
					( 2.42 , 0.4960248 )
					( 2.43 , 0.4959465 )
					( 2.44 , 0.4958648 )
					( 2.45 , 0.4957878 )
					( 2.46 , 0.4957114 )
					( 2.47 , 0.4956356 )
					( 2.48 , 0.4955565 )
					( 2.49 , 0.4954819 )
					( 2.5 , 0.495408 )
					( 2.51 , 0.4953347 )
					( 2.52 , 0.4952579 )
					( 2.53 , 0.4951858 )
					( 2.54 , 0.4951142 )
					( 2.55 , 0.4950392 )
					( 2.56 , 0.4949688 )
					( 2.57 , 0.4948988 )
					( 2.58 , 0.4948256 )
					( 2.59 , 0.4947568 )
					( 2.6 , 0.4946885 )
					( 2.61 , 0.4946169 )
					( 2.62 , 0.4945496 )
					( 2.63 , 0.4944791 )
					( 2.64 , 0.4944129 )
					( 2.65 , 0.4943472 )
					( 2.66 , 0.4942782 )
					( 2.67 , 0.4942135 )
					( 2.68 , 0.4941455 )
					( 2.69 , 0.4940818 )
					( 2.7 , 0.4940148 )
					( 2.71 , 0.493952 )
					( 2.72 , 0.493886 )
					( 2.73 , 0.4938242 )
					( 2.74 , 0.4937591 )
					( 2.75 , 0.4936982 )
					( 2.76 , 0.4936341 )
					( 2.77 , 0.493574 )
					( 2.78 , 0.4935108 )
					( 2.79 , 0.4934516 )
					( 2.8 , 0.4933893 )
					( 2.81 , 0.493331 )
					( 2.82 , 0.4932695 )
					( 2.83 , 0.493212 )
					( 2.84 , 0.4931514 )
					( 2.85 , 0.4930947 )
					( 2.86 , 0.493035 )
					( 2.87 , 0.4929756 )
					( 2.88 , 0.4929201 )
					( 2.89 , 0.4928616 )
					( 2.9 , 0.4928069 )
					( 2.91 , 0.4927491 )
					( 2.92 , 0.4926952 )
					( 2.93 , 0.4926382 )
					( 2.94 , 0.4925816 )
					( 2.95 , 0.4925288 )
					( 2.96 , 0.492473 )
					( 2.97 , 0.4924175 )
					( 2.98 , 0.4923658 )
					( 2.99 , 0.492311 )
					( 3.0 , 0.4922567 )
					( 3.01 , 0.492206 )
					( 3.02 , 0.4921523 )
					( 3.03 , 0.492099 )
					( 3.04 , 0.4920493 )
					( 3.05 , 0.4919967 )
					( 3.06 , 0.4919444 )
					( 3.07 , 0.4918958 )
					( 3.08 , 0.4918442 )
					( 3.09 , 0.4917929 )
					( 3.1 , 0.4917452 )
					( 3.11 , 0.4916945 )
					( 3.12 , 0.4916442 )
					( 3.13 , 0.4915942 )
					( 3.14 , 0.4915478 )
					( 3.15 , 0.4914984 )
					( 3.16 , 0.4914494 )
					( 3.17 , 0.4914038 )
					( 3.18 , 0.4913553 )
					( 3.19 , 0.4913072 )
					( 3.2 , 0.4912594 )
					( 3.21 , 0.491215 )
					( 3.22 , 0.4911677 )
					( 3.23 , 0.4911207 )
					( 3.24 , 0.4910741 )
					( 3.25 , 0.4910308 )
					( 3.26 , 0.4909847 )
					( 3.27 , 0.4909388 )
					( 3.28 , 0.4908933 )
					( 3.29 , 0.490848 )
					( 3.3 , 0.4908061 )
					( 3.31 , 0.4907613 )
					( 3.32 , 0.4907169 )
					( 3.33 , 0.4906727 )
					( 3.34 , 0.4906287 )
					( 3.35 , 0.4905881 )
					( 3.36 , 0.4905446 )
					( 3.37 , 0.4905015 )
					( 3.38 , 0.4904586 )
					( 3.39 , 0.4904159 )
					( 3.4 , 0.4903735 )
					( 3.41 , 0.4903343 )
					( 3.42 , 0.4902924 )
					( 3.43 , 0.4902507 )
					( 3.44 , 0.4902093 )
					( 3.45 , 0.4901681 )
					( 3.46 , 0.4901272 )
					( 3.47 , 0.4900893 )
					( 3.48 , 0.4900489 )
					( 3.49 , 0.4900086 )
					( 3.5 , 0.4899686 )
					( 3.51 , 0.4899288 )
					( 3.52 , 0.4898892 )
					( 3.53 , 0.4898499 )
					( 3.54 , 0.4898136 )
					( 3.55 , 0.4897746 )
					( 3.56 , 0.489736 )
					( 3.57 , 0.4896975 )
					( 3.58 , 0.4896592 )
					( 3.59 , 0.4896212 )
					( 3.6 , 0.4895833 )
					( 3.61 , 0.4895457 )
					( 3.62 , 0.489511 )
					( 3.63 , 0.4894738 )
					( 3.64 , 0.4894368 )
					( 3.65 , 0.4894 )
					( 3.66 , 0.4893634 )
					( 3.67 , 0.489327 )
					( 3.68 , 0.4892908 )
					( 3.69 , 0.4892547 )
					( 3.7 , 0.4892189 )
					( 3.71 , 0.4891833 )
					( 3.72 , 0.4891478 )
					( 3.73 , 0.4891153 )
					( 3.74 , 0.4890802 )
					( 3.75 , 0.4890453 )
					( 3.76 , 0.4890106 )
					( 3.77 , 0.4889761 )
					( 3.78 , 0.4889418 )
					( 3.79 , 0.4889077 )
					( 3.8 , 0.4888737 )
					( 3.81 , 0.4888399 )
					( 3.82 , 0.4888063 )
					( 3.83 , 0.4887728 )
					( 3.84 , 0.4887396 )
					( 3.85 , 0.4887065 )
					( 3.86 , 0.4886736 )
					( 3.87 , 0.4886408 )
					( 3.88 , 0.4886082 )
					( 3.89 , 0.4885758 )
					( 3.9 , 0.4885436 )
					( 3.91 , 0.4885141 )
					( 3.92 , 0.4884821 )
					( 3.93 , 0.4884504 )
					( 3.94 , 0.4884188 )
					( 3.95 , 0.4883873 )
					( 3.96 , 0.4883561 )
					( 3.97 , 0.4883249 )
					( 3.98 , 0.488294 )
					( 3.99 , 0.4882632 )
					( 4.0 , 0.4882325 )
				};
				\addplot[thick, color=red]
				coordinates{
					( 4 , 0.5 )
					( 4.02 , 0.4998856 )
					( 4.04 , 0.4997772 )
					( 4.06 , 0.499665 )
					( 4.08 , 0.4995588 )
					( 4.1 , 0.4994488 )
					( 4.12 , 0.4993398 )
					( 4.14 , 0.4992367 )
					( 4.16 , 0.4991298 )
					( 4.18 , 0.4990239 )
					( 4.2 , 0.498919 )
					( 4.22 , 0.4988199 )
					( 4.24 , 0.498717 )
					( 4.26 , 0.498615 )
					( 4.28 , 0.498514 )
					( 4.3 , 0.498414 )
					( 4.32 , 0.4983194 )
					( 4.34 , 0.4982212 )
					( 4.36 , 0.4981239 )
					( 4.38 , 0.4980274 )
					( 4.4 , 0.4979318 )
					( 4.42 , 0.4978371 )
					( 4.44 , 0.4977432 )
					( 4.46 , 0.4976502 )
					( 4.48 , 0.497558 )
					( 4.5 , 0.4974667 )
					( 4.52 , 0.4973761 )
					( 4.54 , 0.4972907 )
					( 4.56 , 0.4971974 )
					( 4.58 , 0.4971092 )
					( 4.6 , 0.4970217 )
					( 4.62 , 0.4969351 )
					( 4.64 , 0.4968491 )
					( 4.66 , 0.4967639 )
					( 4.68 , 0.4966795 )
					( 4.7 , 0.4965957 )
					( 4.72 , 0.4965127 )
					( 4.74 , 0.4964304 )
					( 4.76 , 0.4963487 )
					( 4.78 , 0.4962636 )
					( 4.8 , 0.4961833 )
					( 4.82 , 0.4961037 )
					( 4.84 , 0.4960248 )
					( 4.86 , 0.4959465 )
					( 4.88 , 0.4958648 )
					( 4.9 , 0.4957878 )
					( 4.92 , 0.4957114 )
					( 4.94 , 0.4956356 )
					( 4.96 , 0.4955565 )
					( 4.98 , 0.4954819 )
					( 5.0 , 0.495408 )
				};
			\end{axis}
		\end{tikzpicture}
		\caption{Plot of the {\color{blue} upper bound for $\mathscr{C}_p/p$} obtained in Corollary \ref{cor:C_p_2_4} and Corollary \ref{cor:C_p_4_5} and the {\color{red} upper bound for $\mathscr{C}_p/p$} obtained in \cite{Brevig} for $2 \leq p \leq 5$. The discontinuity in $p=4$ for the latter bound is a consequence of the power trick.}
		\label{fig:plot}
	\end{figure}
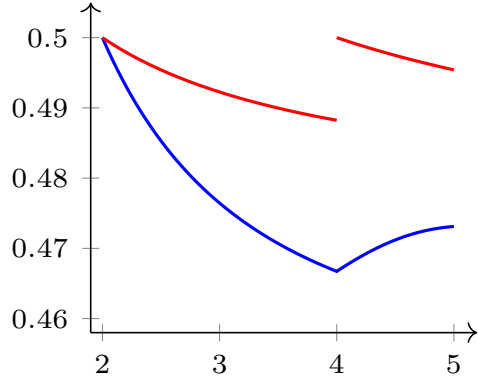

	\begin{theorem}\label{thm:new:E_p4to5}
		Fix $4 < p \leq 5$ and let $1/2\leq \delta_1 \leq 1/2+3/p$ and $1-2/p \leq \delta_2 \leq 1$. Then
		\begin{align*}
			\mathscr{E}_p (\delta_1, \delta_2) = \int_{0}^{2/p}\frac{\sin^2\frac{p}{2}\pi x}{\pi^2 x^2}dx &+ \sum_{n=0}^{\infty} \int_{n+4/p}^{n+1/2+3/p} \frac{\sin^2\frac{p}{2}\pi(x-n)}{\pi^2x^2}\, dx \\ &+ \sum_{n=0}^{\infty}\int_{n+1/2+3/p}^{n+1+2/p} \frac{\sin^2\frac{p}{2}\pi(x-n-1)}{\pi^2x^2} \, dx.
		\end{align*}
		The supremum is attained if and only if $\tau$ is the sequence given by $\tau_n=n-1/2+3/p$ for all $n \geq 1$.
	\end{theorem}
	We refer to Figure \ref{fig:p3lambdas} and Figure \ref{fig:p5lambdas} for illustrations of Theorem \ref{thm:new:Ep2to4} and Theorem \ref{thm:new:E_p4to5}. When proving Theorem \ref{thm:new:Ep2to4} we split the proof into the cases $2 \leq p \leq 3$ and $3<p \leq 4$, respectively. The former case is the easiest one, and here we obtain a more general result (see Theorem \ref{thm:E_p2to4:gensep}). The case $3<p<4$ requires some technical estimates. For $4 < p \leq 5$ the proof, as well as the result, takes a different form. This is partly explained by the fact that $K_p(\tau;x)$ behaves differently for $4<p<6$ than for $2<p<4$. We will see that the case $p=4$ is particularly simple. From Theorem \ref{thm:BCOS} we immediately obtain the following corollaries of Theorem \ref{thm:new:Ep2to4} and Theorem \ref{thm:new:E_p4to5}.
	
	\begin{corollary}\label{cor:C_p_2_4}
		If $2 \leq p \leq 4$, then
		\[\mathscr{C}_p \leq 2\sum_{n=0}^{\infty} \int_{n}^{n+2/p} \frac{\sin^2\frac{p}{2}\pi(x-n)}{\pi^2 x^2}\,dx.\]
	\end{corollary}
	
	\begin{corollary}\label{cor:C_p_4_5}
		If $4 \leq p \leq 5$, then \begin{align*}
			\mathscr{C}_p \leq 2\int_{0}^{2/p}\frac{\sin^2\frac{p}{2}\pi x}{\pi^2 x^2}\,dx &+ 2\sum_{n=0}^{\infty} \int_{n+4/p}^{n+1/2+3/p} \frac{\sin^2\frac{p}{2}\pi(x-n)}{\pi^2x^2}\,dx \\ &+ 2\sum_{n=0}^{\infty}\int_{n+1/2+3/p}^{n+1+2/p} \frac{\sin^2\frac{p}{2}\pi(x-n-1)}{\pi^2x^2}\,dx.
		\end{align*}
	\end{corollary}	
	
	\begin{figure}
		\centering
		\begin{tikzpicture}[scale=1.5]
			\def\p{3}
			\def\s{6}
			\begin{axis}[
				axis equal image,
				axis lines = none,
				trig format plots=rad]
				
				\node at (axis cs: -1,0) {$\scriptstyle 0$};
				\node at (axis cs: -1,-3) {$\scriptstyle 1$};
				\node at (axis cs: -1,-6) {$\scriptstyle 2$};
				\node at (axis cs: -1,-9) {$\scriptstyle 3$};
				\node at (axis cs: -1,-12) {$\scriptstyle 4$};
				
				\addplot[thin, name path=b1] coordinates {(0,0) (4,0)};
				\addplot[domain=0:4, samples=100, color=black!75, thin, name path=t1] ({x},{sin(pi*x*\p/(2*\s))*sin(pi*x*\p/(2*\s))});
				\addplot[red!50] fill between [of=b1 and t1];
				
				\addplot[thin] coordinates {(8,0) (12,0)};
				\addplot[domain=8:12, samples=100, color=black!75, thin] ({x},{sin(pi*x*\p/(2*\s))*sin(pi*x*\p/(2*\s))});
				
				\addplot[thin] coordinates {(16,0) (20,0)};
				\addplot[domain=16:20, samples=100, color=black!75, thin] ({x},{sin(pi*x*\p/(2*\s))*sin(pi*x*\p/(2*\s))});
				
				\addplot[thin] coordinates {(24,0) (28,0)};
				\addplot[domain=24:28, samples=100, color=black!75, thin] ({x},{sin(pi*x*\p/(2*\s))*sin(pi*x*\p/(2*\s))});

				\addplot[thin] coordinates {(0,-3) (2,-3)};
				\addplot[domain=0:2, samples=100, color=black!75, thin] ({x},{sin(pi*(x-\s)*\p/(2*\s))*sin(pi*(x-\s)*\p/(2*\s))-3});
				
				\addplot[thin, name path=b2] coordinates {(6,-3) (10,-3)};
				\addplot[domain=6:10, samples=100, color=black!75, thin, name path=t2] ({x},{sin(pi*(x-\s)*\p/(2*\s))*sin(pi*(x-\s)*\p/(2*\s))-3});
				\addplot[red!50] fill between [of=b2 and t2];
				
				\addplot[thin] coordinates {(14,-3) (18,-3)};
				\addplot[domain=14:18, samples=100, color=black!75, thin] ({x},{sin(pi*(x-\s)*\p/(2*\s))*sin(pi*(x-\s)*\p/(2*\s))-3});
				
				\addplot[thin] coordinates {(22,-3) (26,-3)};
				\addplot[domain=22:26, samples=100, color=black!75, thin] ({x},{sin(pi*(x-\s)*\p/(2*\s))*sin(pi*(x-\s)*\p/(2*\s))-3});

				\addplot[thin] coordinates {(4,-6) (8,-6)};
				\addplot[domain=4:8, samples=100, color=black!75, thin] ({x},{sin(pi*(x-2*\s)*\p/(2*\s))*sin(pi*(x-2*\s)*\p/(2*\s))-6});

				\addplot[thin, name path=b3] coordinates {(12,-6) (16,-6)};
				\addplot[domain=12:16, samples=100, color=black!75, thin, name path=t3] ({x},{sin(pi*(x-2*\s)*\p/(2*\s))*sin(pi*(x-2*\s)*\p/(2*\s))-6});
				\addplot[red!50] fill between [of=b3 and t3];

				\addplot[thin] coordinates {(20,-6) (24,-6)};
				\addplot[domain=20:24, samples=100, color=black!75, thin] ({x},{sin(pi*(x-2*\s)*\p/(2*\s))*sin(pi*(x-2*\s)*\p/(2*\s))-6});
				
				\addplot[thin] coordinates {(2,-9) (6,-9)};
				\addplot[domain=2:6, samples=100, color=black!75, thin] ({x},{sin(pi*(x-3*\s)*\p/(2*\s))*sin(pi*(x-3*\s)*\p/(2*\s))-9});
				
				\addplot[thin] coordinates {(10,-9) (14,-9)};
				\addplot[domain=10:14, samples=100, color=black!75, thin] ({x},{sin(pi*(x-3*\s)*\p/(2*\s))*sin(pi*(x-3*\s)*\p/(2*\s))-9});
				
				\addplot[thin, name path=b4] coordinates {(18,-9) (22,-9)};
				\addplot[domain=18:22, samples=100, color=black!75, thin, name path=t4] ({x},{sin(pi*(x-3*\s)*\p/(2*\s))*sin(pi*(x-3*\s)*\p/(2*\s))-9});
				\addplot[red!50] fill between [of=b4 and t4];
				
				\addplot[thin] coordinates {(26,-9) (28,-9)};
				\addplot[domain=26:28, samples=100, color=black!75, thin] ({x},{sin(pi*(x-3*\s)*\p/(2*\s))*sin(pi*(x-3*\s)*\p/(2*\s))-9});
				
				\addplot[thin] coordinates {(0,-12) (4,-12)};
				\addplot[domain=0:4, samples=100, color=black!75, thin] ({x},{sin(pi*(x-4*\s)*\p/(2*\s))*sin(pi*(x-4*\s)*\p/(2*\s))-12});
				\addplot[thin] coordinates {(8,-12) (12,-12)};
				\addplot[domain=8:12, samples=100, color=black!75, thin] ({x},{sin(pi*(x-4*\s)*\p/(2*\s))*sin(pi*(x-4*\s)*\p/(2*\s))-12});
				\addplot[thin] coordinates {(16,-12) (20,-12)};
				\addplot[domain=16:20, samples=100, color=black!75, thin] ({x},{sin(pi*(x-4*\s)*\p/(2*\s))*sin(pi*(x-4*\s)*\p/(2*\s))-12});
				\addplot[thin] coordinates {(24,-12) (28,-12)};
				\addplot[domain=24:28, samples=100, color=black!75, thin] ({x},{sin(pi*(x-4*\s)*\p/(2*\s))*sin(pi*(x-4*\s)*\p/(2*\s))-12});
				
				\addplot[thin, name path=b5] coordinates {(24,-12) (28,-12)};
				\addplot[domain=24:28, samples=100, color=black!75, thin, name path=t5] ({x},{sin(pi*(x-4*\s)*\p/(2*\s))*sin(pi*(x-4*\s)*\p/(2*\s))-12});
				\addplot[red!50] fill between [of=b5 and t5];
				
				\addplot[thin,color=blue,->] coordinates {(4,0+0.0) (4,-3-0.0)};
				\addplot[thin,color=blue,->] coordinates {(10,-3+0.0) (10,-6-0.0)};
				\addplot[thin,color=blue,->] coordinates {(16,-6+0.0) (16,-9-0.0)};
				\addplot[thin,color=blue,->] coordinates {(22,-9+0.0) (22,-12-0.0)};
			\end{axis}
		\end{tikzpicture}
		\caption{An illustration of Theorem \ref{thm:new:Ep2to4} for $p=3$, with the sequence $\CB{\lambda}$ given by $\lambda_n=n-1/3$. The \CR{shaded} area represents $\mathscr{E}_3(2/\pi, 2/3)=E_3(\lambda)$ without considering the integrand factor $1/(\pi x)^2$.}
		\label{fig:p3lambdas}
		
		\centering
		\begin{tikzpicture}[scale=1.5]
			\def\p{5}
			\def\s{10}
			\begin{axis}[
				axis equal image,
				axis lines = none,
				trig format plots=rad]
				
				\node at (axis cs: -1,0) {$\scriptstyle 0$};
				\node at (axis cs: -1,-3) {$\scriptstyle 1$};
				\node at (axis cs: -1,-6) {$\scriptstyle 2$};
				\node at (axis cs: -1,-9) {$\scriptstyle 3$};
				\node at (axis cs: -1,-12) {$\scriptstyle 4$};
				
				\addplot[thin, name path=b1] coordinates {(0,0) (4,0)};
				\addplot[domain=0:4, samples=100, color=black!75, thin, name path=t1] ({x},{sin(pi*x*\p/(2*\s))*sin(pi*x*\p/(2*\s))});
				\addplot[red!50] fill between [of=b1 and t1];
				
				\addplot[thin, name path=b11test] coordinates {(8,0) (11,0)};
				\addplot[domain=8:11, samples=100, color=black!75, thin, name path=t11test] ({x},{sin(pi*x*\p/(2*\s))*sin(pi*x*\p/(2*\s))});
				
				\addplot[thin, name path=b11] coordinates {(11,0) (12,0)};
				\addplot[domain=11:12, samples=100, color=black!75, thin, name path=t11] ({x},{sin(pi*x*\p/(2*\s))*sin(pi*x*\p/(2*\s))});
				\addplot[red!50] fill between [of=b11test and t11test];
				
				\addplot[thin] coordinates {(16,0) (20,0)};
				\addplot[domain=16:20, samples=100, color=black!75, thin] ({x},{sin(pi*x*\p/(2*\s))*sin(pi*x*\p/(2*\s))});
				
				\addplot[thin] coordinates {(24,0) (28,0)};
				\addplot[domain=24:28, samples=100, color=black!75, thin] ({x},{sin(pi*x*\p/(2*\s))*sin(pi*x*\p/(2*\s))});

				\addplot[thin] coordinates {(2,-3) (6,-3)};
				\addplot[domain=2:6, samples=100, color=black!75, thin] ({x},{sin(pi*(x-\s)*\p/(2*\s))*sin(pi*(x-\s)*\p/(2*\s))-3});
				
				\addplot[thin, name path=b2] coordinates {(11,-3) (14,-3)};
				\addplot[domain=11:14, samples=100, color=black!75, thin, name path=t2] ({x},{sin(pi*(x-\s)*\p/(2*\s))*sin(pi*(x-\s)*\p/(2*\s))-3});
				\addplot[thin] coordinates {(10,-3) (11,-3)};
				\addplot[domain=10:11, samples=100, color=black!75, thin] ({x},{sin(pi*(x-\s)*\p/(2*\s))*sin(pi*(x-\s)*\p/(2*\s))-3});
				\addplot[red!50] fill between [of=b2 and t2];
				
				\addplot[thin, name path=b22] coordinates {(18,-3) (21,-3)};
				\addplot[domain=18:21, samples=100, color=black!75, thin, name path=t22] ({x},{sin(pi*(x-\s)*\p/(2*\s))*sin(pi*(x-\s)*\p/(2*\s))-3});
				\addplot[thin] coordinates {(21,-3) (22,-3)};
				\addplot[domain=21:22, samples=100, color=black!75, thin] ({x},{sin(pi*(x-\s)*\p/(2*\s))*sin(pi*(x-\s)*\p/(2*\s))-3});
				\addplot[red!50] fill between [of=b22 and t22];
				
				\addplot[thin] coordinates {(26,-3) (30,-3)};
				\addplot[domain=26:30, samples=100, color=black!75, thin] ({x},{sin(pi*(x-\s)*\p/(2*\s))*sin(pi*(x-\s)*\p/(2*\s))-3});

				\addplot[thin] coordinates {(4,-6) (8,-6)};
				\addplot[domain=4:8, samples=100, color=black!75, thin] ({x},{sin(pi*(x-2*\s)*\p/(2*\s))*sin(pi*(x-2*\s)*\p/(2*\s))-6});
				
				\addplot[thin] coordinates {(12,-6) (16,-6)};
				\addplot[domain=12:16, samples=100, color=black!75, thin] ({x},{sin(pi*(x-2*\s)*\p/(2*\s))*sin(pi*(x-2*\s)*\p/(2*\s))-6});
				
				\addplot[thin, name path=b3] coordinates {(21,-6) (24,-6)};
				\addplot[domain=21:24, samples=100, color=black!75, thin, name path=t3] ({x},{sin(pi*(x-2*\s)*\p/(2*\s))*sin(pi*(x-2*\s)*\p/(2*\s))-6});
				\addplot[thin] coordinates {(20,-6) (21,-6)};
				\addplot[domain=20:21, samples=100, color=black!75, thin] ({x},{sin(pi*(x-2*\s)*\p/(2*\s))*sin(pi*(x-2*\s)*\p/(2*\s))-6});
				\addplot[red!50] fill between [of=b3 and t3];
				
				\addplot[thin, name path=b33] coordinates {(28,-6) (31,-6)};
				\addplot[domain=28:31, samples=100, color=black!75, thin, name path=t33] ({x},{sin(pi*(x-2*\s)*\p/(2*\s))*sin(pi*(x-2*\s)*\p/(2*\s))-6});
				
				\addplot[thin] coordinates {(31,-6) (32,-6)};
				\addplot[domain=31:32, samples=100, color=black!75, thin] ({x},{sin(pi*(x-2*\s)*\p/(2*\s))*sin(pi*(x-2*\s)*\p/(2*\s))-6});
				\addplot[red!50] fill between [of=b33 and t33];

				\addplot[thin] coordinates {(0,-9) (2,-9)};
				\addplot[domain=0:2, samples=100, color=black!75, thin] ({x},{sin(pi*(x-3*\s)*\p/(2*\s))*sin(pi*(x-3*\s)*\p/(2*\s))-9});
				
				\addplot[thin] coordinates {(6,-9) (10,-9)};
				\addplot[domain=6:10, samples=100, color=black!75, thin] ({x},{sin(pi*(x-3*\s)*\p/(2*\s))*sin(pi*(x-3*\s)*\p/(2*\s))-9});
				
				\addplot[thin] coordinates {(14,-9) (18,-9)};
				\addplot[domain=14:18, samples=100, color=black!75, thin] ({x},{sin(pi*(x-3*\s)*\p/(2*\s))*sin(pi*(x-3*\s)*\p/(2*\s))-9});
				
				\addplot[thin] coordinates {(22,-9) (26,-9)};
				\addplot[domain=22:26, samples=100, color=black!75, thin] ({x},{sin(pi*(x-3*\s)*\p/(2*\s))*sin(pi*(x-3*\s)*\p/(2*\s))-9});
				
				\addplot[thin, name path=b4] coordinates {(31,-9) (34,-9)};
				\addplot[domain=31:34, samples=100, color=black!75, thin, name path=t4] ({x},{sin(pi*(x-3*\s)*\p/(2*\s))*sin(pi*(x-3*\s)*\p/(2*\s))-9});
				
				
				\addplot[thin] coordinates {(30,-9) (31,-9)};
				\addplot[domain=30:31, samples=100, color=black!75, thin] ({x},{sin(pi*(x-3*\s)*\p/(2*\s))*sin(pi*(x-3*\s)*\p/(2*\s))-9});
				\addplot[red!50] fill between [of=b4 and t4];

				\addplot[thin] coordinates {(0,-12) (4,-12)};
				\addplot[domain=0:4, samples=100, color=black!75, thin] ({x},{sin(pi*(x-4*\s)*\p/(2*\s))*sin(pi*(x-4*\s)*\p/(2*\s))-12});
				\addplot[thin] coordinates {(8,-12) (12,-12)};
				\addplot[domain=8:12, samples=100, color=black!75, thin] ({x},{sin(pi*(x-4*\s)*\p/(2*\s))*sin(pi*(x-4*\s)*\p/(2*\s))-12});
				\addplot[thin] coordinates {(16,-12) (20,-12)};
				\addplot[domain=16:20, samples=100, color=black!75, thin] ({x},{sin(pi*(x-4*\s)*\p/(2*\s))*sin(pi*(x-4*\s)*\p/(2*\s))-12});
				\addplot[thin] coordinates {(24,-12) (28,-12)};
				\addplot[domain=24:28, samples=100, color=black!75, thin] ({x},{sin(pi*(x-4*\s)*\p/(2*\s))*sin(pi*(x-4*\s)*\p/(2*\s))-12});
				
				\addplot[thin] coordinates {(24,-12) (28,-12)};
				\addplot[domain=24:28, samples=100, color=black!75] ({x},{sin(pi*(x-4*\s)*\p/(2*\s))*sin(pi*(x-4*\s)*\p/(2*\s))-12});
				
				\addplot[thin] coordinates {(32,-12) (36,-12)};
				\addplot[domain=32:36, samples=100, color=black!75] ({x},{sin(pi*(x-4*\s)*\p/(2*\s))*sin(pi*(x-4*\s)*\p/(2*\s))-12});
				
				\addplot[thin,color=blue,->] coordinates {(11,0+0.0) (11,-3-0.0)};
				\addplot[thin,color=blue,->] coordinates {(21,-3+0.0) (21,-6-0.0)};
				\addplot[thin,color=blue,->] coordinates {(31,-6+0.0) (31,-9-0.0)};
			\end{axis}
		\end{tikzpicture}
		\caption{An illustration of Theorem \ref{thm:new:E_p4to5} for $p=5$, with the sequence $\CB{\lambda}$ given by $\lambda_n=n+1/10$. The \CR{shaded} area represents $\mathscr{E}_5(1/2, 3/5)=E_5(\lambda)$ without considering the integrand factor $1/(\pi x)^2$.}
		\label{fig:p5lambdas}
	\end{figure}
	
	What is shown in \cite[Theorem 1.1]{Brevig} is that
	\[\mathscr{C}_p  \leq 2 \left(\int_{0}^{2/p} \frac{\sin^2 \frac{p}{2} \pi x}{\pi^2 x^2} \, dx + \int_{1}^{\infty} \frac{\sin^2\frac{p}{2}\pi(x-1)}{\pi^2 x^2} \, dx \right)\]
	for $2 \leq p \leq 4$. Using the power trick, it is then shown that $\mathscr{C}_p < p/2$ for all $p>2$. This approach leads to a discontinuity in the upper bound at $p=4$. To the contrary, we obtain continuity at $p=4$ as we are determining the upper bound $\mathscr{C}_p \leq \mathscr{E}_p(1,1)$ for all $2 \leq p \leq 5$. A plot of the bound for $\mathscr{C}_p/p$ obtained from Corollary \ref{cor:C_p_2_4} and Corollary \ref{cor:C_p_4_5} compared with the bounds obtained in \cite{Brevig} can be seen in Figure 1. It is conjectured in \cite{Brevig} that $\mathscr{C}_p/p$ is decreasing for $p>2$ and the new improved bound for $2<p<4$ can be seen as further evidence in favour of this conjecture.
	
	Note that our strategy when studying $\mathscr{E}_p$ for $2 \leq p \leq 5$ is to consider all possible sequences of zeroes of $\varphi_p$ and then \emph{determining} $\mathscr{E}_p(\delta_1, \delta_2)$ by identifying those sequences which attain the supremum in \eqref{eq:optprob}. This approach is different from that in \cite[Section 6]{Brevig} for $2<p\leq 4$, where only a local analysis of the zeroes is performed. As a consequence, the strategy in \cite{Brevig} leads to an upper bound for $\mathscr{E}_p(2/\pi, 2/3)$ rather than determining it. We remark that without much work it is possible to determine the supremum $\mathscr{E}_p(2/\pi, 2/3)$ for $2<p<3$ from the proof of \cite[Theorem 1.1]{Brevig}. However, the proof we give in this paper is significantly less technical in this range.
	
	This paper is organized as follows: In Section \ref{sec:prelim} we show some preliminary results that are needed to prove the desired bounds for $\mathscr{C}_p$ in the range $2 \leq p \leq 5$. In Section \ref{sec:2to3} we prove Theorem \ref{thm:new:Ep2to4} in the range $2 \leq p \leq 3$. The proof in the range $3<p \leq 4$ is provided in Section \ref{sec:3to4}. Certain technical bounds needed for $3<p \leq 4$ are postponed to the appendix. Finally, Section \ref{sec:4to5} contains the proof of Theorem \ref{thm:new:E_p4to5}. 
	
	\section{Preliminaries}\label{sec:prelim}
	For each integer $n \geq 0$ we will as in \cite{Brevig} refer to the connected components of 
	\[\left\{x \in \bb{R} : \sin{\frac{p}{2} \pi (x-n)} >0 \right\}\]
	as \emph{intervals at level n}. Note that these intervals all have length $2/p$ and are of the form $(\xi, \xi+2/p)$ where $\xi=n-(4/p)j$ for some integer $j$. The value $E_p(\tau)$ is obtained by integrating the function \[x \mapsto \frac{\sin^2\frac{p}{2}\pi (x-n)}{\pi^2 x^2}\] over the intervals at level $n$ intersected with the interval $(\tau_n, \tau_{n+1})
	$ for each $n \geq 1$. See Figure \ref{fig:p3lambdas} and Figure \ref{fig:p5lambdas} for an illustration of this in connection to the statements of Theorem \ref{thm:new:Ep2to4} and Theorem \ref{thm:new:E_p4to5}. Here, and throughout this paper we will use the convention that $\tau_0=0$. When studying $E_p(\tau)$ we will compare the size of the contribution to $E_p(\tau)$ from different intervals, and we will frequently need the following simple observation from \cite{Brevig}.
	
	\begin{lemma}\label{lem:Brevig:lem6.5}
		Let $2 < p < 6$ and let $(\xi, \xi+2/p)$ be an interval at level $n$.  Then $\xi= n+(4/p)k$ for some integer $k$. Furthermore, the interval $(\xi, \xi+2/p)$ intersects exactly one interval at level $n+1$ and this interval is $(\xi+1-4/p, \xi+1-2/p)$. When $p=2$ and $p=6$ the endpoint of an interval at level $n$ is the endpoint of an interval at level $n+1$, and the interval $(\xi, \xi+2/p)$ at level $n$ does not intersect any interval at level $n+1$.
	\end{lemma}
	We note that $\xi-1+4/p > \xi$ when $2 < p < 4$, meaning that the interval at level $n+1$ is to the left of the interval at level $n$, whereas when $4<p<6$ the intersecting interval at level $n+1$ is to the right of the interval at level $n$. This can be seen in Figure \ref{fig:p3lambdas} and Figure \ref{fig:p5lambdas}, and is the main reason why the cases $2 \leq p \leq 4$ and $4 \leq p \leq 6$ are different. The case $p=4$ is special because then the intervals at level $n$ and level $n+1$ fully intersect. We will frequently apply Lemma \ref{lem:Brevig:lem6.5} in combination with the fact that $x \mapsto 1/x$ is a decreasing function when comparing the contribution to $E_p(\tau)$ from different intervals. 
	
	Note that if $(\xi, \xi+2/p)$ is an interval at level $n$ then this interval intersects exactly one interval at level $n-1$ and this interval is $(\xi-1+4/p, \xi-1+6/p)$. We will also refer to Lemma \ref{lem:Brevig:lem6.5} when we are considering intervals at level $n-1$ and $n$. 
	
	Let $(\xi, \xi + 2/p)$ be an interval at level $n$ and let $2 \leq p \leq 4$. The for $\tau_{n+1}$ in the interval $[\xi+1-4/p, \xi+1]$ we define
	\begin{equation}\label{eq:defS}
		\begin{split}
			S(\tau_{n+1})&=\int_{\xi}^{\xi+2/p} \chi_{(\xi, \tau_{n+1})}(x) \frac{\sin^2 \frac{p}{2} (x-n)}{\pi^2 x^2 } \, dx  \\  & + \int_{\xi+1-4/p}^{\xi+1-2/p} \chi_{(\tau_{n+1}, \xi+1-2/p)}(x) \frac{\sin^2\frac{p}{2} \pi (x-n-1)}{\pi^2 x^2} \, dx.
		\end{split}
	\end{equation}
	
	See Figure \ref{fig:Sdef2to4} for an illustration of this definition. If $b \leq a$ we consider $\chi_{(a,b)}$ to be the zero function. Then using Lemma \ref{lem:Brevig:lem6.5} and the fact that $x \mapsto 1/x$ is a decreasing function we observe the following.
	
	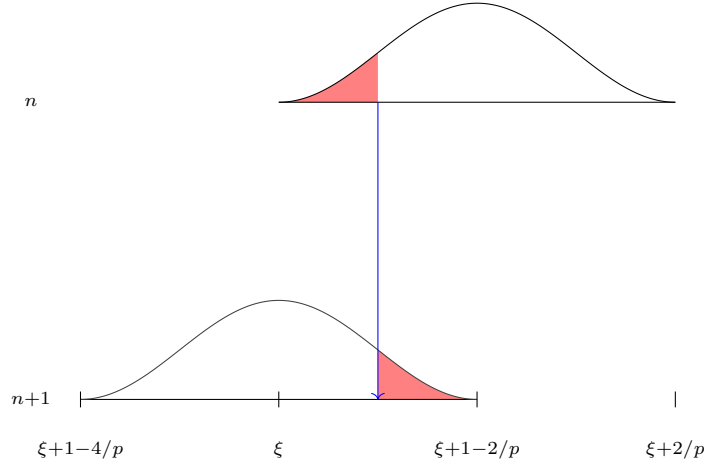
\begin{figure}
		\centering
		\begin{tikzpicture}[scale=1.5]
			\def\p{3}
			\def\s{6}
			\begin{axis}[
				axis equal image,
				axis lines = none,
				trig format plots=rad]
				
				\node at (axis cs: 5.5,0) {$\scriptstyle n$};
				\node at (axis cs: 5.5,-3) {$\scriptstyle n+1$};
				
				\addplot[thin, name path=t0] coordinates {(8,0) (9,0)};
				\addplot[domain=8:9, samples=100, thin, name path=b0] ({x},{sin(pi*x*\p/(2*\s))*(sin(pi*(x)*\p/(2*\s)))});
				\addplot[red!50] fill between [of=b0 and t0];
				
				\addplot[thin] coordinates {(9,0) (12,0)};
				\addplot[domain=9:12, samples=100, thin] ({x},{sin(pi*x*\p/(2*\s))*(sin(pi*(x)*\p/(2*\s)))});
				
				\addplot[thin] coordinates {(6,-3) (9,-3)};
				\addplot[domain=6:9, samples=100, color=black!75, thin] ({x},{(sin(pi*(x-\s)*\p/(2*\s)))*(sin(pi*(x-\s)*\p/(2*\s)))-3});
				
				\addplot[thin, name path=b2] coordinates {(9,-3) (10,-3)};
				\addplot[domain=9:10, samples=100, color=black!75, thin, name path=t2] ({x},{(sin(pi*(x-\s)*\p/(2*\s)))*(sin(pi*(x-\s)*\p/(2*\s)))-3});
				\addplot[red!50] fill between [of=b2 and t2];
				
				\addplot[thin, color=black!0] coordinates {(5.5,-3.6) (9,-3.6)};
				\addplot[thin,color=blue,->] coordinates {(9,0) (9,-3)};
				\addplot[only marks,mark=|,color=black,mark size=2pt] coordinates {(6,-3) (8,-3) (10,-3) (12,-3)};
				\node at (axis cs: 6,-3.5) {$\scriptstyle \xi+1-4/p$};
				\node at (axis cs: 8,-3.5) {$\scriptstyle \xi$};
				\node at (axis cs: 10,-3.5) {$\scriptstyle \xi+1-2/p$};
				\node at (axis cs: 12,-3.5) {$\scriptstyle \xi+2/p$};
			\end{axis}
		\end{tikzpicture}
		\caption{An illustration of $\CR{S(\tau_{n+1})}$ for $\CB{\tau_{n+1}}=\xi+1/2-1/p$, without considering the integrand factor $1/(\pi x)^2$.}
		\label{fig:Sdef2to4}
	\end{figure}
	
	\begin{lemma}\label{lem:maximizeS}
		Let $2 \leq p \leq 4$ and let $(\xi, \xi+2/p)$ be an interval at level $n$. Let $\xi +1-4/p \leq \tau_{n+1} \leq \xi+1$. Then $S(\tau_{n+1})$ is maximized for $\tau_{n+1}=\xi+1-4/p$.
	\end{lemma}
	Let $(\xi, \xi+2/p)$ be an interval at level $n$ and let $(\xi+1-4/p, \xi+1-2/p)$ be the intersecting interval at level $n+1$. Then for $2<p<6$ we will denote the midpoint of the intersection as
	\begin{equation}\label{eq:defmxi}
		m_{\xi}=\xi+1/2-1/p.
	\end{equation}
	We also extend this definition to $p=2$ and $p=6$ and will also refer to $m_\xi$ as the midpoint here, even though the intervals do not intersect. We will see that the midpoint will be important in our analysis of $E_p(\tau)$ for different sequences $\tau$. When $2 \leq p \leq 4$ this is mainly due to the following simple observation.
	
	\begin{lemma}\label{lem:midpointlem}
		Assume $2 \leq p \leq 4$. Let $(\xi, \xi+2/p)$ be an interval at level $n$. Then $S(\tau_{n+1})$ is decreasing for $\tau_{n+1}$ in $(\xi+1-4/p, m_{\xi})$ and increasing for $\tau_{n+1}$ in the interval $(m_\xi, \xi + 2/p)$. For $\tau_{n+1}$ in $(\xi+2/p, \xi+1)$ the value $S(\tau_{n+1})$ is constant.
	\end{lemma}
	
	\begin{proof}
		This follows from the pointwise estimates 
		\[\sin^2\frac{p}{2} \pi (x-n) \leq \sin^2 \frac{p}{2} \pi (x-n-1), \] for $\xi \leq x \leq m_\xi$ and 
		\[\sin^2\frac{p}{2} \pi (x-n) \geq \sin^2\frac{p}{2} \pi (x-n-1), \] for $m_\xi \leq x \leq \xi+1-2/p$.
	\end{proof}
	\section{Proof of Theorem \ref{thm:new:Ep2to4} for \texorpdfstring{$2 \leq p \leq 3$}{2≤p≤3}.}\label{sec:2to3}
	In this section we prove the following result.
	\begin{theorem}\label{thm:E_p2to4:gensep}
		Fix $2\leq p \leq 4$ and let $1-2/p \leq \delta_1 \leq 1$ and $2-4/p \leq \delta_2 \leq 1$. Then \[\mathscr{E}_p(\delta_1, \delta_2)=\sum_{n=0}^{\infty} \int_{n}^{n+2/p} \frac{\sin^2\frac{p}{2}\pi(x-n)}{\pi^2 x^2}\dx.\]
		For $2 \leq p <4$ the supremum is attained if and only if $\tau$ in $T(\delta_1, \delta_2)$ is a sequence such that $\tau_n$ is in the interval $[n-1+2/p,n]$ for all $n \geq 1$. If $p=4$ the supremum is attained for all sequences $\tau$ in $T(\delta_1, \delta_2)$. 
	\end{theorem}
	
	We observe that $2-4/p \leq 2/3$ for $2 \leq p \leq 3$. Thus in the range $2 \leq p \leq 3$ Theorem \ref{thm:new:Ep2to4} follows immediately from Theorem \ref{thm:E_p2to4:gensep}. For the separation conditions  $1-2/p \leq \delta_1 \leq 1$ and $2-4/p \leq \delta_2 \leq 1$ let $\lambda$ denote the sequence given by $\lambda_0=0$ and $\lambda_n=n-1+2/p$ for all $n \geq 1$. We observe that $\lambda$ is a sequence in $T(\delta_1, \delta_2)$ and that 
	\[E_p(\lambda)=\sum_{n=0}^{\infty} \int_{n}^{n+\frac{2}{p}} \frac{\sin^2 \frac{p}{2}\pi(x-n)}{\pi^2 x^2} dx.\] Thus to prove Theorem \ref{thm:E_p2to4:gensep} it will be enough to show that $E_p(\tau) \leq E_p(\lambda)$ for all sequences $\tau $ in the set of sequences $T(\delta_1, \delta_2)$. Note that $E_p(\tau)=E_p(\lambda)$ if $n-2/p \leq \tau_n \leq n$ for each $n \geq 1$.
	The following lemma states that we do not need to consider sequences with large elements. 
	
	\begin{lemma}\label{lem:tau_n_leq_n}
		Let $2 \leq p \leq 4$. Assume $0 < \delta_1 \leq 1$ and $0 < \delta_2 \leq 1$.
		Let $\tau$ be a sequence in $T(\delta_1, \delta_2)$ and let $\gamma$ be the sequence given by $\gamma_n = \min(\tau_n, n)$. Then $\gamma$ is a sequence in $T(\delta_1, \delta_2)$ and $E_p(\gamma) \geq E_p(\tau)$. 
	\end{lemma}
	
	\begin{proof}
		We see that $\gamma_1 = \min(\tau_1, 1) \geq \delta_1$. Combining the fact that $\tau_{n+1} - \tau_n \geq \delta_2$ and $n+1-n=1\geq \delta_2$ it follows that $\gamma_{n+1}- \gamma_n \geq \delta_2$, and hence $\gamma$ is in the set of sequences $T(\delta_1, \delta_2)$. If $\tau_n \leq n$ for all $n \geq 0$ we are done. If not let $m$ be the smallest integer such that $\tau_m >m$ and let $l\geq 0$ be the largest integer such that $\tau_m \geq m+(4/p)l$. By Lemma \ref{lem:maximizeS} replacing $\tau_m$ with $m+(4/p)l$ will increase the value of $E_p(\tau)$, and since since $m+(4/p)l \geq m$ the new sequence is still in $T(\delta_1, \delta_2)$. By periodicity, Lemma \ref{lem:Brevig:lem6.5} and using that $x \mapsto 1/x$ is a decreasing function replacing $\tau_m$ with $m$ will increase the value of $E_p(\tau)$ further, and this sequence is again in $T(\delta_1, \delta_2)$. Repeating this argument for increasing values of $m$ such that $\tau_m>m$ concludes the proof. 
	\end{proof}
	Note that if $\tau_1 \leq 2/p$ we could easily replace $\gamma_n=\min(\tau_n, n)$ with $\gamma_n=\min(\tau_n, n-1+2/p)$ in Lemma \ref{lem:tau_n_leq_n}.
	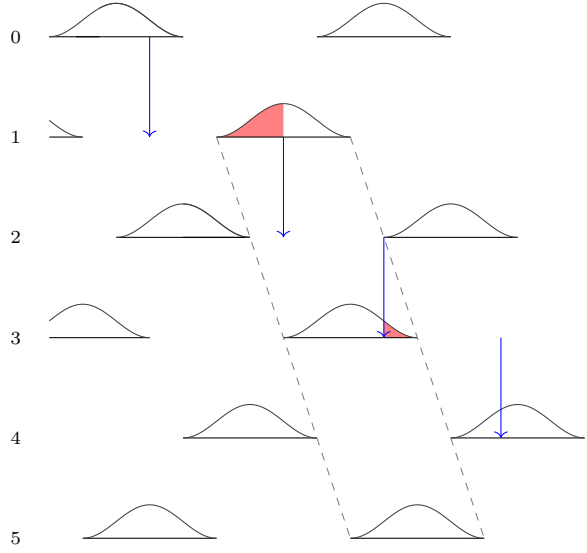
\begin{figure}
		\centering
		\begin{tikzpicture}[scale=1.5]
			\def\p{2.5}
			\def\s{5}
			\begin{axis}[
				axis equal image,
				axis lines = none,
				trig format plots=rad]
				
				\node at (axis cs: -1,0) {$\scriptstyle 0$};
				\node at (axis cs: -1,-3) {$\scriptstyle 1$};
				\node at (axis cs: -1,-6) {$\scriptstyle 2$};
				\node at (axis cs: -1,-9) {$\scriptstyle 3$};
				\node at (axis cs: -1,-12) {$\scriptstyle 4$};
				\node at (axis cs: -1,-15) {$\scriptstyle 5$};
				\node at (axis cs: -1,-18) {$\scriptstyle 6$};
				\node at (axis cs: -1,-21) {$\scriptstyle 7$};
				\node at (axis cs: -1,-24) {$\scriptstyle 8$};
				
				\addplot[thin, name path=b1] coordinates {(0,0) (4,0)};
				\addplot[domain=0:4, samples=100, color=black!75, thin, name path=t1] ({x},{sin(pi*x*\p/(2*\s))*sin(pi*x*\p/(2*\s))});
				
				\addplot[thin] coordinates {(1.5,0) (2/\p,0)};
				\addplot[domain=1.5:4, samples=100, color=black!75, thin] ({x},{sin(pi*x*\p/(2*\s))*sin(pi*x*\p/(2*\s))});
				
				\addplot[thin] coordinates {(8,0) (12,0)};
				\addplot[domain=8:12, samples=100, color=black!75, thin] ({x},{sin(pi*x*\p/(2*\s))*sin(pi*x*\p/(2*\s))});
				\addplot[thin] coordinates {(0,-3) (1,-3)};
				\addplot[domain=0:1, samples=100, color=black!75, thin] ({x},{sin(pi*(x-\s)*\p/(2*\s))*sin(pi*(x-\s)*\p/(2*\s))-3});
				
				\addplot[thin, name path=b1] coordinates {(5,-3) (7,-3)};
				\addplot[domain=5:7, samples=100, color=black!75, thin, name path=t1] ({x},{sin(pi*(x-\s)*\p/(2*\s))*sin(pi*(x-\s)*\p/(2*\s))-3});
				\addplot[red!50] fill between [of=b1 and t1];
				
				\addplot[thin] coordinates {(7,-3) (9,-3)};
				\addplot[domain=7:9, samples=100, color=black!75, thin] ({x},{sin(pi*(x-\s)*\p/(2*\s))*sin(pi*(x-\s)*\p/(2*\s))-3});
				\addplot[thin] coordinates {(2,-6) (6,-6)};
				\addplot[domain=2:6, samples=100, color=black!75, thin] ({x},{sin(pi*(x-2*\s)*\p/(2*\s))*sin(pi*(x-2*\s)*\p/(2*\s))-6});
				
				\addplot[thin, name path=b2] coordinates {(4,-6) (6,-6)};
				\addplot[domain=4:6, samples=100, color=black!75, thin, name path=t2] ({x},{sin(pi*(x-2*\s)*\p/(2*\s))*sin(pi*(x-2*\s)*\p/(2*\s))-6});
				
				\addplot[thin] coordinates {(10,-6) (14,-6)};
				\addplot[domain=10:14, samples=100, color=black!75, thin] ({x},{sin(pi*(x-2*\s)*\p/(2*\s))*sin(pi*(x-2*\s)*\p/(2*\s))-6});
				
				\addplot[thin] coordinates {(0,-9) (3,-9)};
				\addplot[domain=0:3, samples=100, color=black!75, thin] ({x},{sin(pi*(x-3*\s)*\p/(2*\s))*sin(pi*(x-3*\s)*\p/(2*\s))-9});
				
				\addplot[thin] coordinates {(7,-9) (10,-9)};
				\addplot[domain=7:10, samples=100, color=black!75, thin] ({x},{sin(pi*(x-3*\s)*\p/(2*\s))*sin(pi*(x-3*\s)*\p/(2*\s))-9});
				
				\addplot[thin, name path=b3] coordinates {(10,-9) (11,-9)};
				\addplot[domain=10:11, samples=100, color=black!75, thin, name path=t3] ({x},{sin(pi*(x-3*\s)*\p/(2*\s))*sin(pi*(x-3*\s)*\p/(2*\s))-9});
				\addplot[red!50] fill between [of=b3 and t3];
				
				\addplot[thin] coordinates {(4,-12) (8,-12)};
				\addplot[domain=4:8, samples=100, color=black!75, thin] ({x},{sin(pi*(x-4*\s)*\p/(2*\s))*sin(pi*(x-4*\s)*\p/(2*\s))-12});
				
				\addplot[thin, name path=b44] coordinates {(12,-12) (12.5,-12)};
				\addplot[domain=12:12.5, samples=100, color=black!75, thin, name path=t44] ({x},{sin(pi*(x-4*\s)*\p/(2*\s))*sin(pi*(x-4*\s)*\p/(2*\s))-12});
				
				\addplot[thin] coordinates {(12.5,-12) (16,-12)};
				\addplot[domain=12.5:16, samples=100, color=black!75, thin] ({x},{sin(pi*(x-4*\s)*\p/(2*\s))*sin(pi*(x-4*\s)*\p/(2*\s))-12});
				
				\addplot[thin] coordinates {(1,-15) (5,-15)};
				\addplot[domain=1:5, samples=100, color=black!75, thin] ({x},{sin(pi*(x-5*\s)*\p/(2*\s))*sin(pi*(x-5*\s)*\p/(2*\s))-15});
				
				\addplot[thin,  name path=t5] coordinates {(9,-15) (13,-15)};
				\addplot[domain=9:13, samples=100, color=black!75, thin, name path=b5]         ({x},{sin(pi*(x-5*\s)*\p/(2*\s))*sin(pi*(x-5*\s)*\p/(2*\s))-15});
				
				\addplot[thin,color=blue,->] coordinates {(3,0+0.0) (3,-3-0.0)};
				\addplot[thin,color=blue,->] coordinates {(7,-3) (7,-6)};
				\addplot[thin,color=blue,->] coordinates {(10,-6) (10,-9)};
				\addplot[thin,color=blue,->] coordinates {(13.5,-9) (13.5,-12)};
				\addplot[thin, dashed, gray] coordinates {(5,-3) (9,-15)};
				\addplot[thin, dashed, gray] coordinates {(9,-3) (13,-15)};
			\end{axis}
		\end{tikzpicture}
		\caption{Marked with the dashed lines is the ray $(I_{1,j})_{j \geq 0}$. The \CR{shaded} area represents $A_{2.5}(\tau; 1)$ for the given sequence $\CB{\tau}$, without considering the integrand factor $1/(\pi x)^2$. }
		\label{fig:defray}
	\end{figure}
	
	Let $\varrho=2-4/p$. Given integers $k, j \geq 0$ we define the interval
	\[I_{k,j} = (k+j\varrho, k+j\varrho+2/p)\] on level $n=k+2j$. For a fixed $k \geq 0$ we then refer to the collection of intervals $(I_{k,j})_{j \geq 0}$ as the \emph{$k$th ray}. Note that there is always exactly one or zero intervals on a fixed level $n$ on the $k$th ray. We then let
	\[A_p(\tau; k) = \sum_{j=0}^{\infty} \int_{I_{k,j}} \chi_{(\tau_{k+2j},\tau_{k+2j+1})}(x) \frac{\sin^2\frac{p}{2}\pi (x-k-2j)}{\pi^2 x^2}dx.\]
	We can think of $A_p(\tau; k)$ as the contribution to the integral $E_p(\tau)$ coming from the $k$th ray. See Figure \ref{fig:defray} for an illustration of rays and $A_p(\tau; k)$.
	In particular we have the following. 
	\begin{lemma}\label{lem:A(n)=Dp}
		Let $2 \leq p \leq 4$. Assume $1-2/p \leq \delta_1 \leq 1$ and $1-2/p \leq \delta_2 \leq 1$. Let $\tau$ be a sequence in $T(\delta_1, \delta_2)$ such that $ \tau_n \leq n $ for all $n \geq 1$. Then
		\[ E_p(\tau)= \sum_{k=0}^{\infty}A_p(\tau;k). \]
	\end{lemma}
	\begin{proof}
		We consider the collection of intervals $I=(I_{k, j})_{k, j \geq 0}$ and observe that the intervals on level $n$ in the collection $I$ are exactly the intervals $(\xi, \xi + 2/p)$ on level $n$ such that $ n(1-2/p) \leq \xi \leq n $. Further, due to the separation conditions we have $\tau_n \geq \delta_1+ \delta_2(n-1) \geq n(1-2/p)$. Hence it follows that 
		\[E_p(\tau)=\sum_{k=0}^{\infty}A_p(\tau; k). \qedhere\]
	\end{proof}
	We note that when $n$ is even then $n(1-2/p)$ is the left endpoint of an interval at level $n$ whereas when $n$ is odd then $n(1-2/p)$ is the right endpoint of an interval at level $n$. We now proceed with a lemma to show that the largest contribution to the sum $E_p(\tau)$ we can possibly get from the $k$th ray is the one from the first interval $(k, k+2/p)$ when $\tau_k \leq k$ and $\tau_{k+1} \geq k+2/p$.
	This lemma will be key in proving Theorem \ref{thm:E_p2to4:gensep}. Figure \ref{fig:p2p5} illustrates Lemma \ref{lem:F(n)}. 
	\begin{lemma}\label{lem:F(n)} 
		Let $2 \leq p \leq 4$. Assume $1-2/p \leq \delta_1 \leq 1$ and $ 2-4/p \leq  \delta_2 \leq 1$. Let $\tau$ be a sequence in $T (\delta_1, \delta_2)$ such that $\tau_n \leq n$ for all $n \geq 1$. Then 
		\[A_p(\tau; k) \leq  \int_{k}^{k+\frac{2}{p}} \frac{\sin^2 \frac{p}{2}\pi(x-k)}{\pi^2 x^2}\, dx,\]
		for all $k \geq 0$. Furthermore equality is attained if and only if $p=4$ or $\tau_{k+1} \geq k+2/p$.
	\end{lemma}
	
	\begin{proof}
		By a substitution and using that $j \varrho \geq 0$, we get 
		\begin{equation}\label{eq:Apupper}
			\begin{split}
				A_p(\tau; k)&=\sum_{j=0}^{\infty} \int_{k+j \varrho}^{k+j \varrho +2/p} \chi_{(\tau_{k+2j},\tau_{k+2j+1})}(x) \frac{\sin^2\frac{p}{2}\pi (x-k-2j)}{\pi^2 x^2} \, dx\\
				&= \sum_{j=0}^{\infty} \int_{k}^{k+2/p} \chi_{(\tau_{k+2j} - j \varrho, \tau_{k+2j+1}-j \varrho)}(x)\frac{\sin^2\frac{p}{2}\pi (x-k)}{\pi^2 (x+j\varrho)^2} \, dx\\
				&\leq \sum_{j=0}^{\infty} \int_{k}^{k+2/p} \chi_{(\tau_{k+2j} - j \varrho, \tau_{k+2j+1}-j \varrho)}(x)\frac{\sin^2\frac{p}{2}\pi (x-k)}{\pi^2 x^2} \, dx.
			\end{split}
		\end{equation}
		Since $\tau_{n+1}-\tau_n \geq \varrho$ it follows that \[\tau_{k+2j+1}-j \varrho \leq \tau_{k+2(j+1)}-(j+1) \varrho,\] and this shows that there is no overlap between the indicator functions arising from the $j$th term and the succeeding term. It thus follows that 
		\[\begin{split}
			A_p(\tau; k)&\leq \sum_{j=0}^{\infty} \int_{k}^{k+2/p} \chi_{(\tau_{k+2j} - j \varrho, \tau_{k+2j+1}-j \varrho)}(x)\frac{\sin^2\frac{p}{2}\pi (x-k)}{\pi^2 x^2} \, dx\\
			&\leq  \int_{k}^{k+\frac{2}{p}} \frac{\sin^2 \frac{p}{2}\pi(x-k)}{\pi^2 x^2}\, dx.\end{split}\]
		Further, it is clear that equality is attained if $p=4$ or $\tau_{k+1} \geq k+2/p$. Assume  $p<4$ and $\tau_{k+1}<k+2/p$. If the $j$th term in \eqref{eq:Apupper} is non-zero for some $j>0$ it follows that the inequality \eqref{eq:Apupper} is strict. If no such integer $j$ exists, then 
		\[A_p(\tau; k)=\int_{k}^{k+2/p} \chi_{(k, \tau_{k+1})}\frac{\sin^2\frac{p}{2}\pi (x-k)}{\pi^2 x^2}\, dx < \int_{k}^{k+2/p} \frac{\sin^2\frac{p}{2}\pi (x-k)}{\pi^2 x^2}\, dx. \qedhere\]
	\end{proof}
	We are now ready to prove the main theorem of this section. 
	
	\begin{proof}[Proof of Theorem \ref{thm:E_p2to4:gensep}.]
		Let $\lambda$ be the sequence in $T(\delta_1, \delta_2)$ given by $\lambda_n=n-1+2/p$. Then it is clear that  
		\[E_p(\lambda)=\sum_{n=0}^{\infty} \int_{n}^{n+\frac{2}{p}} \frac{\sin^2 \frac{p}{2} \pi (x-n)}{\pi^2 x^2 } dx.\] Combining this with Lemma \ref{lem:tau_n_leq_n} it is enough to show $E_p(\tau) \leq E_p(\lambda)$ for all sequences $\tau$ in $T(\delta_1, \delta_2)$ such that $\tau_n \leq n$ for all $n \geq 1$. Fix such a sequence $\tau$.
		Applying Lemma \ref{lem:A(n)=Dp} and Lemma \ref{lem:F(n)} it follows that
		\[E_p(\tau) = \sum_{k=0}^{\infty}A_p(\tau; k) \leq \sum_{k=0}^{\infty} \int_{k}^{k+\frac{2}{p}} \frac{\sin^2 \frac{p}{2} \pi (x-k)}{\pi^2 x^2 } dx=E_p(\lambda),\] and from Lemma \ref{lem:F(n)} it follows that equality is attained if and only if $p=4$ or $n-1+2/p \leq \tau_n \leq n$ for all $n \geq 1$. 
	\end{proof}
	
	\begin{figure}
		\centering
		\begin{tikzpicture}[scale=1.9]
			\def\p{2.5}
			\def\s{5}
			\begin{axis}[
				axis equal image,
				axis lines = none,
				trig format plots=rad]
				
				\node at (axis cs: -1,0) {$\scriptstyle 0$};
				\node at (axis cs: -1,-3) {$\scriptstyle 1$};
				\node at (axis cs: -1,-6) {$\scriptstyle 2$};
				\node at (axis cs: -1,-9) {$\scriptstyle 3$};
				\node at (axis cs: -1,-12) {$\scriptstyle 4$};
				\node at (axis cs: -1,-15) {$\scriptstyle 5$};
				\node at (axis cs: -1,-18) {$\scriptstyle 6$};
				\node at (axis cs: -1,-21) {$\scriptstyle 7$};
				\node at (axis cs: -1,-24) {$\scriptstyle 8$};
				
				\addplot[thin, name path=b1] coordinates {(0,0) (1.5,0)};
				\addplot[domain=0:1.5, samples=100, color=black!75, thin, name path=t1] ({x},{sin(pi*x*\p/(2*\s))*sin(pi*x*\p/(2*\s))});
				\addplot[red!50] fill between [of=b1 and t1];
				
				\addplot[thin] coordinates {(1.5,0) (2/\p,0)};
				\addplot[domain=1.5:4, samples=100, color=black!75, thin] ({x},{sin(pi*x*\p/(2*\s))*sin(pi*x*\p/(2*\s))});

				\addplot[thin] coordinates {(8,0) (12,0)};
				\addplot[domain=8:12, samples=100, color=black!75, thin] ({x},{sin(pi*x*\p/(2*\s))*sin(pi*x*\p/(2*\s))});
				
				\addplot[thin] coordinates {(16,0) (20,0)};
				\addplot[domain=16:20, samples=100, color=black!75, thin] ({x},{sin(pi*x*\p/(2*\s))*sin(pi*x*\p/(2*\s))});

				\addplot[thin] coordinates {(0,-3) (1,-3)};
				\addplot[domain=0:1, samples=100, color=black!75, thin] ({x},{sin(pi*(x-\s)*\p/(2*\s))*sin(pi*(x-\s)*\p/(2*\s))-3});
				
				\addplot[thin] coordinates {(5,-3) (9,-3)};
				\addplot[domain=5:9, samples=100, color=black!75, thin] ({x},{sin(pi*(x-\s)*\p/(2*\s))*sin(pi*(x-\s)*\p/(2*\s))-3});
				
				\addplot[thin] coordinates {(13,-3) (17,-3)};
				\addplot[domain=13:17, samples=100, color=black!75, thin] ({x},{sin(pi*(x-\s)*\p/(2*\s))*sin(pi*(x-\s)*\p/(2*\s))-3});
				
				\addplot[thin] coordinates {(2,-6) (6,-6)};
				\addplot[domain=2:6, samples=100, color=black!75, thin] ({x},{sin(pi*(x-2*\s)*\p/(2*\s))*sin(pi*(x-2*\s)*\p/(2*\s))-6});
				
				\addplot[thin, name path=b2] coordinates {(4,-6) (6,-6)};
				\addplot[domain=4:6, samples=100, color=black!75, thin, name path=t2] ({x},{sin(pi*(x-2*\s)*\p/(2*\s))*sin(pi*(x-2*\s)*\p/(2*\s))-6});
				\addplot[red!50] fill between [of=b2 and t2];
				
				\addplot[thin] coordinates {(10,-6) (14,-6)};
				\addplot[domain=10:14, samples=100, color=black!75, thin] ({x},{sin(pi*(x-2*\s)*\p/(2*\s))*sin(pi*(x-2*\s)*\p/(2*\s))-6});
				
				\addplot[thin] coordinates {(18,-6) (22,-6)};
				\addplot[domain=18:22, samples=100, color=black!75, thin] ({x},{sin(pi*(x-2*\s)*\p/(2*\s))*sin(pi*(x-2*\s)*\p/(2*\s))-6});
				
				\addplot[thin] coordinates {(0,-9) (3,-9)};
				\addplot[domain=0:3, samples=100, color=black!75, thin] ({x},{sin(pi*(x-3*\s)*\p/(2*\s))*sin(pi*(x-3*\s)*\p/(2*\s))-9});
				
				\addplot[thin, name path=b3] coordinates {(7,-9) (9,-9)};
				\addplot[domain=7:9, samples=100, color=black!75, thin, name path=t3] ({x},{sin(pi*(x-3*\s)*\p/(2*\s))*sin(pi*(x-3*\s)*\p/(2*\s))-9});
				\addplot[green!50] fill between [of=b3 and t3];
				
				\addplot[thick] coordinates {(15,-9) (19,-9)};
				\addplot[domain=15:19, samples=100, color=black!75, thick] ({x},{sin(pi*(x-3*\s)*\p/(2*\s))*sin(pi*(x-3*\s)*\p/(2*\s))-9});
				
				\addplot[thin] coordinates {(9,-9) (11,-9)};
				\addplot[domain=9:11, samples=100, color=black!75, thin] ({x},{sin(pi*(x-3*\s)*\p/(2*\s))*sin(pi*(x-3*\s)*\p/(2*\s))-9});
				
				\addplot[thin] coordinates {(4,-12) (8,-12)};
				\addplot[domain=4:8, samples=100, color=black!75, thin] ({x},{sin(pi*(x-4*\s)*\p/(2*\s))*sin(pi*(x-4*\s)*\p/(2*\s))-12});
				
				\addplot[thin, name path=b44] coordinates {(12,-12) (12.5,-12)};
				\addplot[domain=12:12.5, samples=100, color=black!75, thin, name path=t44] ({x},{sin(pi*(x-4*\s)*\p/(2*\s))*sin(pi*(x-4*\s)*\p/(2*\s))-12});
				\addplot[brown!75] fill between [of=b44 and t44];
				
				\addplot[thin] coordinates {(12.5,-12) (16,-12)};
				\addplot[domain=12.5:16, samples=100, color=black!75, thin] ({x},{sin(pi*(x-4*\s)*\p/(2*\s))*sin(pi*(x-4*\s)*\p/(2*\s))-12});
				
				\addplot[very thick,solid, violet, name path=b33] coordinates {(20,-12) (24,-12)};
				\addplot[domain=20:24, samples=100, color=black!75, very thick, violet, name path=t33] ({x},{sin(pi*(x-4*\s)*\p/(2*\s))*sin(pi*(x-4*\s)*\p/(2*\s))-12});
				
				
				\addplot[thin] coordinates {(1,-15) (5,-15)};
				\addplot[domain=1:5, samples=100, color=black!75, thin] ({x},{sin(pi*(x-5*\s)*\p/(2*\s))*sin(pi*(x-5*\s)*\p/(2*\s))-15});
				
				\addplot[thin] coordinates {(9,-15) (12.5,-15)};
				\addplot[domain=9:12.5, samples=100, color=black!75, thin] ({x},{sin(pi*(x-5*\s)*\p/(2*\s))*sin(pi*(x-5*\s)*\p/(2*\s))-15});
				
				\addplot[thin, black, name path=t5] coordinates {(12.5,-15) (13,-15)};
				\addplot[domain=12.5:13, samples=100, color=black!75, thin, black, name path=b5] ({x},{sin(pi*(x-5*\s)*\p/(2*\s))*sin(pi*(x-5*\s)*\p/(2*\s))-15});
				\addplot[green!75] fill between [of=b5 and t5];
				
				\addplot[thin] coordinates {(17,-15) (21,-15)};
				\addplot[domain=17:21, samples=100, color=black!75, thin] ({x},{sin(pi*(x-5*\s)*\p/(2*\s))*sin(pi*(x-5*\s)*\p/(2*\s))-15});
				
				\addplot[thin] coordinates {(0,-18) (2,-18)};
				\addplot[domain=0:2, samples=100, color=black!75, thin]
				({x},{sin(pi*(x-6*\s)*\p/(2*\s))-18});
				
				\addplot[thin] coordinates {(6,-18) (10,-18)};
				\addplot[domain=6:10, samples=100, color=black!75, thin]
				({x},{sin(pi*(x-6*\s)*\p/(2*\s))*sin(pi*(x-6*\s)*\p/(2*\s))-18});
				
				\addplot[thin, name path=b6] coordinates {(14.5,-18) (17,-18)};
				\addplot[domain=14.5:17, samples=100, color=black!75, thin, name path=t6]
				({x},{sin(pi*(x-6*\s)*\p/(2*\s))*sin(pi*(x-6*\s)*\p/(2*\s))-18});
				\addplot[brown!75] fill between [of=b6 and t6];
				
				\addplot[thin] coordinates {(14,-18) (14.5,-18)};
				\addplot[domain=14:14.5, samples=100, color=black!75, thin]
				({x},{sin(pi*(x-6*\s)*\p/(2*\s))*sin(pi*(x-6*\s)*\p/(2*\s))-18});
				
				\addplot[thin] coordinates {(17,-18) (18,-18)};
				\addplot[domain=17:18, samples=100, color=black!75, thin]
				({x},{sin(pi*(x-6*\s)*\p/(2*\s))*sin(pi*(x-6*\s)*\p/(2*\s))-18});
				
				\addplot[thin] coordinates {(22,-18) (26,-18)};
				\addplot[domain=22:26, samples=100, color=black!75, thin]
				({x},{sin(pi*(x-6*\s)*\p/(2*\s))*sin(pi*(x-6*\s)*\p/(2*\s))-18});
				\addplot[thin] coordinates {(3,-21) (7,-21)};
				\addplot[domain=3:7, samples=100, color=black!75, thin]
				({x},{sin(pi*(x-7*\s)*\p/(2*\s))*sin(pi*(x-7*\s)*\p/(2*\s))-21});
				
				\addplot[thin] coordinates {(11,-21) (15,-21)};
				\addplot[domain=11:15, samples=100, color=black!75, thin]
				({x},{sin(pi*(x-7*\s)*\p/(2*\s))*sin(pi*(x-7*\s)*\p/(2*\s))-21});
				
				\addplot[thin] coordinates {(19,-21) (23,-21)};
				\addplot[domain=19:23, samples=100, color=black!75, thin]
				({x},{sin(pi*(x-7*\s)*\p/(2*\s))*sin(pi*(x-7*\s)*\p/(2*\s))-21});
				
				\addplot[thin] coordinates {(0,-24) (4,-24)};
				\addplot[domain=0:4, samples=100, color=black!75, thin]
				({x},{sin(pi*(x-8*\s)*\p/(2*\s))*sin(pi*(x-8*\s)*\p/(2*\s))-24});
				
				\addplot[thin] coordinates {(8,-24) (12,-24)};
				\addplot[domain=8:12, samples=100, color=black!75, thin]
				({x},{sin(pi*(x-8*\s)*\p/(2*\s))*sin(pi*(x-8*\s)*\p/(2*\s))-24});
				
				\addplot[thin] coordinates {(16,-24) (19,-24)};
				\addplot[domain=16:19, samples=100, color=black!75, thin]
				({x},{sin(pi*(x-8*\s)*\p/(2*\s))*sin(pi*(x-8*\s)*\p/(2*\s))-24});
				
				\addplot[thin, name path=b8] coordinates {(19,-24) (20,-24)};
				\addplot[domain=19:20, samples=100, color=black!75, thin, name path=t8]
				({x},{sin(pi*(x-8*\s)*\p/(2*\s))*sin(pi*(x-8*\s)*\p/(2*\s))-24});
				\addplot[brown!75] fill between [of=b8 and t8];
				
				\addplot[thin, name path=b88] coordinates {(24,-24) (28,-24)};
				\addplot[domain=24:28, samples=100, color=black!75, thin, name path=t88]
				({x},{sin(pi*(x-8*\s)*\p/(2*\s))*sin(pi*(x-8*\s)*\p/(2*\s))-24});
				\addplot[violet!75] fill between [of=b88 and t88];

				\addplot[thin,color=blue,->] coordinates {(1.5,0+0.0) (1.5,-3-0.0)};
				\addplot[thin,color=blue,->] coordinates {(4,-3+0.0) (4,-6-0.0)};
				\addplot[thin,color=blue,->] coordinates {(6.5,-6+0.0) (6.5,-9-0.0)};
				\addplot[thin,color=blue,->] coordinates {(9,-9+0.0) (9,-12-0.0)};
				\addplot[thin,color=blue,->] coordinates {(12.5,-12+0.0) (12.5,-15-0.0)};
				\addplot[thin,color=blue,->] coordinates {(14.5,-15+0.0) (14.5,-18-0.0)};
				\addplot[thin,color=blue,->] coordinates {(17,-18+0.0) (17,-21-0.0)};
				\addplot[thin,color=blue,->] coordinates {(19,-21+0.0) (19,-24-0.0)};
				
				\addplot[very thick,solid,red] coordinates {(0,0) (4,0)};
				\addplot[domain=0:4, samples=100, color=black!75, very thick, red] ({x},{sin(pi*x*\p/(2*\s))*sin(pi*x*\p/(2*\s))});
				
				\addplot[very thick,solid,green] coordinates {(5,-3) (9,-3)};
				\addplot[domain=5:9, samples=100, color=black!75, very thick, green] ({x},{sin(pi*(x-1*\s)*\p/(2*\s))*sin(pi*(x-1*\s)*\p/(2*\s))-3});

				\addplot[very thick,solid,brown] coordinates {(10,-6) (14,-6)};
				\addplot[domain=10:14, samples=100, color=black!75, very thick, brown] ({x},{sin(pi*(x-2*\s)*\p/(2*\s))*sin(pi*(x-2*\s)*\p/(2*\s))-6});
				
				\addplot[very thick,solid,teal] coordinates {(15,-9) (19,-9)};
				\addplot[domain=15:19, samples=100, color=black!75, very thick, teal] ({x},{sin(pi*(x-3*\s)*\p/(2*\s))*sin(pi*(x-3*\s)*\p/(2*\s))-9});
				
				\addplot[thin,color= gray,->] coordinates {(5,0+0.0) (5,-3-0.0)};
				\addplot[thin,color= gray,->] coordinates {(10,-3+0.0) (10,-6-0.0)};
				\addplot[thin,color= gray,->] coordinates {(15,-6+0.0) (15,-9-0.0)};
				\addplot[thin,color= gray,->] coordinates {(20,-9+0.0) (20,-12-0.0)};
				
				\addplot[thin, dashed, gray] coordinates {(0,0) (8,-24)};
				\addplot[thin, dashed, gray] coordinates {(4,0) (12,-24)};
				\addplot[thin, dashed, gray] coordinates {(9,-3) (16,-24)};
				\addplot[thin, dashed, gray] coordinates {(14,-6) (20,-24)};
				\addplot[thin, dashed, gray] coordinates {(19,-9) (24,-24)};
				\addplot[thin, dashed, gray] coordinates {(24,-12) (28,-24)};
			\end{axis}
		\end{tikzpicture}
		\caption{The sequence $\CB{\tau}$ where $\tau_1=3/10$, $\tau_2=4/5$, $\tau_3=13/10$, $\tau_4=9/5$, $\tau_5=5/2$, $\tau_6=29/10$, $\tau_7=17/5$, $\tau_8=19/5$ and $\tau_9 \geq 22/5$ for $p=2.5$. Shaded is an illustration without considering the integrand factor $1/(\pi x)^2$ of \CR{$A_p(\tau; 0)$},  \CG{$A_p(\tau; 1)$}, \CBR{$A_p(\tau; 2)$} and \CV{$A_p(\tau; 4)$}. Note that $\CT{A_p(\tau; 3)}=0$. The thick lines illustrate the parts of $E_p(\lambda)$ that the areas $A_p(\tau; n)$ are compared with in Lemma \ref{lem:F(n)}. The dashed lines show the rays.} 
		\label{fig:p2p5}
	\end{figure}
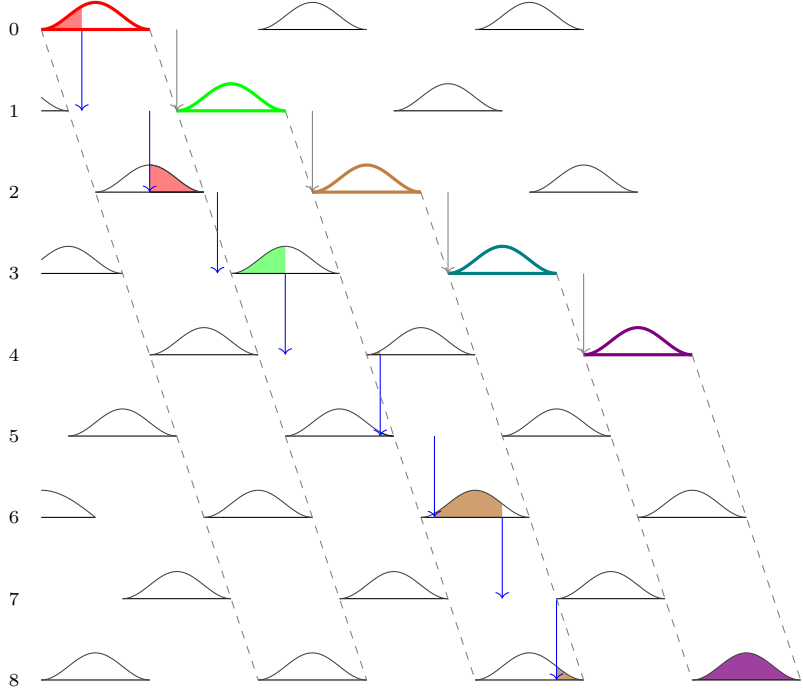
	
	\section{Proof of Theorem \ref{thm:new:Ep2to4} for \texorpdfstring{$3<p \leq 4$}{3<p≤4}.}\label{sec:3to4}
	
	In this section we prove Theorem \ref{thm:new:Ep2to4} in the range $3<p\leq 4$. We note that the case $p=4$ is trivial because $E_p(\tau)$ takes the same value for all sequences $\tau$. We will thus assume $p<4$. 
	To make calculations easier we will introduce a new set of sequences. Recall that by Lemma \ref{lem:tau_n_leq_n} we need only consider sequences $\tau$ where $\tau_n \leq n$ for all $n$. We also recall that for an interval $(\xi, \xi+2/p)$ at level $n$ and the intersecting interval $(\xi+1-4/p, \xi+1-2/p)$ at level $n+1$ we denote the midpoint between these two intervals by $m_\xi$ as defined in \eqref{eq:defmxi}. We now denote by $T_1$ the set of sequences $\tau=(\tau_n)_{n \geq 0}$ such that $\tau_0=0$, $\tau_1 \geq \min(2/\pi, 2/p)$, $\tau_n \leq n$ and at least one of the following three properties holds for each $n \geq 0$:
	\begin{enumerate}
		\item[(i)] $\tau_{n+1}-\tau_{n} \geq 2/3$.
		\item[(ii)] There exists $\xi \geq 1$ such that $(\xi, \xi+2/p)$ is an interval at level $n$ and $\tau_{n} < m_{\xi-1}+4/p$, $\tau_{n+1} = \xi+2/p$ and $\tau_{n+2} \geq m_\xi+4/p$. 
		\item[(iii)] There exists $\xi \geq 1$ such that $(\xi, \xi+2/p)$ is an interval at level $n$, $\tau_{n} = \xi-1+2/p$ and $\tau_{n+1} \geq m_\xi$.
	\end{enumerate}
	Note that in (ii) the assumption $\tau_{n+2} \geq m_\xi+4/p$ is motivated by the assumption $\tau_{n+2} \geq \tau_{n-1}+2$, that is on average the separation should not deviate too much from $2/3$.
	Let $\min(2/\pi, 2/p) \leq \delta_1 \leq 1$ and $2/3 \leq \delta_2 \leq 1$ then it is clear that the set of sequences in $T(\delta_1, \delta_2)$ for which $\tau_n \leq n$ for all $n$ is a subset of $ T_1$.  Consequently
	\begin{equation}\label{eq:supT1}
		\mathscr{E}_p(\delta_1, \delta_2) \leq \sup_{\tau \in T_1(\delta_1, \delta_2)} E_p(\tau).
	\end{equation}
	Thus to prove Theorem \ref{thm:new:Ep2to4} for $3<p<4$ we show that
	\[\sup_{\tau \in T_1} E_p(\tau) = E_p(\lambda),\]
	where $\lambda$ is the sequence given by $\lambda_n=n-1+2/p$ for all $n \geq 1$. An example of a sequence $\tau$ which is in $T_1$, but not in $T(2/\pi, 2/3)$, can be seen in Figure \ref{fig:3to4mainlemtau}.
	
	We will see in Lemma \ref{lem:T3} below that it is enough to consider the supremum in \eqref{eq:supT1} over a smaller subset $T_2$ of $T_1$. To show this we first establish several lemmas which allow us to make assumptions on $\tau_{n+1}$ (and in some cases also several succeeding terms) based on the value $\tau_n$. We start with a generalization of Lemma \ref{lem:tau_n_leq_n}.
	
	\begin{lemma}\label{lem:taunleqngen}
		Let $\tau$ be a sequence in $T_1$ and define the sequence $\gamma$ by $\gamma_0=0$, $\gamma_1=2/p$ and
		\[\gamma_{n+1}=\min(\tau_{n+1}, \xi+2/p),\] where $\xi=n-(4/p)j$ for some integer $j$ such that
		\[m_{\xi-1} < \tau_n \leq m_{\xi-1}+\frac{4}{p}.\] Then $\gamma$ is a sequence in $T_1$ and $E_p(\tau) \leq E_p(\gamma)$.
	\end{lemma}
	
	\begin{proof}
		We start by showing that $\gamma$ is in $T_1$. By definition $\gamma_0=0$ and $\gamma_1=2/p$, so it remains to show that the pair $\gamma_n$ and $\gamma_{n+1}$ satisfies one of the three conditions in the definition of $T_1$ for each $n$. Fix $n$ and let $\xi=n-(4/p)j$ be such that $m_{\xi-1}<\tau_n \leq m_{\xi-1}+4/p$. Then $\gamma_n \leq \tau_n \leq m_{\xi-1}+4/p$. If $\tau_{n+1} \leq \xi+2/p$ it follows that $\gamma_{n+1}=\tau_{n+1}$ and at least one of the three conditions is satisfied. 
		
		Now assume $\tau_{n+1} > \xi+2/p$. Then  $\gamma_{n+1}=\xi+2/p$. If $\tau_{n+2} \geq m_\xi+4/p$ it follows that $\gamma_{n+2}=\min(\gamma_{n+1}, \xi+1+2/p) \geq m_{\xi}+4/p$ and thus $\gamma$ is in $T_1$ by condition (ii). On the other hand, if $\tau_{n+2} < m_\xi+4/p$ it follows that \[\xi+2/p<\tau_{n+1} \leq \tau_{n+2}-2/3 < \xi+1\] and in particular since $\tau_{n+1}>\xi+2/p$ it must be the case that \[\tau_n \leq \tau_{n+1}-2/3 \leq \tau_{n+2}-4/3.\] We claim that $\gamma_n=\xi-1+2/p$, and so $\gamma_{n+1}-\gamma_n=1$ will satisfy condition (i). If $\tau_n \leq \xi-1+2/p$ this is trivial. If $\tau_n>\xi-1+2/p$ it must be the case that \[\tau_{n-1} \leq \tau_n-2/3 \leq \tau_{n+2}-2 \leq m_\xi-2+4/p,\] and so $\gamma_n=\min(\xi-1+2/p, \tau_n)=\xi-1+2/p$. Finally, combining Lemma \ref{lem:Brevig:lem6.5} and the fact that $x \mapsto 1/x$ is a decreasing function it follows that $E_p(\tau) \leq E_p(\gamma)$. 
	\end{proof}
	
	We will see that if an element $\tau_{n+1}$ is not a right endpoint of an interval at level $n$, meaning $\tau_{n+1}\neq \xi+2/p$, then it will be important for us whether $\tau_{n+1}$ is greater than or less than $m_\xi$. Note, however, that if $\tau_n=\xi-1+2/p$ then the latter can only happen if $p>3.6$. Hence for some of the proceeding results we will consider the cases $p \geq 3.6$ and $p<3.6$ separately, and some results will only be needed in one of the two cases.
	
	\begin{lemma}\label{lem:pleq36spes1}
		Let $3 < p < 3.6$ and assume $\tau$ is a sequence in $T_1$ such that \[\xi+5/6-1/p \leq \tau_{k+1}<\xi+2/p\] where $(\xi, \xi+2/p)$ is an interval at level $k$. Let $\gamma$ be a sequence given by $\gamma_n=\tau_n$ for all $n \neq k+1$ and $\gamma_{k+1}=\xi+2/p$. Then $\gamma$ is in $T_1$ and $E_p(\tau)<E_p(\gamma)$.
	\end{lemma}
	
	\begin{proof}
		We see that $\gamma$ is in $T_1$ due to property (iii) and the fact that $\tau$ is in $T_1$. It follows from Lemma \ref{lem:midpointlem} that $E_p(\tau)<E_p(\gamma)$.
	\end{proof}	  
	
	The next lemma tells us that if $\tau_{n+1}$ is smaller than the corresponding midpoint we can safely assume $\tau_{n+1}=\tau_n+2/3$.
	
	\begin{lemma}\label{lem:pgeq36spes2}
		Let $3.6<p<4$. Let $\tau$ be a sequence in $T_1$. Fix an integer $j>0$ and let $k$ be the smallest integer such that $\tau_{k+1} <k-(4/p)j+2/p$. Assume further that $\tau_{k+1}<m_{k}-(4/p)j$, and $\tau_{k+1} \neq \tau_k+2/3$. Let $\gamma$ be the sequence given by $\gamma_n=\tau_n$ for all $n \neq k+1$ and assume $\gamma_{k+1}=\gamma_{k}+2/3$. Then $\gamma$ is in $T_1$ and $E_p(\tau)<E_p(\gamma)$.
	\end{lemma}
	
	\begin{proof}
		It follows that $\gamma$ is in $T_1$ since $\tau$ is in $T_1$ and $\gamma_{k+1}=\gamma_k+2/3$. Furthermore it follows from Lemma \ref{lem:midpointlem} that $E_p(\tau)<E_p(\gamma)$.
	\end{proof}
	
	Next we move on to the case where we have an element $\tau_{n+1}$ in the sequence $\tau$ greater than the corresponding midpoint. The next lemma tells us that we may then assume $\tau_{n+2}=\tau_{n+1}+2/3$.
	
	\begin{lemma}\label{lem:23diffbestleq36}
		Assume $3<p<4$. Let $ \tau$ be a sequence in $T_1$. Fix an integer $j \geq 0$. If it exists, let $k \geq 1$ be the  smallest integer such that $\tau_{k+1}<k-(4/p)j+2/p$, and let $\xi=k-(4/p)j$. Further, if $p>3.6$ assume that $\tau_{k+1} \geq m_{\xi}$. If $\tau_{k+2} \neq \tau_{k+1} + 2/3$, then there exists a sequence $\gamma$ in $T_1$ such that $E_p(\tau) < E_p(\gamma)$. 
	\end{lemma}
	
	\begin{proof}
		We start by claiming that we can assume $\tau_{k+1} \geq m_\xi=\xi+1/2-1/p$. For $p>3.6$ this follows from assumption, so let $p\leq 3.6$.
		We first see that it holds in the case $k=1$, $j=0$ since $\min(2/\pi, 2/p)+2/3 \geq 3/2-1/p $. In all other cases it then follows that 
		\[\tau_{k+1} \geq \tau_k+2/3 \geq \xi-1+2/p+2/3=\xi-1/3+2/p \geq m_\xi.\] From now on we consider $3<p<4$. Assume first that $\tau_{k+2} \geq m_\xi+1$. Then we define the sequence $\gamma$ by $\gamma_{k+1}=\xi+2/p$ and $\gamma_n=\tau_n$ for all $n \neq k+1$. We see that $\gamma$ is in $T_1$ since $\tau$ is in $T_1$  and  $\gamma_{k+2} \geq m_\xi+1$, so property (iii) in the definition of $T_1$ is satisfied. Further since $\tau_{k+1} \geq m_\xi$ it follows that $E_p(\gamma)>E_p(\tau)$ by Lemma \ref{lem:midpointlem}.
		
		Next assume $\tau_{k+2}<m_\xi+1$.  Then we define the sequence $\gamma$ by $\gamma_{k+2}=\gamma_{k+1}+2/3$ and $\gamma_n=\tau_n$ for all $n \neq k+2$. We see that $\gamma$ is in $T_1$ since $\tau$ is, and $\gamma_{k+2}-\gamma_{k+1} = \delta_2$. Further by Lemma \ref{lem:midpointlem} it follows that $E_p(\gamma)>E_p(\tau)$.
	\end{proof}
	
	In light of the lemmas above, let us now define a subset of sequences $T_2 \subseteq T_1$. Let $y_{max}=5/6-1/p$ if $p \leq 3.6$ and $y_{max}=2/p$ if $p>3.6$. We define $T_2$ to be the set of all sequences $\tau$ with the following properties. 
	\begin{enumerate}
		\item[(I)] $\tau_0=0$ and $\tau_1=2/p$.
		\item[(II)] If $\tau_{k+1} \leq m_k-(4/p)j$ for some integer $j$ then, $\tau_n \leq n-1-(4/p)j+2/p$ for all $n \geq k+2$.
		\item[(III)] Fix an integer $j \geq 0$. If it exists, let $k$ be the smallest integer such that $\tau_{k+1}<k+2/p-(4/p)j$ and let $\xi=k-(4/p)j$. Then one of the following two conditions holds:
		\begin{enumerate}
			\item[(a)] $m_\xi \leq \tau_{k+1}<\xi+y_{max}$, $\tau_{k+2}=\tau_{k+1}+2/3<m_{\xi}+1$, $\tau_{k+3}=\xi+2-2/p$ and $\tau_{k+4} \geq m_\xi+2$.
			\item[(b)] $p>3.6$, $\tau_{k+1}=\xi-1/3+2/p$, $\tau_{k+2}=\xi +1-2/p$, and $\tau_{k+3} \geq m_\xi+1$.
		\end{enumerate}
	\end{enumerate}
	Combining the above lemmas we obtain the following result.
	\begin{lemma}\label{lem:T3}
		Let $3<p<4$. Then
		\[\mathscr{E}_p(\min(2/p, 2/\pi), 2/3) \leq \sup_{\tau \in T_2} E_p(\tau).\]
	\end{lemma}	
	
	\begin{proof}
		We need to show that if $\tau$ in $T_1$ is a candidate for attaining the supremum then the properties (I), (II) and (III) hold.
		Combining the fact that $\tau_1 \geq \min(2/p, 2/\pi)$ and Lemma \ref{lem:taunleqngen} it follows that sequences not satisfying (I) and (II) can be excluded from the set of sequences we take the supremum over. For property (III) fix an integer $j \geq 0$ and assume there exists a smallest integer $k$ such that $\tau_{k+1}<k+2/p-(4/p)j$ and let $\xi=k-(4/p)j$. Then if $\tau_{k+1} \geq m_\xi$ it follows from Lemma \ref{lem:pleq36spes1} and Lemma \ref{lem:23diffbestleq36} that we can assume $m_\xi \leq \tau_{k+1}<\xi+y_{max}$ and $\tau_{k+2}=\tau_{k+1}+2/3<m_\xi+1$. Further by the definition of $T_1$ it follows that $\tau_{k+4} \geq m_\xi+2$
		since either $\tau_{k+4} \geq \tau_{k+1}+2 \geq m_\xi+2$ or 
		$\tau_{k+3}=\xi+2-2/p$ and $\tau_{k+4} \geq m_\xi+2$ by (ii). Thus the optimal choice is to follow (ii) and consider  $\tau_{k+3}=\xi+2-2/p$ and $\tau_{k+4} \geq m_\xi+2$. That is (III)(a) holds.
		
		If $\tau_{k+1} < m_\xi$ it follows that $p>3.6$ since $\tau_{k+1}>\tau_k+2/3$ and hence it follows from Lemma \ref{lem:pgeq36spes2} that $\tau_{k+1}=\xi-1/3+2/p$. Further by the definition of $T_1$ it follows that $\tau_{k+3} \geq m_\xi+1$ since either $\tau_{k+3} \geq \tau_{k+1}+4/3$ or 
		$\tau_{k+2}=\xi+1-2/p$, and $\tau_{k+3} \geq m_\xi+1$ (condition (ii)). Thus the optimal choice is to follow (ii) and consider $\tau_{k+2}=\xi+1-2/p$, and $\tau_{k+3} \geq m_\xi+1$, which means that (III)(b) holds.  
	\end{proof}
	Recall that $\lambda$ is the sequence given by $\lambda_n=n-1+2/p$ for $n \geq 1$. It is clear that $\lambda$ is in $T_2$, so to prove Theorem \ref{thm:new:Ep2to4} for $3<p<4$ we will show that 
	\[\sup_{\tau \in T_2} E_p(\tau)=E_p(\lambda).\]
	The following lemma essentially tells us that if we have a sequence $\tau$ in $T_2$ and an integer $k>1$ such that $\tau_n < n-1+2/p$ for all $n<k$ it follows that $E_p(\tau)$ increases with the value of $k$. See Figure \ref{fig:3to4mainlemtau} and Figure \ref{fig:3to4mainlemgamma} for an illustration. We will use this result iteratively to conclude the proof of Theorem \ref{thm:new:Ep2to4}.
	
	\begin{lemma}\label{lem:postponediviationgeneralized}
		Assume $3<p<4$. Let $\tau \neq \lambda $ be a sequence in $T_2$. Let $\gamma$ be the sequence given by $\gamma_1=2/p$ and $\gamma_n=\tau_{n-1}+1$ for $n >1$. Then $\gamma$ is in $T_2$ and $E_p(\tau) < E_p(\gamma)$. 
	\end{lemma}
	
	\begin{figure}
		\centering
		\begin{tikzpicture}[scale=1.8]
			\def\p{3.6}
			\def\s{2*\p}
			\def\tone{3*(\p/2-1)+2.5/18*2*\p}
			\def\ttwo{\tone+2/3*2*\p}
			\begin{axis}[
				axis equal image,
				axis lines = none,
				trig format plots=rad]
				
				\node at (axis cs: -1.5,0) {$\scriptstyle n$};
				\node at (axis cs: -1.5,-3) {$\scriptstyle n+1$};
				\node at (axis cs: -1.5,-6) {$\scriptstyle n+2$};
				\node at (axis cs: -1.5,-9) {$\scriptstyle n+3$};
				\node at (axis cs: -1.5,-12) {$\scriptstyle n+4$};
				
				\addplot[thin, name path=b1] coordinates {(0,0) (\tone,0)};
				\addplot[domain=0:\tone, samples=100, color=black!75, thin, name path=t1] ({x},{sin(pi*x*\p/(2*\s))*sin(pi*x*\p/(2*\s))});
				\addplot[red!50] fill between [of=b1 and t1];
				
				\addplot[thin] coordinates {(0,0) (4,0)};
				\addplot[domain=0:4, samples=100, color=black!75, thin] ({x},{sin(pi*x*\p/(2*\s))*sin(pi*x*\p/(2*\s))});
				
				\addplot[thin] coordinates {(8,0) (12,0)};
				\addplot[domain=8:12, samples=100, color=black!75, thin] ({x},{sin(pi*x*\p/(2*\s))*sin(pi*x*\p/(2*\s))});
				
				\addplot[thin] coordinates {(16,0) (20,0)};
				\addplot[domain=16:20, samples=100, color=black!75, thin] ({x},{sin(pi*x*\p/(2*\s))*sin(pi*x*\p/(2*\s))});
				
				\addplot[thin] coordinates {(24,0) (28,0)};
				\addplot[domain=24:28, samples=100, color=black!75, thin] ({x},{sin(pi*x*\p/(2*\s))*sin(pi*x*\p/(2*\s))});
				
				\addplot[thin] coordinates {(0,-3) (4*\p/2-4,-3)};
				\addplot[domain=0:4*\p/2-4, samples=100, color=black!75, thin] ({x},{sin(pi*(x-\s)*\p/(2*\s))*sin(pi*(x-\s)*\p/(2*\s))-3});
				
				\addplot[thin] coordinates {(4*\p/2,-3) (4*\p/2+4,-3)};
				\addplot[domain=4*\p/2:4*\p/2+4, samples=100, color=black!75, thin] ({x},{sin(pi*(x-\s)*\p/(2*\s))*sin(pi*(x-\s)*\p/(2*\s))-3});
				
				\addplot[thin, name path=b12] coordinates {(4*\p/2,-3) (\ttwo,-3)};
				\addplot[domain=4*\p/2:\ttwo, samples=100, color=black!75, thin, name path=t12] ({x},{sin(pi*(x-\s)*\p/(2*\s))*sin(pi*(x-\s)*\p/(2*\s))-3});
				\addplot[red!50] fill between [of=b12 and t12];
				
				\addplot[thin] coordinates {(4*\p/2+8,-3) (4*\p/2+12,-3)};
				\addplot[domain=4*\p/2+8:4*\p/2+12
				, samples=100, color=black!75, thin] ({x},{sin(pi*(x-\s)*\p/(2*\s))*sin(pi*(x-\s)*\p/(2*\s))-3});
				
				\addplot[thin] coordinates {(4*\p/2+16,-3) (4*\p/2+20,-3)};
				\addplot[domain=4*\p/2+16:4*\p/2+20
				, samples=100, color=black!75, thin] ({x},{sin(pi*(x-\s)*\p/(2*\s))*sin(pi*(x-\s)*\p/(2*\s))-3});
				
				\addplot[thin] coordinates {(0,-6) (2*4*\p/2-12,-6)};
				\addplot[domain=0:2*4*\p/2-12, samples=100, color=black!75, thin] ({x},{sin(pi*(x-2*\s)*\p/(2*\s))*sin(pi*(x-2*\s)*\p/(2*\s))-6});
				
				\addplot[thin] coordinates {(2*4*\p/2-8,-6) (2*4*\p/2-4,-6)};
				\addplot[domain=2*4*\p/2-8:2*4*\p/2-4, samples=100, color=black!75, thin] ({x},{sin(pi*(x-2*\s)*\p/(2*\s))*sin(pi*(x-2*\s)*\p/(2*\s))-6});
				
				\addplot[thin, name path=b2] coordinates {(\ttwo,-6) (2*4*\p/2-4,-6)};
				\addplot[domain=\ttwo:2*4*\p/2-4, samples=100, color=black!75, thin, name path=t2] ({x},{sin(pi*(x-2*\s)*\p/(2*\s))*sin(pi*(x-2*\s)*\p/(2*\s))-6});
				\addplot[red!50] fill between [of=b2 and t2];

				\addplot[thin] coordinates {(2*4*\p/2,-6) (2*4*\p/2+4,-6)};
				\addplot[domain=2*4*\p/2:2*4*\p/2+4, samples=100, color=black!75, thin] ({x},{sin(pi*(x-2*\s)*\p/(2*\s))*sin(pi*(x-2*\s)*\p/(2*\s))-6});
				
				\addplot[thin] coordinates {(2*4*\p/2+8,-6) (2*4*\p/2+12,-6)};
				\addplot[domain=2*4*\p/2+8:2*4*\p/2+12, samples=100, color=black!75, thin] ({x},{sin(pi*(x-2*\s)*\p/(2*\s))*sin(pi*(x-2*\s)*\p/(2*\s))-6});
				
				\addplot[thin] coordinates {(0,-9) (3*4*\p/2-20,-9)};
				\addplot[domain=0:3*4*\p/2-20, samples=100, color=black!75, thin] ({x},{sin(pi*(x-3*\s)*\p/(2*\s))*sin(pi*(x-3*\s)*\p/(2*\s))-9});
				
				\addplot[thin] coordinates {(3*4*\p/2-16,-9) (3*4*\p/2-12,-9)};
				\addplot[domain=3*4*\p/2-16:3*4*\p/2-12, samples=100, color=black!75, thin] ({x},{sin(pi*(x-3*\s)*\p/(2*\s))*sin(pi*(x-3*\s)*\p/(2*\s))-9});
				
				\addplot[thin, name path=b3] coordinates {(3*4*\p/2-8,-9) (3*4*\p/2-4,-9)};
				\addplot[domain=3*4*\p/2-8:3*4*\p/2-4, samples=100, color=black!75, thin, name path=t3] ({x},{sin(pi*(x-3*\s)*\p/(2*\s))*sin(pi*(x-3*\s)*\p/(2*\s))-9});
				\addplot[red!50] fill between [of=b3 and t3];
				
				\addplot[thin] coordinates {(3*4*\p/2,-9) (3*4*\p/2+4,-9)};
				\addplot[domain=3*4*\p/2:3*4*\p/2+4, samples=100, color=black!75, thin] ({x},{sin(pi*(x-3*\s)*\p/(2*\s))*sin(pi*(x-3*\s)*\p/(2*\s))-9});
				
				\addplot[thin] coordinates {(0,-12) (4*4*\p/2-28,-12)};
				\addplot[domain=0:4*4*\p/2-28, samples=100, color=black!75, thin] ({x},{sin(pi*(x-4*\s)*\p/(2*\s))*sin(pi*(x-4*\s)*\p/(2*\s))-12});
				
				\addplot[thin] coordinates {(4*4*\p/2-24,-12) (4*4*\p/2-20,-12)};
				\addplot[domain=4*4*\p/2-24:4*4*\p/2-20, samples=100, color=black!75, thin] ({x},{sin(pi*(x-4*\s)*\p/(2*\s))*sin(pi*(x-4*\s)*\p/(2*\s))-12});
				
				\addplot[thin] coordinates {(4*4*\p/2-16,-12) (4*4*\p/2-12,-12)};
				\addplot[domain=4*4*\p/2-16:4*4*\p/2-12, samples=100, color=black!75, thin] ({x},{sin(pi*(x-4*\s)*\p/(2*\s))*sin(pi*(x-4*\s)*\p/(2*\s))-12});
				
				\addplot[thin, name path=b4] coordinates {(4*4*\p/2-8,-12) (4*4*\p/2-4,-12)};
				\addplot[domain=4*4*\p/2-8:4*4*\p/2-4, samples=100, color=black!75, thin, name path=t4] ({x},{sin(pi*(x-4*\s)*\p/(2*\s))*sin(pi*(x-4*\s)*\p/(2*\s))-12});
				\addplot[red!50] fill between [of=b4 and t4];
				
				\addplot[thin, color=gray, dashed] coordinates {(\p-2,0+0.0) (\p-2,-3.1)};
				\addplot[thin, color=gray, dashed] coordinates {(\p-2+2*\p,-3+0.0) (\p-2+2*\p,-6.1)};
				\addplot[thin, color=gray, dashed] coordinates {(\p-2+4*\p,-6+0.0) (\p-2+4*\p,-9.1)};
				\addplot[thin, color=gray, dashed] coordinates {(\p-2+6*\p,-9+0.0) (\p-2+6*\p,-12.1)};
				
				\addplot[thin,color= blue,->] coordinates {(\tone,0+0.0) (\tone,-3-0.0)};
				\addplot[thin,color= blue,->] coordinates {(\ttwo,-3+0.0) (\ttwo,-6-0.0)};
				\addplot[thin,color= blue,->] coordinates {(2*4*\p/2-4,-6+0.0) (2*4*\p/2-4,-9-0.0)};
				\addplot[thin,color= blue,->] coordinates {(2*4*\p/2-4+2*\p,-9+0.0) (2*4*\p/2-4+2*\p,-12-0.0)};
			\end{axis}
		\end{tikzpicture}
		\caption{An example of a sequence \CB{$\tau$} as given in Lemma \ref{lem:postponediviationgeneralized}. The \CR{shaded} area represents $E_p(\tau)$ without considering the integrand factor $1/(\pi x)^2$. Figure \ref{fig:3to4mainlemgamma} shows the corresponding sequence $\gamma$ and $E_p(\gamma)$. Lemma \ref{lem:postponediviationgeneralized} tells us that $E_p(\tau)<E_p(\gamma)$. The dashed lines show the midpoints between two overlapping intervals.}
		\label{fig:3to4mainlemtau}
		
		\centering
		\begin{tikzpicture}[scale=1.8]
			\def\p{3.6}
			\def\s{2*\p}
			\def\toneorg{3*(\p/2-1)+2.5/18*2*\p}
			\def\gammaone{4}
			\def\gammatwo{\toneorg+2*\p}
			\def\gammathree{\gammatwo+2*\p*2/3}
			\begin{axis}[
				axis equal image,
				axis lines = none,
				trig format plots=rad]
				
				\node at (axis cs: -1.5,0) {$\scriptstyle n$};
				\node at (axis cs: -1.5,-3) {$\scriptstyle n+1$};
				\node at (axis cs: -1.5,-6) {$\scriptstyle n+2$};
				\node at (axis cs: -1.5,-9) {$\scriptstyle n+3$};
				\node at (axis cs: -1.5,-12) {$\scriptstyle n+4$};
				
				\addplot[thin, name path=b1] coordinates {(0,0) (\gammaone,0)};
				\addplot[domain=0:\gammaone, samples=100, color=black!75, thin, name path=t1] ({x},{sin(pi*x*\p/(2*\s))*sin(pi*x*\p/(2*\s))});
				\addplot[red!50] fill between [of=b1 and t1];
				
				\addplot[thin] coordinates {(0,0) (4,0)};
				\addplot[domain=0:4, samples=100, color=black!75, thin] ({x},{sin(pi*x*\p/(2*\s))*sin(pi*x*\p/(2*\s))});
				
				\addplot[thin] coordinates {(8,0) (12,0)};
				\addplot[domain=8:12, samples=100, color=black!75, thin] ({x},{sin(pi*x*\p/(2*\s))*sin(pi*x*\p/(2*\s))});
				
				\addplot[thin] coordinates {(16,0) (20,0)};
				\addplot[domain=16:20, samples=100, color=black!75, thin] ({x},{sin(pi*x*\p/(2*\s))*sin(pi*x*\p/(2*\s))});
				
				\addplot[thin] coordinates {(24,0) (28,0)};
				\addplot[domain=24:28, samples=100, color=black!75, thin] ({x},{sin(pi*x*\p/(2*\s))*sin(pi*x*\p/(2*\s))});
				
				\addplot[thin] coordinates {(0,-3) (4*\p/2-4,-3)};
				\addplot[domain=0:4*\p/2-4, samples=100, color=black!75, thin] ({x},{sin(pi*(x-\s)*\p/(2*\s))*sin(pi*(x-\s)*\p/(2*\s))-3});

				\addplot[thin] coordinates {(4*\p/2,-3) (4*\p/2+4,-3)};
				\addplot[domain=4*\p/2:4*\p/2+4, samples=100, color=black!75, thin] ({x},{sin(pi*(x-\s)*\p/(2*\s))*sin(pi*(x-\s)*\p/(2*\s))-3});
				
				\addplot[thin, name path=b12] coordinates {(4*\p/2,-3) (\gammatwo,-3)};
				\addplot[domain=4*\p/2:\gammatwo, samples=100, color=black!75, thin, name path=t12] ({x},{sin(pi*(x-\s)*\p/(2*\s))*sin(pi*(x-\s)*\p/(2*\s))-3});
				\addplot[red!50] fill between [of=b12 and t12];
				
				\addplot[thin] coordinates {(4*\p/2+8,-3) (4*\p/2+12,-3)};
				\addplot[domain=4*\p/2+8:4*\p/2+12
				, samples=100, color=black!75, thin] ({x},{sin(pi*(x-\s)*\p/(2*\s))*sin(pi*(x-\s)*\p/(2*\s))-3});
				
				\addplot[thin] coordinates {(4*\p/2+16,-3) (4*\p/2+20,-3)};
				\addplot[domain=4*\p/2+16:4*\p/2+20
				, samples=100, color=black!75, thin] ({x},{sin(pi*(x-\s)*\p/(2*\s))*sin(pi*(x-\s)*\p/(2*\s))-3});
				
				\addplot[thin] coordinates {(0,-6) (2*4*\p/2-12,-6)};
				\addplot[domain=0:2*4*\p/2-12, samples=100, color=black!75, thin] ({x},{sin(pi*(x-2*\s)*\p/(2*\s))*sin(pi*(x-2*\s)*\p/(2*\s))-6});
				
				\addplot[thin] coordinates {(2*4*\p/2-8,-6) (2*4*\p/2-4,-6)};
				\addplot[domain=2*4*\p/2-8:2*4*\p/2-4, samples=100, color=black!75, thin] ({x},{sin(pi*(x-2*\s)*\p/(2*\s))*sin(pi*(x-2*\s)*\p/(2*\s))-6});
				
				\addplot[thin] coordinates {(2*4*\p/2,-6) (2*4*\p/2+4,-6)};
				\addplot[domain=2*4*\p/2:2*4*\p/2+4, samples=100, color=black!75, thin] ({x},{sin(pi*(x-2*\s)*\p/(2*\s))*sin(pi*(x-2*\s)*\p/(2*\s))-6});
				
				\addplot[thin, name path=b22] coordinates {(2*4*\p/2,-6) (\gammathree,-6)};
				\addplot[domain=2*4*\p/2:\gammathree, samples=100, color=black!75, thin, name path=t22] ({x},{sin(pi*(x-2*\s)*\p/(2*\s))*sin(pi*(x-2*\s)*\p/(2*\s))-6});
				\addplot[red!50] fill between [of=b22 and t22];

				\addplot[thin] coordinates {(2*4*\p/2+8,-6) (2*4*\p/2+12,-6)};
				\addplot[domain=2*4*\p/2+8:2*4*\p/2+12, samples=100, color=black!75, thin] ({x},{sin(pi*(x-2*\s)*\p/(2*\s))*sin(pi*(x-2*\s)*\p/(2*\s))-6});
				
				\addplot[thin] coordinates {(0,-9) (3*4*\p/2-20,-9)};
				\addplot[domain=0:3*4*\p/2-20, samples=100, color=black!75, thin] ({x},{sin(pi*(x-3*\s)*\p/(2*\s))*sin(pi*(x-3*\s)*\p/(2*\s))-9});
				
				\addplot[thin] coordinates {(3*4*\p/2-16,-9) (3*4*\p/2-12,-9)};
				\addplot[domain=3*4*\p/2-16:3*4*\p/2-12, samples=100, color=black!75, thin] ({x},{sin(pi*(x-3*\s)*\p/(2*\s))*sin(pi*(x-3*\s)*\p/(2*\s))-9});
				
				\addplot[thin] coordinates {(3*4*\p/2-8,-9) (3*4*\p/2-4,-9)};
				\addplot[domain=3*4*\p/2-8:3*4*\p/2-4, samples=100, color=black!75, thin] ({x},{sin(pi*(x-3*\s)*\p/(2*\s))*sin(pi*(x-3*\s)*\p/(2*\s))-9});
				
				\addplot[thin, name path=b3] coordinates {(\gammathree,-9) (3*4*\p/2-4,-9)};
				\addplot[domain=\gammathree:3*4*\p/2-4, samples=100, color=black!75, thin, name path=t3] ({x},{sin(pi*(x-3*\s)*\p/(2*\s))*sin(pi*(x-3*\s)*\p/(2*\s))-9});
				\addplot[red!50] fill between [of=b3 and t3];
				
				\addplot[thin] coordinates {(3*4*\p/2,-9) (3*4*\p/2+4,-9)};
				\addplot[domain=3*4*\p/2:3*4*\p/2+4, samples=100, color=black!75, thin] ({x},{sin(pi*(x-3*\s)*\p/(2*\s))*sin(pi*(x-3*\s)*\p/(2*\s))-9});
				
				\addplot[thin] coordinates {(0,-12) (4*4*\p/2-28,-12)};
				\addplot[domain=0:4*4*\p/2-28, samples=100, color=black!75, thin] ({x},{sin(pi*(x-4*\s)*\p/(2*\s))*sin(pi*(x-4*\s)*\p/(2*\s))-12});
				
				\addplot[thin] coordinates {(4*4*\p/2-24,-12) (4*4*\p/2-20,-12)};
				\addplot[domain=4*4*\p/2-24:4*4*\p/2-20, samples=100, color=black!75, thin] ({x},{sin(pi*(x-4*\s)*\p/(2*\s))*sin(pi*(x-4*\s)*\p/(2*\s))-12});
				
				\addplot[thin] coordinates {(4*4*\p/2-16,-12) (4*4*\p/2-12,-12)};
				\addplot[domain=4*4*\p/2-16:4*4*\p/2-12, samples=100, color=black!75, thin] ({x},{sin(pi*(x-4*\s)*\p/(2*\s))*sin(pi*(x-4*\s)*\p/(2*\s))-12});
				
				\addplot[thin, name path=b4] coordinates {(4*4*\p/2-8,-12) (4*4*\p/2-4,-12)};
				\addplot[domain=4*4*\p/2-8:4*4*\p/2-4, samples=100, color=black!75, thin, name path=t4] ({x},{sin(pi*(x-4*\s)*\p/(2*\s))*sin(pi*(x-4*\s)*\p/(2*\s))-12});
				\addplot[red!50] fill between [of=b4 and t4];
				
				\addplot[thin, color=gray, dashed] coordinates {(\p-2,0+0.0) (\p-2,-3.1)};
				\addplot[thin, color=gray, dashed] coordinates {(\p-2+2*\p,-3+0.0) (\p-2+2*\p,-6.1)};
				\addplot[thin, color=gray, dashed] coordinates {(\p-2+4*\p,-6+0.0) (\p-2+4*\p,-9.1)};
				\addplot[thin, color=gray, dashed] coordinates {(\p-2+6*\p,-9+0.0) (\p-2+6*\p,-12.1)};
				
				\addplot[thin,color= blue,->] coordinates {(\gammaone,0+0.0) (\gammaone,-3-0.0)};
				\addplot[thin,color= blue,->] coordinates {(\gammatwo,-3+0.0) (\gammatwo,-6-0.0)};
				\addplot[thin,color= blue,->] coordinates {(\gammathree,-6+0.0) (\gammathree,-9-0.0)};
				\addplot[thin,color= blue,->] coordinates {(2*4*\p/2-4+2*\p,-9+0.0) (2*4*\p/2-4+2*\p,-12-0.0)};
			\end{axis}
		\end{tikzpicture}
		\caption{An example of a sequence \CB{$\gamma$} as given in Lemma \ref{lem:postponediviationgeneralized}. The \CR{shaded} area represents $E_p(\gamma)$ without considering the integrand factor $1/(\pi x)^2$. Figure \ref{fig:3to4mainlemtau} shows the corresponding sequence $\tau$. Lemma \ref{lem:postponediviationgeneralized} tells us that $E_p(\tau)<E_p(\gamma)$. The dashed lines show the midpoints between two overlapping intervals.}
		\label{fig:3to4mainlemgamma}
	\end{figure}
	
	To prove Lemma \ref{lem:postponediviationgeneralized} we need the next two lemmas which essentially tell us that Lemma \ref{lem:postponediviationgeneralized} holds locally.
	We thus introduce the following new notation. Let $r>0$ be some real number, and let $\tau^r$ be the (vertically) truncated sequence where $\tau_n^r=\min(\tau_n, r)$. In the local analysis we will study $E_p(\tau^r)$ for some real number $r>0$. Recall that $y_{max}=5/6-1/p$ for $p<3.6$ and $y_{max}=2/p$ for $p \geq 3.6$ 
	
	\begin{lemma}\label{lem:postponeactint}
		Assume $3<p<4$. Let $(\xi, \xi+2/p)$ be an interval at level $k$ such that $\xi \geq 1$. Let $1/2-1/p \leq y < y_{max}$. Let $\tau$ be a sequence in $T_2$ such that $\tau_k=\xi-1+2/p$, $\tau_{k+1}=\xi+y$, $\tau_{k+2}=\tau_{k+1}+2/3$ and $\tau_{k+3}=\xi+2-2/p$. Let $\gamma$ be the sequence given by $\gamma_n=\tau_n$ for all  $n \leq k+1$ and $\gamma_n=\tau_{n-1}+1$ for all $n>k$. Then $\gamma$ is in $T_2$ and 
		\begin{equation}\label{eq:appendixfinal}
			E_p(\tau^{r}) < E_p(\gamma^{r})
		\end{equation}
		where $r=\xi+3-2/p=\gamma_{k+4}$.
	\end{lemma}
	The fact that $\gamma$ is in $T_2$ is a direct consequence of the fact that $\tau$ is in $T_2$. The proof of \eqref{eq:appendixfinal} however consists of several long calculations and we postpone the proof to the appendix.
	Figure \ref{fig:3to4mainlemtau} and Figure \ref{fig:3to4mainlemgamma} give an illustration of Lemma \ref{lem:postponediviationgeneralized} and Lemma \ref{lem:postponeactint}. Similarly we have the following local result.
	\begin{lemma}\label{lem:pgeq36spes}
		Let $p>3.6$ and let $(\xi, \xi+2/p)$ be an interval at level $k$ such that $\xi \geq 1$. Let $\tau$ be a sequence in $T_2$ such that
		$\tau_k=\xi-1+2/p$, $\tau_{k+1}=\tau_{k}+2/3$ and $\tau_{k+2}=\xi+1-2/p$. Let $\gamma$ be the sequence given by $\gamma_n=\tau_n$ for $n \leq k$, and $\gamma_n=\tau_{n-1}+1$ for $n>k$
		Then 
		\begin{equation}\label{eq:appendixfinalpgeq36}
			E_p(\tau^{r}) < E_p(\gamma^{r}),
		\end{equation}
		where $r=\xi+2-2/p=\gamma_{k+3}$.
	\end{lemma}
	The fact that $\gamma$ is in $T_2$ is a direct consequence of the fact that $\tau$ is in $T_2$. The proof of \eqref{eq:appendixfinalpgeq36} is a long and technical calculation and we postpone the proof to the appendix. We use Lemma \ref{lem:postponeactint} and Lemma \ref{lem:pgeq36spes} to prove Lemma \ref{lem:postponediviationgeneralized}. 
	
	\begin{proof}[Proof of Lemma \ref{lem:postponediviationgeneralized}]
		It is clear that $\gamma$ is in $T_2$ since $\tau$ is in $T_2$. Let $k_1$ be the smallest integer such that $\gamma_{k_1+1} \neq \tau_{k_1+1}$. The fact that $k_1$ exists follows from the assumption $\tau \neq \lambda$. Since $\gamma_{k_1+1} \neq \tau_{k_1+1}$ and $\gamma_{n+1}=\tau_n+1$ it follows that \[\tau_{k_1+1} < \tau_{k_1}+1=k_1+2/p.\] Then either (III)(a) or (III)(b) holds (with $j=0$). Assume first that (III)(a) holds. It then follows from Lemma \ref{lem:postponeactint} that $E_p(\tau^{r_1}) < E_p(\gamma^{r_1})$, where $r_1=\gamma_{k_1+4}$. If $\tau_{n+1}=\tau_n+1$ for all $n \geq k_1+4 $ we are done. If not let $k_2$ be the smallest integer greater than or equal to $k_1+3$ such that $\tau_{k_2+1} < \tau_{k_2}+1$.
		
		Now consider the possibility that $p>3.6$ and  $\tau_{k_1+1}<m_\xi$. We then see that $\tau_{k_1+2}=\tau_{k_1+1}+2/3$ and $\tau_{k_1+3}=\xi+1-2/p$ satisfy the conditions of Lemma \ref{lem:pgeq36spes}. In particular letting $r_1=\gamma_{k_1+3}$ it follows from Lemma \ref{lem:pgeq36spes} that $E_p(\tau^{r_1}) < E_p(\gamma^{r_1})$. If $\tau_{n+1}=\tau_n+1$ for all $n \geq k_1+3 $ we are done. If not let $k_2 \geq k_1+2$ be the smallest integer such that $\tau_{k_2+1} < \tau_{k_2}+1$. 
		
		Repeating the argument above it follows that $E_p(\tau^{r_2}) < E_p(\gamma^{r_2})$ where we have $r_2=\gamma_{k_2+4}$ if $\tau_{k_2+1}$ is of the type (III)(a) and $r_2=\gamma_{k_2+3}$ if $\tau_{k_2+1}$ is of the type (III)(b). Continuing this way it follows that $E_p(\tau) < E_p(\gamma)$.
	\end{proof}
	We are now ready to prove the main theorem. 
	\begin{proof}[Proof of Theorem \ref{thm:new:Ep2to4} for $3<p<4$]
		It suffices to consider $\delta_1=\min(2/\pi, 2/p)$ and $\delta_2=2/3$. We denote by $\lambda$ the sequence given by $\lambda_n=n-1+2/p$. It is clear that $E_p(\tau)=E_p(\lambda)$ for all sequences $\tau$ in $T(\delta_1, \delta_2)$ where $\tau_n$ is in the interval $[n-1+2/p,n]$ for all $n \geq 1$. We aim to show that \[\mathscr{E}_p(\delta_1,\delta_2)=\sup_{\tau \in T(\delta_1, \delta_2)} E_p(\tau) = E_p(\lambda).\] By Lemma \ref{lem:T3} it suffices to consider the supremum over sequences in $T_2$. Let $\tau$ be a sequence in $T_2$ such that $\tau \neq \lambda$. To prove the claim it suffices to show that $E_p(\tau) < E_p(\lambda)$. Since $\tau \neq \lambda $ there exists a smallest integer $k$ such that $\tau_{k+1} < k+2/p$. We define the sequence $\gamma^1$ as in Lemma \ref{lem:postponediviationgeneralized}. That is $\gamma_n^1=\tau_{n-1}+1$ for all $n >1$ and $\gamma_1^1=1-2/p$ (not to be confused with the definition of a vertically truncated sequence.). Then by Lemma \ref{lem:postponediviationgeneralized} it follows that $E_p(\tau) < E_p(\gamma^1)$. Similarly we can define $\gamma^{2}$ by $\gamma_n^2=\gamma_{n-1}^1+1$ for $n>1$. Then again by Lemma \ref{lem:postponediviationgeneralized} it follows that $E_p(\gamma^1) < E_p(\gamma^2)$. By continuing this way we see that \[E_p(\tau)< E_p(\gamma^1) < E_p(\gamma^{2}) < \dots < E_p(\gamma^l).\]
		We observe that the sequence $\gamma^l$ satisfies $\gamma_n^l=\lambda_n$ for $n<k+l$. Finally as $l$ tends to infinity we see that $E_p(\tau) < E_p(\lambda)$.
	\end{proof}
	
	\section{Proof of Theorem \ref{thm:new:E_p4to5}}\label{sec:4to5}
	In this section we prove Theorem \ref{thm:new:E_p4to5}. As stated in Section \ref{sec:3to4} the case $p=4$ is trivial so we may assume $p>4$. We begin by showing the following related result, where we assume that $\delta_1 \geq 1-2/p \geq 1/2$. 
	\begin{theorem}\label{thm:E_p4to5weak}
		Fix $4 \leq p \leq 6$ and let $1-2/p \leq \delta_1 \leq 1/2+3/p$ and $1-2/p \leq \delta_2 \leq 1$. Then
		\begin{align*}
			\mathscr{E}_p (\delta_1, \delta_2) = \int_{0}^{2/p}\frac{\sin^2\frac{p}{2}\pi x}{\pi^2 x^2}\, dx &+ \sum_{n=0}^{\infty} \int_{n+4/p}^{n+1/2+3/p} \frac{\sin^2\frac{p}{2}\pi(x-n)}{\pi^2x^2}\dx \\ &+ \sum_{n=0}^{\infty}\int_{n+1/2+3/p}^{n+1+2/p} \frac{\sin^2\frac{p}{2}\pi(x-n-1)}{\pi^2x^2} \dx.
		\end{align*}
		For $p>4$ the supremum is attained if and only if $\tau$ is the sequence in $T(\delta_1, \delta_2)$ given by $\tau_n=n-1/2+3/p$ for all $n \geq 1$.
	\end{theorem}
	Note that Theorem \ref{thm:E_p4to5weak} does \emph{not} give an upper bound for the constant $\mathscr{C}_p$ since $\delta_1 \geq 1-2/p > 1/2$ for all $p>4$ and $\delta_2 \geq 1-2/p > 3/5$ if $p>5$. However, if one were to improve the results from \cite{Brevig} on the separation of the zeroes of the extremal function $\varphi_p$ sufficiently, then Theorem \ref{thm:E_p4to5weak} would lead to an upper bound for $\mathscr{C}_p$.
	
	Our aim is to combine Theorem \ref{thm:E_p4to5weak} with Lemma \ref{lem:techinq4to5} to obtain the upper bound $\mathscr{C}_p \leq 2 \mathscr{E}_p(1,1)$ for $4\leq p\leq 5$ as given in Theorem \ref{thm:new:E_p4to5}. The reason we do not give an upper bound for $\mathscr{C}_p$ for $5 < p \leq 6$ is that we have verified that Theorem \ref{thm:new:E_p4to5} with $\delta_1=1/2$ does not hold for $p<6$ sufficiently close to $6$. Additionally, while it seems likely that one can weaken the assumption $\delta_2 \geq 1-2/p$ to $\delta_2 \geq 3/5$ in this range this would yield for another (more complicated) proof than the one we provide.
	
	We start by observing that $E_p(\tau)$ is locally maximized if $\tau_{n+1}$ is the midpoint $m_\xi=\xi+1/2-1/p$ between an interval $(\xi,\xi+2/p)$ at level $n$ and the interval $(\xi+1-4/p,\xi+1-2/p)$ at level $n+1$. Let $(\xi, \xi+2/p)$ be an interval at level $n$. Similarly as in \eqref{eq:defS} for $\xi \leq \tau_{n+1} \leq \xi+4/p$ and $4 \leq p \leq 6$ we define 
	\begin{align*}
		S(\tau_{n+1}) &= \int_{\xi}^{\xi+2/p} \chi_{(\xi, \tau_{n+1})}(x) \frac{\sin^2\frac{p}{2}\pi(x-n)}{\pi^2 x^2} \, dx\\ &+ \int_{\xi+1-4/p}^{\xi+1-2/p}\chi_{( \tau_{n+1},\xi+1-2/p)}(x)\frac{\sin^2\frac{p}{2}\pi(x-n-1)}{\pi^2 x^2} \, dx.
	\end{align*} 
	See Figure \ref{fig:4to5choosemid} for an illustration of $S(\tau_{n+1})$ for $p>4$, as well as the following lemma. 
	
	\begin{lemma}\label{lem:4to5localmid} 
		Fix $4\leq p \leq 6$. Let $(\xi, \xi+2/p)$ be an interval at level $n$ and let $\xi \leq \tau_{n+1} \leq \xi+4/p$. Then $S(\tau_{n+1})$ is maximized by $\tau_{n+1}=m_\xi$.
	\end{lemma}
	
	\begin{proof}
		This follows from the pointwise estimates 
		\[\sin^2\frac{p}{2}\pi(x-n) > \sin^2\frac{p}{2}\pi(x-n-1),\] for all $x$ in $(\xi+1-4/p, m_\xi)$ and
		\[\sin^2\frac{p}{2}\pi(x-n) < \sin^2\frac{p}{2}\pi(x-n-1),\] for all $x$ in $(m_\xi, \xi+2/p)$.
	\end{proof}
	We note that Lemma \ref{lem:4to5localmid} differs significantly from Lemma \ref{lem:maximizeS}. This can be explained by the fact that the intervals at level $n$ and level $n+1$ intersect differently for $4<p<6$ than for $2<p<4$. We now move forward to show that to increase $E_p(\tau)$ the value $\tau_n$ should not be too large.
	\begin{figure}
		\centering
		\begin{tikzpicture}[scale=1.5]
			\def\p{5}
			\def\s{10}
			\begin{axis}[
				axis equal image,
				axis lines = none,
				trig format plots=rad]
				
				\node at (axis cs: -1,0) {$\scriptstyle n$};
				\node at (axis cs: -1,-3) {$\scriptstyle n+1$};

				\addplot[thin, name path=b1] coordinates {(0,0) (4,0)};
				\addplot[domain=0:4, samples=100, color=black!75, thin, name path=t1] ({x},{sin(pi*x*\p/(2*\s))*sin(pi*x*\p/(2*\s))});
				
				\addplot[thin, name path=b11test] coordinates {(8,0) (11,0)};
				\addplot[domain=8:11, samples=100, color=black!75, thin, name path=t11test] ({x},{sin(pi*x*\p/(2*\s))*sin(pi*x*\p/(2*\s))});
				
				\addplot[thin, name path=b11] coordinates {(8,0) (12,0)};
				\addplot[domain=8:12, samples=100, color=black!75, thin, name path=t11] ({x},{sin(pi*x*\p/(2*\s))*sin(pi*x*\p/(2*\s))});
				\addplot[red!50] fill between [of=b11test and t11test];
				
				\addplot[thin] coordinates {(16,0) (20,0)};
				\addplot[domain=16:20, samples=100, color=black!75, thin] ({x},{sin(pi*x*\p/(2*\s))*sin(pi*x*\p/(2*\s))});

				\addplot[thin] coordinates {(2,-3) (6,-3)};
				\addplot[domain=2:6, samples=100, color=black!75, thin] ({x},{sin(pi*(x-\s)*\p/(2*\s))*sin(pi*(x-\s)*\p/(2*\s))-3});
				
				\addplot[thin, name path=b2] coordinates {(11,-3) (14,-3)};
				\addplot[domain=11:14, samples=100, color=black!75, thin, name path=t2] ({x},{sin(pi*(x-\s)*\p/(2*\s))*sin(pi*(x-\s)*\p/(2*\s))-3});
				\addplot[thin] coordinates {(10,-3) (14,-3)};
				\addplot[domain=10:14, samples=100, color=black!75, thin] ({x},{sin(pi*(x-\s)*\p/(2*\s))*sin(pi*(x-\s)*\p/(2*\s))-3});
				\addplot[red!50] fill between [of=b2 and t2];
				
				\addplot[thin, name path=b22] coordinates {(18,-3) (21,-3)};
				\addplot[domain=18:21, samples=100, color=black!75, thin, name path=t22] ({x},{sin(pi*(x-\s)*\p/(2*\s))*sin(pi*(x-\s)*\p/(2*\s))-3});
				\addplot[thin] coordinates {(18,-3) (22,-3)};
				\addplot[domain=18:22, samples=100, color=black!75, thin] ({x},{sin(pi*(x-\s)*\p/(2*\s))*sin(pi*(x-\s)*\p/(2*\s))-3});
				\addplot[thin, color=black!0] coordinates {(0,-3.6) (22,-3.6)};
				\addplot[thin,color=blue,->] coordinates {(11, 0) (11,-3)};
				
				\addplot[only marks,mark=|,color=black,mark size=2pt] coordinates {(8,-3) (11,-3) (16,-3)};
				\node at (axis cs: 8,-3.6) {$\scriptstyle \xi$};
				\node at (axis cs: 11,-3.6) {$\scriptstyle m_\xi$};
				\node at (axis cs: 16,-3.6) {$\scriptstyle \xi+4/p$};
			\end{axis}
		\end{tikzpicture}
		\caption{An illustration of Lemma \ref{lem:4to5localmid} with $p=5$. The \CR{shaded} area represents $S(\tau_{n+1})$ without considering the integrand factor $1/(\pi x)^2$. We see that $S(\tau_{n+1})$ is maximized for $\CB{\tau_{n+1}}=m_\xi=\xi+1/2-1/p$. }
		\label{fig:4to5choosemid}
	\end{figure}
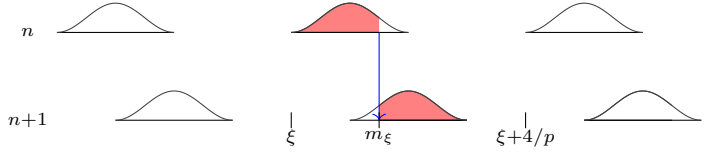
	\begin{lemma}\label{lem:goleft4to6}
		Fix $4 < p \leq 6$. Let $0< \delta_1 \leq 1/2+3/p$ and $0<\delta_2 \leq 1$. Let $(\xi, \xi +2/p)$ be an interval at level $n$. Assume $\tau$ is a sequence in $T(\delta_1, \delta_2)$ such that $\tau_{n} \leq \xi$, $\tau_{n+1}=m_\xi+4/p$ and $\tau_{n+2} \geq \xi +1+ 2/p$. Further let $\gamma$ be the sequence such that $\gamma_{n+1}=m_\xi$ and $\gamma_k=\tau_k$ for all $k \neq n+1$.  Then $E_p(\tau) < E_p(\gamma)$.
	\end{lemma}
	
	\begin{proof}
		To prove the lemma it is enough to show that
		\begin{equation}\label{eq:4to5left}
			\begin{split}
				&\int_{\xi}^{\xi+2/p} \frac{\sin^2 \frac{p}{2}\pi(x-n)}{\pi^2x^2} \, dx +
				\int_{\xi+4/p}^{m_\xi+4/p} \frac{\sin^2 \frac{p}{2}\pi(x-n)}{\pi^2x^2} \, dx  \\ &+
				\int_{m_\xi+4/p}^{\xi+1+2/p} \frac{\sin^2 \frac{p}{2}\pi(x-n-1)}{\pi^2x^2} \, dx  \\
				& \leq 	\int_{\xi}^{m_\xi} \frac{\sin^2 \frac{p}{2}\pi(x-n)}{\pi^2x^2} \, dx + 	\int_{m_\xi}^{\xi+1-2/p} \frac{\sin^2 \frac{p}{2}\pi(x-n-1)}{\pi^2x^2} \, dx\\ &+\int_{\xi+1}^{\xi+1+2/p} \frac{\sin^2 \frac{p}{2}\pi(x-n-1)}{\pi^2x^2} \, dx. 
			\end{split}
		\end{equation}
		See Figure \ref{fig:goleft4to6part1} and Figure \ref{fig:goleft4to6part2} for an illustration of this inequality. By eliminating the equal parts on each side of the inequality \eqref{eq:4to5left} and observing that
		\[\int_{\xi+4/p}^{\xi+1} \frac{\sin^2 \frac{p}{2} \pi (x-n)}{\pi^2 x^2} \, dx \leq \int_{\xi+2/p}^{\xi+1-2/p} \frac{\sin^2 \frac{p}{2}\pi(x-n-1)}{\pi^2 x^2} \, dx,\] it follows that \eqref{eq:4to5left} holds if
		\begin{equation}\label{eq:4to5left2}
			\begin{split}
				&\int_{m_\xi}^{\xi+2/p} \frac{\sin^2\frac{p}{2}\pi(x-n-1)}{\pi^2x^2} \, dx - 
				\int_{m_\xi}^{\xi+2/p} \frac{\sin^2\frac{p}{2}\pi(x-n)}{\pi^2x^2} \, dx \\
				&\geq 
				\int_{\xi+1}^{m_\xi+4/p} \frac{\sin^2\frac{p}{2}\pi(x-n)}{\pi^2 x^2} \, dx - 
				\int_{\xi+1}^{m_\xi+4/p} \frac{\sin^2\frac{p}{2}\pi(x-n-1)}{\pi^2 x^2 } \, dx.
			\end{split}
		\end{equation}
		By substitutions we see that \eqref{eq:4to5left2} is equivalent to 
		\begin{align*}\label{eq:4to5left3}
			\int_{0}^{3/p-1/2} \frac{\sin^2\frac{p}{2}\pi(x+1)-\sin^2\frac{p}{2}\pi x}{\pi^2(x-\xi-2/p)^2}\, dx \geq \int_{0}^{3/p-1/2} \frac{\sin^2\frac{p}{2}\pi(x+1)-\sin^2\frac{p}{2}\pi x}{\pi^2(x+\xi+1)^2} \, dx.
		\end{align*}
		This indeed follows from the pointwise estimates
		\[\sin^2\frac{p}{2} \pi (x+1) \geq \sin^2\frac{p}{2} \pi x,\]
		and 
		\[(x-\xi-2/p)^2 \leq (x+\xi+1)^2\] for all $x$ in the interval $(0, 3/p-1/2)$, which concludes the proof. 
	\end{proof}
	
	\begin{figure}
		\centering
		\begin{tikzpicture}[scale=1.5]
			\def\p{5}
			\def\s{10}
			\begin{axis}[
				axis equal image,
				axis lines = none,
				trig format plots=rad]
				
				\node at (axis cs: -1,0) {$\scriptstyle n$};
				\node at (axis cs: -1,-3) {$\scriptstyle n+1$};
				
				\addplot[thin, name path=b1] coordinates {(0,0) (4,0)};
				\addplot[domain=0:4, samples=100, color=black!75, thin, name path=t1] ({x},{sin(pi*x*\p/(2*\s))*sin(pi*x*\p/(2*\s))});
				\addplot[red!50] fill between [of=b1 and t1];
				
				\addplot[thin, name path=b11test] coordinates {(8,0) (11,0)};
				\addplot[domain=8:11, samples=100, color=black!75, thin, name path=t11test] ({x},{sin(pi*x*\p/(2*\s))*sin(pi*x*\p/(2*\s))});
				\addplot[red!50] fill between [of=b11test and t11test];
				
				\addplot[thin] coordinates {(11,0) (12,0)};
				\addplot[domain=11:12, samples=100, color=black!75, thin] ({x},{sin(pi*x*\p/(2*\s))*sin(pi*x*\p/(2*\s))});

				\addplot[thin] coordinates {(2,-3) (6,-3)};
				\addplot[domain=2:6, samples=100, color=black!75, thin] ({x},{sin(pi*(x-\s)*\p/(2*\s))*sin(pi*(x-\s)*\p/(2*\s))-3});
				
				\addplot[thin, name path=b2] coordinates {(11,-3) (14,-3)};
				\addplot[domain=11:14, samples=100, color=black!75, thin, name path=t2] ({x},{sin(pi*(x-\s)*\p/(2*\s))*sin(pi*(x-\s)*\p/(2*\s))-3});
				\addplot[thin] coordinates {(10,-3) (11,-3)};
				\addplot[domain=10:11, samples=100, color=black!75, thin] ({x},{sin(pi*(x-\s)*\p/(2*\s))*sin(pi*(x-\s)*\p/(2*\s))-3});
				\addplot[red!50] fill between [of=b2 and t2];
				
				\addplot[thin, color=black!0] coordinates {(0,-3.6) (10,-3.6)};
				\addplot[thin,color=blue,->] coordinates {(11,0+0.0) (11,-3-0.0)};
				
				\addplot[only marks,mark=|,color=black,mark size=2pt] coordinates {(0,-3) (3,-3) (11,-3)};
				\node at (axis cs: 0,-3.6) {$\scriptstyle \xi$};
				\node at (axis cs: 3,-3.6) {$\scriptstyle m_\xi$};
				\node at (axis cs: 11,-3.6) {$\scriptstyle m_\xi+4/p$};
			\end{axis}
		\end{tikzpicture}
		\caption{ Together with Figure \ref{fig:goleft4to6part2} this gives an illustration of Lemma \ref{lem:goleft4to6}. Here we see the value $\CB{\tau_{n+1}}=m_\xi+4/p$. The \CR{shaded} area represents the left hand side of \eqref{eq:4to5left} without considering the integrand factor $1/(\pi x)^2$.}
		\label{fig:goleft4to6part1}
		
		\centering
		\begin{tikzpicture}[scale=1.5]
			\def\p{5}
			\def\s{10}
			\begin{axis}[
				axis equal image,
				axis lines = none,
				trig format plots=rad]
				
				\node at (axis cs: -1,0) {$\scriptstyle n$};
				\node at (axis cs: -1,-3) {$\scriptstyle n+1$};
				
				\addplot[thin, name path=b1] coordinates {(0,0) (3,0)};
				\addplot[domain=0:3, samples=100, color=black!75, thin, name path=t1] ({x},{sin(pi*x*\p/(2*\s))*sin(pi*x*\p/(2*\s))});
				\addplot[red!50] fill between [of=b1 and t1];
				
				\addplot[thin] coordinates {(3,0) (4,0)};
				\addplot[domain=3:4, samples=100, color=black!75, thin] ({x},{sin(pi*x*\p/(2*\s))*sin(pi*x*\p/(2*\s))});
				
				\addplot[thin] coordinates {(8,0) (12,0)};
				\addplot[domain=8:12, samples=100, color=black!75, thin, ] ({x},{sin(pi*x*\p/(2*\s))*sin(pi*x*\p/(2*\s))});

				\addplot[thin] coordinates {(2,-3) (3,-3)};
				\addplot[domain=2:3, samples=100, color=black!75, thin] ({x},{sin(pi*(x-\s)*\p/(2*\s))*sin(pi*(x-\s)*\p/(2*\s))-3});
				
				\addplot[thin, name path=b4] coordinates {(3,-3) (6,-3)};
				\addplot[domain=3:6, samples=100, color=black!75, thin, name path=t4] ({x},{sin(pi*(x-\s)*\p/(2*\s))*sin(pi*(x-\s)*\p/(2*\s))-3});
				\addplot[red!50] fill between [of=b4 and t4];
				
				\addplot[thin, name path=b2] coordinates {(10,-3) (14,-3)};
				\addplot[domain=10:14, samples=100, color=black!75, thin,name path=t2] ({x},{sin(pi*(x-\s)*\p/(2*\s))*sin(pi*(x-\s)*\p/(2*\s))-3});
				\addplot[red!50] fill between [of=b2 and t2];
				\addplot[thin, color=black!0] coordinates {(0,-3.6) (10,-3.6)};
				\addplot[thin,color=blue,->] coordinates {(3,0+0.0) (3,-3-0.0)};
				\addplot[only marks,mark=|,color=black,mark size=2pt] coordinates {(0,-3) (3,-3) (11, -3)};
				\node at (axis cs: 3,-3.6) {$\scriptstyle m_\xi$};
				\node at (axis cs: 0,-3.6) {$\scriptstyle \xi$};
				\node at (axis cs: 11,-3.6) {$\scriptstyle m_\xi+4/p$};
			\end{axis}
		\end{tikzpicture}
		\caption{Together with Figure \ref{fig:goleft4to6part1} this gives an illustration of Lemma \ref{lem:goleft4to6}. Here we see the value $\CB{\tau_{n+1}}=m_\xi$. The \CR{shaded} area represents the right hand side of \eqref{eq:4to5left} without considering the integrand factor $1/(\pi x)^2$. }
		\label{fig:goleft4to6part2}
	\end{figure}
	We can now use Lemma \ref{lem:goleft4to6} to show the following result similar to Lemma \ref{lem:tau_n_leq_n}. The lemma tells us that when determining the supremum $\mathscr{E}_p(\delta_1, \delta_2)$ we need not consider sequences with large elements, and may in fact assume $\tau_n \leq n-1/2+3/p$ for all $n \geq 1$.
	\begin{lemma}\label{lem:4to6nottolagrezeros}
		Assume $4 \leq p \leq 6$.
		Let $0 < \delta_1 \leq 1/2+3/p$ and $0 < \delta_2 \leq 1$. Let $\tau$ be a sequence in $T(\delta_1, \delta_2)$. Define the sequence $\gamma$ by $\gamma_n=\min(n-1/2+3/p, \tau_n)$. Then $\gamma$ is a sequence in $T(\delta_1, \delta_2)$ and $E_p(\tau) \leq E_p(\gamma)$. 
	\end{lemma}
	
	\begin{proof}
		We see that $\gamma_1 = \min(\tau_1, 1/2+3/p) \geq \delta_1$. Combining the fact that $\tau_{n+1} - \tau_n \geq \delta_2$ and that $n+1-n=1\geq \delta_2$ it follows that $\gamma_{n+1}- \gamma_n \geq \delta_2$, and hence $\gamma$ is in the set of sequences $T(\delta_1, \delta_2)$. We now show that $E_p(\tau) \leq E_p(\gamma)$.  If $\tau_n \leq n-1/2+3/p$ for all $n \geq 0$ we are done. Otherwise, let $m$ be the smallest integer such that $\tau_m >m-1/2+3/p$ and let $l\geq 0$ be the largest integer such that $\tau_m \geq m-1+(4/p)l$. By Lemma \ref{lem:4to5localmid} replacing $\tau_m$ with $m+(4/p)l-1/2+3/p$ will increase the value of $E_p(\tau)$. Further, by Lemma \ref{lem:goleft4to6} replacing $\tau_m$ with $m-1/2+3/p$ will increase the value of $E_p(\tau)$ further, and this sequence is again in $T(\delta_1, \delta_2)$. We conclude the proof by repeating the above argument for increasing values of $m$ such that $\tau_m>m-1/2+3/p$. 
	\end{proof}	
	As in the case $2 \leq p \leq 4$ we will introduce the notion of rays. However the definition will be somewhat different when $4 \leq p \leq 6$. Let $\varrho=1-4/p$. For $k, j \geq 0$ we define 
	$I_{k, j}=(k-1+4/p+j\varrho, k-1+4/p+j\varrho+2/p)$. We call the collection of intervals $I_k=(I_{k, j})_{j \geq 0}$ the $k$th ray. Note that we omit the interval $I_{0,0}$ to ensure that $(0,2/p)$ is the first interval in the first ray. Further we let $J_n=(\tau_n, \tau_{n+1})$. We then let 
	\[A_p(\tau; k)= \sum_{j=0}^{\infty} \int_{I_{k,j}} \chi_{J_k+j}(x) \frac{\sin^2 \frac{p}{2}\pi(x-k-2j)}{\pi^2x^2} dx.\] 
	We can think of $A_p(\tau; k)$ as the contribution to $E_p(\tau)$ from the $k$th ray. See Figure \ref{fig:p5withray} for an illustration.
	
	Similarly as we saw in Lemma \ref{lem:A(n)=Dp} we show that to determine $E_p(\tau)$ we need only consider the intervals in the $k$th ray for all $k \geq 0$. More precisely:
	\begin{figure}
		\centering
		\begin{tikzpicture}[scale=1.5]
			\def\p{5}
			\def\s{10}
			\begin{axis}[
				axis equal image,
				axis lines = none,
				trig format plots=rad]
				
				\node at (axis cs: -1,0) {$\scriptstyle 0$};
				\node at (axis cs: -1,-3) {$\scriptstyle 1$};
				\node at (axis cs: -1,-6) {$\scriptstyle 2$};
				\node at (axis cs: -1,-9) {$\scriptstyle 3$};
				\node at (axis cs: -1,-12) {$\scriptstyle 4$};					
				\node at (axis cs: -1,15) {$\scriptstyle 5$};
				\node at (axis cs: -1,-18) {$\scriptstyle 6$};

				\addplot[thin, name path=b1] coordinates {(0,0) (4,0)};
				\addplot[domain=0:4, samples=100, color=black!75, thin, name path=t1] ({x},{sin(pi*x*\p/(2*\s))*sin(pi*x*\p/(2*\s))});
				\addplot[red!50] fill between [of=b1 and t1];
				
				\addplot[thin, name path=b11test] coordinates {(8,0) (11,0)};
				\addplot[domain=8:11, samples=100, color=black!75, thin, name path=t11test] ({x},{sin(pi*x*\p/(2*\s))*sin(pi*x*\p/(2*\s))});
				
				\addplot[thin, name path=b11] coordinates {(11,0) (12,0)};
				\addplot[domain=11:12, samples=100, color=black!75, thin, name path=t11] ({x},{sin(pi*x*\p/(2*\s))*sin(pi*x*\p/(2*\s))});
				\addplot[red!50] fill between [of=b11test and t11test];
				
				\addplot[thin] coordinates {(16,0) (20,0)};
				\addplot[domain=16:20, samples=100, color=black!75, thin] ({x},{sin(pi*x*\p/(2*\s))*sin(pi*x*\p/(2*\s))});
				
				\addplot[thin] coordinates {(24,0) (28,0)};
				\addplot[domain=24:28, samples=100, color=black!75, thin] ({x},{sin(pi*x*\p/(2*\s))*sin(pi*x*\p/(2*\s))});

				\addplot[thin] coordinates {(2,-3) (6,-3)};
				\addplot[domain=2:6, samples=100, color=black!75, thin] ({x},{sin(pi*(x-\s)*\p/(2*\s))*sin(pi*(x-\s)*\p/(2*\s))-3});
				
				\addplot[thin, name path=b2] coordinates {(11,-3) (14,-3)};
				\addplot[domain=11:14, samples=100, color=black!75, thin, name path=t2] ({x},{sin(pi*(x-\s)*\p/(2*\s))*sin(pi*(x-\s)*\p/(2*\s))-3});
				\addplot[thin] coordinates {(10,-3) (11,-3)};
				\addplot[domain=10:11, samples=100, color=black!75, thin] ({x},{sin(pi*(x-\s)*\p/(2*\s))*sin(pi*(x-\s)*\p/(2*\s))-3});
				\addplot[red!50] fill between [of=b2 and t2];
				
				\addplot[thin, name path=b22] coordinates {(18,-3) (21,-3)};
				\addplot[domain=18:21, samples=100, color=black!75, thin, name path=t22] ({x},{sin(pi*(x-\s)*\p/(2*\s))*sin(pi*(x-\s)*\p/(2*\s))-3});
				\addplot[thin] coordinates {(21,-3) (22,-3)};
				\addplot[domain=21:22, samples=100, color=black!75, thin] ({x},{sin(pi*(x-\s)*\p/(2*\s))*sin(pi*(x-\s)*\p/(2*\s))-3});
				\addplot[red!50] fill between [of=b22 and t22];
				
				\addplot[thin] coordinates {(26,-3) (30,-3)};
				\addplot[domain=26:30, samples=100, color=black!75, thin] ({x},{sin(pi*(x-\s)*\p/(2*\s))*sin(pi*(x-\s)*\p/(2*\s))-3});

				\addplot[thin] coordinates {(4,-6) (8,-6)};
				\addplot[domain=4:8, samples=100, color=black!75, thin] ({x},{sin(pi*(x-2*\s)*\p/(2*\s))*sin(pi*(x-2*\s)*\p/(2*\s))-6});
				
				\addplot[thin] coordinates {(12,-6) (16,-6)};
				\addplot[domain=12:16, samples=100, color=black!75, thin] ({x},{sin(pi*(x-2*\s)*\p/(2*\s))*sin(pi*(x-2*\s)*\p/(2*\s))-6});
				
				\addplot[thin, name path=b3] coordinates {(21,-6) (24,-6)};
				\addplot[domain=21:24, samples=100, color=black!75, thin, name path=t3] ({x},{sin(pi*(x-2*\s)*\p/(2*\s))*sin(pi*(x-2*\s)*\p/(2*\s))-6});
				\addplot[thin] coordinates {(20,-6) (21,-6)};
				\addplot[domain=20:21, samples=100, color=black!75, thin] ({x},{sin(pi*(x-2*\s)*\p/(2*\s))*sin(pi*(x-2*\s)*\p/(2*\s))-6});
				\addplot[red!50] fill between [of=b3 and t3];
				
				\addplot[thin, name path=b33] coordinates {(28,-6) (31,-6)};
				\addplot[domain=28:31, samples=100, color=black!75, thin, name path=t33] ({x},{sin(pi*(x-2*\s)*\p/(2*\s))*sin(pi*(x-2*\s)*\p/(2*\s))-6});
				
				\addplot[thin] coordinates {(31,-6) (32,-6)};
				\addplot[domain=31:32, samples=100, color=black!75, thin] ({x},{sin(pi*(x-2*\s)*\p/(2*\s))*sin(pi*(x-2*\s)*\p/(2*\s))-6});
				\addplot[red!50] fill between [of=b33 and t33];

				\addplot[thin] coordinates {(0,-9) (2,-9)};
				\addplot[domain=0:2, samples=100, color=black!75, thin] ({x},{sin(pi*(x-3*\s)*\p/(2*\s))*sin(pi*(x-3*\s)*\p/(2*\s))-9});
				
				\addplot[thin] coordinates {(6,-9) (10,-9)};
				\addplot[domain=6:10, samples=100, color=black!75, thin] ({x},{sin(pi*(x-3*\s)*\p/(2*\s))*sin(pi*(x-3*\s)*\p/(2*\s))-9});
				
				\addplot[thin] coordinates {(14,-9) (18,-9)};
				\addplot[domain=14:18, samples=100, color=black!75, thin] ({x},{sin(pi*(x-3*\s)*\p/(2*\s))*sin(pi*(x-3*\s)*\p/(2*\s))-9});
				
				\addplot[thin] coordinates {(22,-9) (26,-9)};
				\addplot[domain=22:26, samples=100, color=black!75, thin] ({x},{sin(pi*(x-3*\s)*\p/(2*\s))*sin(pi*(x-3*\s)*\p/(2*\s))-9});
				
				\addplot[thin, name path=b4] coordinates {(31,-9) (34,-9)};
				\addplot[domain=31:34, samples=100, color=black!75, thin, name path=t4] ({x},{sin(pi*(x-3*\s)*\p/(2*\s))*sin(pi*(x-3*\s)*\p/(2*\s))-9});
				
				
				\addplot[thin] coordinates {(30,-9) (31,-9)};
				\addplot[domain=30:31, samples=100, color=black!75, thin] ({x},{sin(pi*(x-3*\s)*\p/(2*\s))*sin(pi*(x-3*\s)*\p/(2*\s))-9});
				\addplot[red!50] fill between [of=b4 and t4];

				\addplot[thin] coordinates {(0,-12) (4,-12)};
				\addplot[domain=0:4, samples=100, color=black!75, thin] ({x},{sin(pi*(x-4*\s)*\p/(2*\s))*sin(pi*(x-4*\s)*\p/(2*\s))-12});
				\addplot[thin] coordinates {(8,-12) (12,-12)};
				\addplot[domain=8:12, samples=100, color=black!75, thin] ({x},{sin(pi*(x-4*\s)*\p/(2*\s))*sin(pi*(x-4*\s)*\p/(2*\s))-12});
				\addplot[thin] coordinates {(16,-12) (20,-12)};
				\addplot[domain=16:20, samples=100, color=black!75, thin] ({x},{sin(pi*(x-4*\s)*\p/(2*\s))*sin(pi*(x-4*\s)*\p/(2*\s))-12});
				\addplot[thin] coordinates {(24,-12) (28,-12)};
				\addplot[domain=24:28, samples=100, color=black!75, thin] ({x},{sin(pi*(x-4*\s)*\p/(2*\s))*sin(pi*(x-4*\s)*\p/(2*\s))-12});
				
				\addplot[thin] coordinates {(24,-12) (28,-12)};
				\addplot[domain=24:28, samples=100, color=black!75] ({x},{sin(pi*(x-4*\s)*\p/(2*\s))*sin(pi*(x-4*\s)*\p/(2*\s))-12});
				
				\addplot[thin] coordinates {(32,-12) (36,-12)};
				\addplot[domain=32:36, samples=100, color=black!75] ({x},{sin(pi*(x-4*\s)*\p/(2*\s))*sin(pi*(x-4*\s)*\p/(2*\s))-12});
				\addplot[thin, dashed] coordinates {(0,0) (8,-12)};
				\addplot[thin, dashed] coordinates {(4,0) (12,-12)};
				
				\addplot[thin, dashed] coordinates {(8,0) (16,-12)};
				\addplot[thin, dashed] coordinates {(12,0) (20,-12)};
				
				\addplot[thin, dashed] coordinates {(18,-3) (24,-12)};
				\addplot[thin, dashed] coordinates {(22,-3) (28,-12)};
				
				\addplot[thin, dashed] coordinates {(28,-6) (32,-12)};
				\addplot[thin, dashed] coordinates {(32,-6) (36,-12)};
				
				\addplot[thin, dashed] coordinates {(11,0) (31,-6)};
				\addplot[thin,color=blue,->] coordinates {(11,0+0.0) (11,-3-0.0)};
				\addplot[thin,color=blue,->] coordinates {(21,-3+0.0) (21,-6-0.0)};
				\addplot[thin,color=blue,->] coordinates {(31,-6+0.0) (31,-9-0.0)};
			\end{axis}
		\end{tikzpicture}
		\caption{Here $p=5$. The dashed lines mark the rays as well as the sector of intervals we know from Lemma \ref{lem:4to6Epsumoverrays} it is sufficient to consider when determining $E_p(\tau)$. The \CR{shaded} area represents $\mathscr{E}_5(1/2, 3/5)=E_5(\lambda)$ without considering the integrand factor $1/(\pi x)^2$, where $\CB{\lambda}$ is given by $\lambda_n=n+1/10$.}
		\label{fig:p5withray}
	\end{figure}
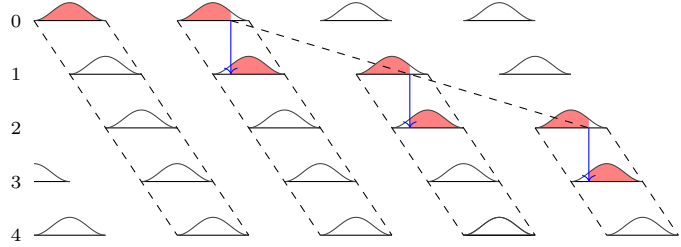
	
	\begin{lemma}\label{lem:4to6Epsumoverrays}
		Let $4 \leq p \leq 6$. Assume $1-4/p \leq \delta_1 \leq 1/2+3/p$ and $1-4/p \leq \delta_2 \leq 1$. Let $\tau$ be a sequence in $T(\delta_1, \delta_2)$ such that $\tau_n \leq n-1/2+3/p$ for all $n \geq 1$. Then
		\[E_p(\tau)= \sum_{k=0}^{\infty} A_p(\tau; k).\]
	\end{lemma}

	\begin{proof}
		We consider the collection of intervals $I=(I_{k, j})_{k, j \geq 0}$ and observe that the intervals on level $n$ in the collection $I$ are exactly those intervals $(\xi, \xi + 2/p)$ where $ n(1-4/p) \leq \xi < n +1/2 + 3/p $. Further, due to the separation conditions we have $\tau_n \geq \delta_1+ \delta_2(n-1) \geq n(1-4/p)+4/p$.  Hence it follows that 
		\[E_p(\tau)=\sum_{k=0}^{\infty}A_p(\tau; k),\] which concludes the proof.
	\end{proof}
	
	Next we see that the largest possible contribution to $E_p(\tau)$ from the $k$th ray is obtained if $\tau_k=k-1/2+3/p = m_{k-1+4/p}$.
	\begin{lemma}\label{lem:4to6maxray}
		Let $4 \leq p \leq 6$. Assume $1-2/p \leq \delta_1 \leq 1/2+3/p$ and $1-2/p \leq \delta_2 \leq 1$, and let $\tau$ be a sequence in $T(\delta_1, \delta_2)$. Then 
		\[A_p(\tau; k) \leq \int_{k-1+4/p}^{k-1/2+3/p} \frac{\sin^2 \frac{p}{2}\pi (x-k)}{\pi^2 x^2}dx + \int_{k-1/2+3/p}^{k+2/p} \frac{\sin^2 \frac{p}{2}\pi (x-k-1)}{\pi^2 x^2}dx, \]
		for all $k \geq 1$. 
	\end{lemma}
	
	\begin{proof}
		We fix $k \geq 1$. Since $\delta_1 \geq 1-2/p$ and $\delta_2 \geq 1-2/p$ it follows from the definition of $A_p(\tau; k)$ that at most three terms in the sum $A_p(\tau; k)$ can be non-zero, and these three must be subsequent terms. Applying Lemma \ref{lem:Brevig:lem6.5} and the fact that $x \mapsto 1/x$ is a decreasing function it follows that $A_p(\tau; k)$ is maximized if all terms with $j \geq 2$ are zero. Assume first that all the three first terms are non-zero. It is then clear that maximizing $A_p(\tau; k)$ requires that $k-1+4/p < \tau_k < k$ and $\tau_{k+1}=\tau_k+\delta_2 > \tau_k+1-2/p$. However we claim that 
		\begin{equation}\label{eq:nottoclosezeroes}
			\begin{split}
				\int_{k-1+4/p}^{k-1+6/p} \chi_{J_{k-1}}(x) 	\frac{\sin^2\frac{p}{2}\pi(x-k+1)}{\pi^2x^2}dx &+ \int_{k}^{k+2/p}\chi_{J_k}(x) \frac{\sin^2\frac{p}{2}\pi(x-k)}{\pi^2x^2}dx \\ + \int_{k+1-4/p}^{k+1-2/p} \chi_{J_{k+1}}(x) \frac{\sin^2\frac{p}{2}\pi(x-k-1)}{\pi^2x^2}&dx \\
				\leq \int_{k-1+4/p}^{k-1/2+3/p} 	\frac{\sin^2\frac{p}{2}\pi(x-k+1)}{\pi^2x^2}dx &+\int_{k-1/2+3/p}^{k+2/p} \frac{\sin^2\frac{p}{2}\pi(x-k)}{\pi^2x^2}dx.
			\end{split}
		\end{equation}
		See Figure \ref{fig:4to5notoclose} and Figure \ref{fig:4to5notoclosepart2} for an illustration of this inequality. Assume that \eqref{eq:nottoclosezeroes} holds. Combining \eqref{eq:nottoclosezeroes} with Lemma \ref{lem:Brevig:lem6.5} and the fact that $x \mapsto 1/x$ is a decreasing function it follows that $A_p(\tau; k)$ is maximized if at most the first two terms in $A_p(\tau; k)$ are non-zero. 
		It then follows from Lemma \ref{lem:4to5localmid} that $A_p(\tau; k)$ is maximized if $\tau_k=k+1/2+3/p$. It thus remains to verify \eqref{eq:nottoclosezeroes}. By assumption $k-1+4/p < \tau_k < k$ and it suffices to consider $\tau_{k+1}=\tau_k+1-2/p$.
		Thus the inequality \eqref{eq:nottoclosezeroes} holds if
		\begin{equation}\label{eq:10116}
			\begin{split}
				\int_{k}^{k-1/2+3/p} \frac{\sin^2\frac{p}{2}\pi(x-k)}{\pi^2x^2}dx &+ \int_{\tau_k+1-2/p}^{k+1-2/p} \frac{\sin^2\frac{p}{2}\pi(x-k-1)}{\pi^2x^2}dx\\  &\leq \int_{\tau_k}^{k-1/2+3/p}\frac{\sin^2\frac{p}{2}\pi(x-k+1)}{\pi^2x^2}dx.
			\end{split}
		\end{equation}
		Using the pointwise estimate 	
		\[\sin^2\frac{p}{2}\pi(x-k+1) > \sin^2\frac{p}{2}\pi(x-k)\] for all $x$ in $(k, k-1/2+3/p)$ we get
		\begin{equation}\label{eq:10126}
			\int_{k}^{k-1/2+3/p} \frac{\sin^2\frac{p}{2}\pi(x-k)}{\pi^2x^2}dx  \leq \int_{k}^{k-1/2+3/p}\frac{\sin^2\frac{p}{2}\pi(x-k+1)}{\pi^2x^2}dx. 
		\end{equation}
		Moreover, we see that
		\[\int_{\tau_k+1-2/p}^{k+1-2/p} \sin^2\frac{p}{2}\pi(x-k-1)dx  = \int_{\tau_k}^{k} \sin^2\frac{p}{2}\pi(x-k+1)dx.\]
		Then, since $x \mapsto 1/x$ is a decreasing function it follows that
		\begin{equation}\label{eq:1017}
			\int_{\tau_k+1-2/p}^{k+1-2/p} \frac{\sin^2\frac{p}{2}\pi(x-k-1)}{\pi^2x^2}dx  \leq \int_{\tau_k}^{k}\frac{\sin^2\frac{p}{2}\pi(x-k+1)}{\pi^2x^2}dx.
		\end{equation}
		Thus, combining \eqref{eq:10126} and \eqref{eq:1017} it follows that \eqref{eq:10116} holds which concludes the proof. 
	\end{proof}
	
	\begin{figure}
		\centering
		\begin{tikzpicture}[scale=1.5]
			\def\p{5}
			\def\s{10}
			\begin{axis}[
				axis equal image,
				axis lines = none,
				trig format plots=rad]
				
				\node at (axis cs: -1,0) {$\scriptstyle k-1$};
				\node at (axis cs: -1,-3) {$\scriptstyle k$};
				\node at (axis cs: -1,-6) {$\scriptstyle k+1$};
				
				\addplot[thin] coordinates {(0,0) (4,0)};
				\addplot[domain=0:4, samples=100, color=black!75, thin] ({x},{sin(pi*x*\p/(2*\s))*sin(pi*x*\p/(2*\s))});
				
				\addplot[thin, name path=t0] coordinates {(8,0) (9,0)};
				\addplot[domain=8:9, samples=100, color=black!75, thin, name path=b0] ({x},{sin(pi*x*\p/(2*\s))*sin(pi*x*\p/(2*\s))});
				\addplot[red!50] fill between [of=b0 and t0];
				
				\addplot[thin] coordinates {(9,0) (12,0)};
				\addplot[domain=9:12, samples=100, color=black!75, thin] ({x},{sin(pi*x*\p/(2*\s))*sin(pi*x*\p/(2*\s))});
				
				\addplot[thin] coordinates {(16,0) (20,0)};
				\addplot[domain=16:20, samples=100, color=black!75, thin] ({x},{sin(pi*x*\p/(2*\s))*sin(pi*x*\p/(2*\s))});
				
				\addplot[thin] coordinates {(2,-3) (6,-3)};
				\addplot[domain=2:6, samples=100, color=black!75, thin] ({x},{sin(pi*(x-\s)*\p/(2*\s))*sin(pi*(x-\s)*\p/(2*\s))-3});
				
				
				\addplot[thin, name path=t1] coordinates {(10,-3) (14,-3)};
				\addplot[domain=10:14, samples=100, color=black!75, thin, name path=b1] ({x},{sin(pi*(x-\s)*\p/(2*\s))*sin(pi*(x-\s)*\p/(2*\s))-3});
				\addplot[red!50] fill between [of=b1 and t1];
				
				\addplot[thin] coordinates {(18,-3) (22,-3)};
				\addplot[domain=18:22, samples=100, color=black!75, thin, black] ({x},{sin(pi*(x-\s)*\p/(2*\s))*sin(pi*(x-\s)*\p/(2*\s))-3});
				
				\addplot[thin] coordinates {(4,-6) (8,-6)};
				\addplot[domain=4:8, samples=100, color=black!75, thin] ({x},{sin(pi*(x-2*\s)*\p/(2*\s))*sin(pi*(x-2*\s)*\p/(2*\s))-6});
				
				\addplot[thin] coordinates {(12,-6) (16,-6)};
				\addplot[domain=12:16, samples=100, color=black!75, thin] ({x},{sin(pi*(x-2*\s)*\p/(2*\s))*sin(pi*(x-2*\s)*\p/(2*\s))-6});
				
				\addplot[thin, name path=t2] coordinates {(15,-6) (16,-6)};
				\addplot[domain=15:16, samples=100, color=black!75, thin, name path=b2] ({x},{sin(pi*(x-2*\s)*\p/(2*\s))*sin(pi*(x-2*\s)*\p/(2*\s))-6});
				\addplot[red!50] fill between [of=b2 and t2];
				
				\addplot[thin] coordinates {(20,-6) (24,-6)};
				\addplot[domain=20:24, samples=100, color=black!75, thin] ({x},{sin(pi*(x-2*\s)*\p/(2*\s))*sin(pi*(x-2*\s)*\p/(2*\s))-6});
				
				\addplot[thin,color=blue,->] coordinates {(9,0+0.0) (9,-3-0.0)};
				\addplot[thin,color=blue,->] coordinates {(15,-3+0.0) (15,-6-0.0)};
				\addplot[thin,dashed,color=gray] coordinates {(8,0) (8,-6)};
				\addplot[thin,dashed,color=gray] coordinates {(10,0) (10,-6)};
			\end{axis}
		\end{tikzpicture}
		\caption{The \CR{shaded} area represents the left hand side of the inequality \eqref{eq:nottoclosezeroes} in Lemma \ref{lem:4to6maxray}, without considering the integrand factor $1/(\pi x)^2$. The values $\CB{\tau_k}$ and $\CB{\tau_{k+1}}$ are indicated by arrows.  An illustration of the right hand side is given in Figure \ref{fig:4to5notoclosepart2}. The dashed lines represent the range $(\xi-1+4/p, \xi+2/p)$ of possible values for $\tau_k$.}
		\label{fig:4to5notoclose}
		
		\centering
		\begin{tikzpicture}[scale=1.5]
			\def\p{5}
			\def\s{10}
			\begin{axis}[
				axis equal image,
				axis lines = none,
				trig format plots=rad]
				
				\node at (axis cs: -1,0) {$\scriptstyle k-1$};
				\node at (axis cs: -1,-3) {$\scriptstyle k$};
				\node at (axis cs: -1,-6) {$\scriptstyle k+1$};
				
				\addplot[thin] coordinates {(0,0) (4,0)};
				\addplot[domain=0:4, samples=100, color=black!75, thin] ({x},{sin(pi*x*\p/(2*\s))*sin(pi*x*\p/(2*\s))});
				
				\addplot[thin, name path=t0] coordinates {(8,0) (11,0)};
				\addplot[domain=8:11, samples=100, color=black!75, thin, name path=b0] ({x},{sin(pi*x*\p/(2*\s))*sin(pi*x*\p/(2*\s))});
				\addplot[red!50] fill between [of=b0 and t0];
				
				\addplot[thin] coordinates {(9,0) (12,0)};
				\addplot[domain=9:12, samples=100, color=black!75, thin] ({x},{sin(pi*x*\p/(2*\s))*sin(pi*x*\p/(2*\s))});
				
				\addplot[thin] coordinates {(16,0) (20,0)};
				\addplot[domain=16:20, samples=100, color=black!75, thin] ({x},{sin(pi*x*\p/(2*\s))*sin(pi*x*\p/(2*\s))});
				
				\addplot[thin] coordinates {(2,-3) (6,-3)};
				\addplot[domain=2:6, samples=100, color=black!75, thin] ({x},{sin(pi*(x-\s)*\p/(2*\s))*sin(pi*(x-\s)*\p/(2*\s))-3});
				
				\addplot[thin] coordinates {(10,-3) (11,-3)};
				\addplot[domain=10:11, samples=100, color=black!75, thin] ({x},{sin(pi*(x-\s)*\p/(2*\s))*sin(pi*(x-\s)*\p/(2*\s))-3});

				\addplot[thin, name path=t1] coordinates {(11,-3) (14,-3)};
				\addplot[domain=11:14, samples=100, color=black!75, thin, name path=b1] ({x},{sin(pi*(x-\s)*\p/(2*\s))*sin(pi*(x-\s)*\p/(2*\s))-3});
				\addplot[red!50] fill between [of=b1 and t1];
				
				\addplot[thin] coordinates {(18,-3) (22,-3)};
				\addplot[domain=18:22, samples=100, color=black!75, thin, black] ({x},{sin(pi*(x-\s)*\p/(2*\s))*sin(pi*(x-\s)*\p/(2*\s))-3});
				
				\addplot[thin] coordinates {(4,-6) (8,-6)};
				\addplot[domain=4:8, samples=100, color=black!75, thin] ({x},{sin(pi*(x-2*\s)*\p/(2*\s))*sin(pi*(x-2*\s)*\p/(2*\s))-6});
				
				\addplot[thin] coordinates {(12,-6) (16,-6)};
				\addplot[domain=12:16, samples=100, color=black!75, thin] ({x},{sin(pi*(x-2*\s)*\p/(2*\s))*sin(pi*(x-2*\s)*\p/(2*\s))-6});
				
				\addplot[thin, name path=t2] coordinates {(15,-6) (16,-6)};
				\addplot[domain=15:16, samples=100, color=black!75, thin, name path=b2] ({x},{sin(pi*(x-2*\s)*\p/(2*\s))*sin(pi*(x-2*\s)*\p/(2*\s))-6});
				
				\addplot[thin] coordinates {(20,-6) (24,-6)};
				\addplot[domain=20:24, samples=100, color=black!75, thin] ({x},{sin(pi*(x-2*\s)*\p/(2*\s))*sin(pi*(x-2*\s)*\p/(2*\s))-6});
				
				\addplot[thin,color=blue,->] coordinates {(11, 0.0) (11,-3)};
			\end{axis}
		\end{tikzpicture}
		\caption{An illustration of equation \eqref{eq:nottoclosezeroes} in Lemma \ref{lem:4to6maxray} with $p=5$. The \CR{shaded} area represents the right hand side of the inequality \eqref{eq:nottoclosezeroes} without considering the integrand factor $1/(\pi x)^2$. The arrow marks the value \CB{$k+3/p-1/2$}. An illustration of the left hand side is given in Figure \ref{fig:4to5notoclose}.} 
		\label{fig:4to5notoclosepart2}
	\end{figure}
	We are now ready to prove Theorem \ref{thm:E_p4to5weak}. 
	\begin{proof}[Proof of Theorem \ref{thm:E_p4to5weak}]
		Let $\lambda$ be the sequence in $T(\delta_1, \delta_2)$ given by $\lambda_0=0$ and $\lambda_n=n-1/2+3/p$. 
		Due to Lemma \ref{lem:4to6nottolagrezeros} it is enough to show that $E_p(\tau) \leq E_p(\lambda)$ for all sequences $\tau$ in $T(\delta_1, \delta_2)$ such that $\tau_n \leq n-1/2 + 3/p$. Fix such a sequence $\tau$. Since $\delta_1 \geq 1-2/p \geq 2/p$ it follows that 
		\begin{equation}\label{eq:Apfork0}
			A_p(\tau; 0)= \int_{0}^{2/p} \frac{\sin^2 \frac{p}{2} \pi x}{\pi^2 x^2}dx.
		\end{equation}
		Applying Lemma \ref{lem:4to6Epsumoverrays} and then equation \eqref{eq:Apfork0} combined with Lemma \ref{lem:4to6maxray} it follows that
		\begin{align*}
			E_p(\tau)&= \sum_{k=0}^{\infty} A_p(\tau; k) \\
			&\leq \int_{0}^{2/p} \frac{\sin^2 \frac{p}{2} \pi x}{\pi^2 x^2}dx + \sum_{k=0}^{\infty} \int_{k+4/p}^{k+1/2+3/p} \frac{\sin^2 	\frac{p}{2}\pi (x-k)}{\pi^2 x^2}dx\\ &+ \sum_{k=0}^{\infty} \int_{k+1/2+3/p}^{k+1+2/p} \frac{\sin^2 \frac{p}{2}\pi (x-k-1)}{\pi^2 x^2}dx\\
			&=E_p(\lambda).
		\end{align*}
		This concludes the proof.
	\end{proof}
	Recall that we seek to prove Theorem \ref{thm:new:E_p4to5},  which is a stronger version of Theorem \ref{thm:E_p4to5weak} for $4<p \leq 5$. As we weaken the lower bound on $\delta_1$ to $1/2$ (compared with $1-2/p$ in Theorem \ref{thm:E_p4to5weak}) we will need the following inequality. 
	
	\begin{lemma}\label{lem:techinq4to5}
		Let $4 < p \leq 5$. Let $\tau$ and $\gamma$ be sequences in $T(1/2, 1-2/p)$ such that $\tau_1=1/2$ and $\gamma_1=3/p+1/2$. Then 
		\[A_p(\tau; 0) + A_p(\tau; 1) < A_p(\gamma; 0)+A_p(\gamma; 1).\]
	\end{lemma}
	
	\begin{proof}
		As in the proof of Lemma \ref{lem:4to6maxray} one can show that we need only consider $\tau_2=3/2-1/p$. 
		This means that we must verify the inequality		
		\begin{equation}\label{eq:4to5techlemstat}
			\begin{split}
				&\int_{4/p}^{1/2+3/p} \frac{\sin^2 \frac{p}{2} \pi x}{\pi^2 x^2}\,dx + \int_{1/2+3/p}^{1+2/p} \frac{\sin^2 \frac{p}{2} \pi (x-1)}{\pi^2 x^2}\,dx \\
				-&\int_{1/2}^{1-2/p} \frac{\sin^2 \frac{p}{2} \pi (x-1)}{\pi^2 x^2}\,dx - \int_{1}^{3/2-1/p} \frac{\sin^2 \frac{p}{2} \pi (x-1)}{\pi^2 x^2}\,dx\\ &- \int_{3/2-1/p}^{2-2/p} \frac{\sin^2 \frac{p}{2} \pi (x-2)}{\pi^2 x^2}\,dx > 0.				
			\end{split}
		\end{equation}	
		We observe that we have equality in \eqref{eq:4to5techlemstat} for $p=4$ so we may assume $p>4$. 
		Using that $x \mapsto 1/x$ is a decreasing function as well as Lemma \ref{lem:Brevig:lem6.5} we observe that it suffices to show
		\begin{equation}\label{eq:firstequiv}
			\begin{split}
				&\int_{7/p-1/2}^{1/2+3/p} \frac{\sin^2 \frac{p}{2} \pi x}{\pi^2 x^2} \, dx + \int_{1/2+4/p}^{1+2/p} \frac{\sin^2 \frac{p}{2} \pi (x-1)}{\pi^2 x^2}\,dx 
				\\ &-\int_{1/2}^{1-2/p} \frac{\sin^2 \frac{p}{2} \pi (x-1)}{\pi^2 x^2\,}dx - \int_{3/2-1/p}^{5/2-5/p} \frac{\sin^2 \frac{p}{2} \pi (x-2)}{\pi^2 x^2}\,dx \\ & - \int_{3/2}^{2-2/p} \frac{\sin^2 \frac{p}{2} \pi (x-2)}{\pi^2 x^2}\,dx > 0.
			\end{split}
		\end{equation}
		Then by substituting we see that this again is equivalent to
		\begin{equation*}\label{eq:11}
			\begin{split}
				&\int_{7/p-1/2}^{5/p}\sin^2{\frac{p}{2}\pi x}\left(\frac{1}{x^2}-\frac{1}{(x+2-8/p)^2}\right) \, dx \\ + &\int_{5/p}^{1/2+3/p}\sin^2{\frac{p}{2}\pi x}\left(\frac{1}{x^2}-\frac{1}{(x+2-8/p)^2}\right)\, dx  \\+ \int_{8/p-1/2}^{6/p}\sin^2{\frac{p}{2}\pi x}&\left(\frac{1}{(x+1-4/p)^2}-\frac{1}{(x+1-8/p)^2}-\frac{1}{(x+2-8/p)^2}\right) dx > 0.
			\end{split}
		\end{equation*}
		We estimate these three integrals by integrating the function $\sin^2((p/2) \pi x)$ and bounding the remaining integrand factors. We then obtain
		\begin{equation}\label{eq:12}
			\begin{split}
				&\int_{7/p-1/2}^{5/p}\sin^2{\frac{p}{2}\pi x}\left(\frac{1}{x^2}-\frac{1}{(x+2-8/p)^2}\right) \, dx\\ &\geq \frac{p^2}{2}\left(\frac{1}{25}-\frac{1}{(2p-3)^2}\right)\left(\frac{1}{2}-\frac{2}{p}+\frac{\sin{\frac{p}{2}\pi}}{p \pi} \right),
			\end{split}
		\end{equation}
		
		\begin{equation}\label{eq:13}
			\begin{split}
				&\int_{5/p}^{1/2+3/p}\sin^2{\frac{p}{2}\pi x}\left(\frac{1}{x^2}-\frac{1}{(x+2-8/p)^2}\right)\,  dx\\ &\geq
				\frac{p^2}{2}\left(\frac{4}{(p+6)^2}-\frac{4}{25(p-2)^2}\right)\left(\frac{1}{2}-\frac{2}{p}+\frac{\sin{\frac{p}{2}\pi}}{p \pi} \right),
			\end{split}
		\end{equation}
		and	
		\begin{equation}\label{eq:14}
			\begin{split}
				\int_{8/p-1/2}^{6/p}\sin^2{\frac{p}{2}\pi x}&\left(\frac{1}{(x+1-4/p)^2}-\frac{1}{(x+2-8/p)^2}-\frac{1}{(x+1-8/p)^2}\right)dx  \\ & \leq
				\frac{p^2}{2}\left(\frac{4}{(8+p)^2}-\frac{40}{9p^2}\right)\left(\frac{1}{2}-\frac{2}{p}-\frac{\sin{\frac{p}{2}\pi}}{p \pi} \right).
			\end{split}
		\end{equation}
		Next we combine the above estimates \eqref{eq:12}, \eqref{eq:13} and \eqref{eq:14} and multiply by \newline $4/(p(p-4))$ to see that \eqref{eq:firstequiv} follows if 
		\[A(p)+B(p)+\frac{2}{\pi} \frac{\sin \frac{p}{2}\pi}{p-4}(A(p)-B(p)) > 0,\]
		where 
		\[A(p)=\frac{1}{25}-\frac{1}{(2p-3)^2}+\frac{4}{(p+6)^2}-\frac{4}{25(p-2)^2},\]
		and 
		\[B(p)=\frac{4}{(8+p)^2}-\frac{40}{9p^2}.\]
		Using that 
		\[\frac{2}{\pi(p-4)}\sin\left( \frac{p}{2}\pi\right)\geq p\left(\frac{2}{\pi}-1\right)+5-\frac{8}{\pi},\]we observe that it suffices to show
		\begin{equation}\label{eq:0956}
			A(p)\left(p\left(\frac{2}{\pi}-1\right)+6-\frac{8}{\pi}\right) > B(p)(p-4)\left(1-\frac{2}{\pi}\right).
		\end{equation}
		Using that 
		\[\left(p\left(\frac{2}{\pi}-1\right)+6-\frac{8}{\pi}\right)/\left(1-\frac{2}{\pi}\right) \geq 4.5,\]
		and multiplying by $4/(p-4)$ we see that \eqref{eq:0956} holds if
		\begin{equation}\label{eq:4to5final}
			\frac{p(p^4+9p^3+96p^2-456p+456)}{25(2p-3)^3(p+6)^2(p-2)^2} -\frac{p^2+160p+640}{9(8+p)^2p^2} > 0.
		\end{equation}
		Rewriting this as a polynomial inequality of degree $9$ one can indeed verify \eqref{eq:4to5final}.
	\end{proof}
	
	We conclude the section by proving Theorem \ref{thm:new:E_p4to5}.
	\begin{proof}[Proof of Theorem \ref{thm:new:E_p4to5}]
		Clearly it is enough to prove Theorem \ref{thm:new:E_p4to5} for $\delta_1=1/2$. If $\tau_1 \geq 1-2/p$ we are done by Theorem \ref{thm:E_p4to5weak}, so assume $\tau_1 < 1-2/p$. Then for $E_p(\tau)$ to be maximal we let $\tau_1=1/2$. Recall that $\lambda$ is the sequence in $T(\delta_1, \delta_2)$ given by $\lambda_n=n-1/2+3/p$. From Lemma \ref{lem:techinq4to5} it follows that 
		\[A_p(\tau; 0) + A_p(\tau; 1) \leq A_p(\lambda;0 )+A_p(\lambda; 1).\] Further by Lemma \ref{lem:4to6maxray} it follows that
		\[A_p(\tau;k) \leq A_p(\lambda; k)\]
		for all $k \geq 2$. That is
		\[E_p(\tau)=\sum_{k=0}^{\infty}A_p(\tau; k) \leq \sum_{k=0}^{\infty}A_p(\lambda; k) = E_p(\gamma), \]
		and equality is attained if and only if $p=4$.
	\end{proof}
	
	\appendix
	\section{Proof of Lemma \ref{lem:postponeactint} and Lemma \ref{lem:pgeq36spes}.}
	In this appendix we prove Lemma \ref{lem:postponeactint} and Lemma \ref{lem:pgeq36spes}.
	We will split the proof of Lemma \ref{lem:postponeactint} into three subcases for different values of $y$. Certain estimates will be used frequently in several subcases, and we thus begin by stating these results separately. 
	
	\begin{lemma}\label{lem:decreasingfctest}
		Let $f(x)$ be a non-negative decreasing function on an interval $[a,b]$. Then 
		\[ \int_{a}^{b} f(x) \sin^2 \frac{p}{2}\pi x \, dx \geq \frac{f(b)}{2}\left(b-a+\frac{\sin p \pi a-\sin p \pi b}{p \pi}\right) \]
		and 
		\[\int_{a}^{b} f(x) \sin^2 \frac{p}{2}\pi x \, dx \leq \frac{f(a)}{2}\left(b-a+\frac{\sin p \pi a-\sin p \pi b}{p \pi}\right). \]
	\end{lemma}
	
	\begin{proof}
		This is easily verified by using the inequalities $f(b) \leq f(x) \leq f(a)$ for all $x$ in $[a,b]$ and then preforming the integral \[ \int_{a}^{b}\sin^2 \frac{p}{2}\pi x \, dx. \qedhere \] 
	\end{proof}
	The following lemma is also straightforward but we state it for ease of reference. 
	\begin{lemma}\label{lem:sin2est}
		Let $f$ be a non-negative integrable function on the interval $[a,b]$, and let $F$ be the antiderivative of $f$. Assume $c_1(t) \leq \sin^2 (t x) \leq c_2(t)$ for all $x$ in $[a,b]$. Then 
		\[c_1(t)(F(b)-F(a)) \leq \int_{a}^{b} \sin^2 (tx) f(x) \dx \leq c_2(t)(F(b)-F(a)).\]
	\end{lemma} 
	Next we state a specific estimate that will come in handy later. 
	\begin{lemma}\label{lem:destforsin}
		Let $3 \leq p \leq 4$. Then 
		\[2\sin\left(\frac{2 p \pi}{3}\right)\cos\left(\frac{p \pi}{3}\right) +\frac{7}{50} p \pi \geq 0. \]
	\end{lemma}
	
	\begin{proof} 
		By studying the derivative of this expression we see that it has precisely one minimal point $p_0$ for $3 \leq p \leq 4$. The claim follows from estimating the value attained in $p_0$.
	\end{proof}
	
	We will now prove Lemma \ref{lem:postponeactint}. We start by observing that $\gamma$ is in $T_2$ since $\tau$ is. 
	Thus what needs to be shown is \eqref{eq:appendixfinal}. We split the proof into three subcases depending on the value of $y$.
	\subsection{The case \texorpdfstring{$1/2-1/p \leq y \leq 1/3$}{1/2-1/p≤y≤1/3}}
	We first consider the smallest possible values for $y$. Assume $3<p<4$, let $1/2-1/p \leq y \leq 1/3$ and let $(\xi, \xi+2/p)$ be an interval at level $k$ (that is $\xi=k-(4/p)j$) such that $\xi \geq 1$. 
	We define the following integrals.
	\begin{align*}
		I_1 &=\max\left( \int_{\xi}^{\xi+y-4/3+4/p} \sin^2{\frac{p}{2}\pi (x-k)} \left(\frac{1}{(x+1)^2}-\frac{1}{(x+3-4/p)^2}\right)\, dx, 0\right),\\
		I_2 &= \int_{\xi+y-4/3+4/p}^{\xi+y} \sin^2{\frac{p}{2}\pi (x-k)} \left(\frac{1}{(x+2-4/p)^2}-\frac{1}{(x+1)^2}\right)\, dx,\\
		I_{3_1} &= \int_{\xi+y}^{\xi+y+2/p-1/2} \sin^2{\frac{p}{2}\pi (x-k)} \left(\frac{1}{x^2}-\frac{1}{(x+2-4/p)^2}\right)\, dx,\\
		I_{3_2} &= \int_{\xi+y+2/p-1/2}^{\xi+y+4/p-1} \sin^2{\frac{p}{2}\pi (x-k)} \left(\frac{1}{x^2}-\frac{1}{(x+2-4/p)^2}\right)\, dx,\\
		I_4 &= \int_{\xi+y+4/p-1}^{\xi+2/p} \sin^2{\frac{p}{2}\pi (x-k)} \left(\frac{1}{(x+1-4/p)^2}-\frac{1}{x^2}\right)\, dx.
	\end{align*}
	\begin{proof}[Proof of inequality \eqref{eq:appendixfinal} for $1/2-1/p \leq y \leq 1/3$.] One can show that proving \eqref{eq:appendixfinal} in the range $1/2-1/p \leq y \leq 1/3$ is equivalent to showing $I_1-I_2+I_{3_1}+I_{3_2}-I_4 > 0$. In particular, since $I_1 \geq 0$ it suffices to show
		\begin{equation}\label{eq:case1equiv}
			I_{3_1}+I_{3_2}-I_2-I_4 > 0.
		\end{equation}	
		We start by showing that $I_{3_2} > I_{4}$. Since $y \geq 1/2-1/p$ it follows that $\sin^2\frac{p}{2} \pi(x-k)$ is decreasing for $x$ in $(\xi+y+2/p-1/2, \xi+2/p)$. Thus by Lemma \ref{lem:sin2est} it suffices to show 
		\begin{equation*}
			\int_{\xi+y+4/p-1}^{\xi+2/p}\left( \frac{1}{(x+1-4/p)^2}-\frac{1}{x^2} \right)dx \leq
			\int_{\xi+y+2/p-1/2}^{\xi+y+4/p-1} \left( \frac{1}{x^2}-\frac{1}{(x+2-4/p)^2}\right)dx.
		\end{equation*}
		Evaluating these integrals, we see that this is true if $f(\xi, p, y) > 0$, where
		\[\begin{split}
			f(\xi,p,y)&=\frac{1}{\xi+y+1}-\frac{1}{\xi+y+3/2-2/p}+\frac{1}{\xi+y+2/p-1/2}\\&-\frac{1}{\xi+y}-\frac{1}{\xi+2/p}+\frac{1}{\xi+1-2/p}.
		\end{split}\]
		As the function $f$ is increasing in $y$, we may insert $y=1/2-1/p$ to obtain
		\[\begin{split}
			&f\left(\xi, p, \frac{1}{2}-\frac{1}{p}\right)\\&=\frac{p(p-2)(4-p)(4\xi^3p^3+5\xi^2p^3-10\xi^2p^2-10\xi p^2+8\xi p-6p^2+11p-2)}{(2\xi p+3p-2)(\xi p+2p-3)(\xi p+1)(2\xi p+p-2)(\xi p+2)(\xi p+p-2)}.
		\end{split}\]
		All factors in the denominator are positive, so $f(\xi,p,1/2-1/p)>0$ if
		\begin{equation}\label{eq:300}
			4\xi^3p^3+5\xi^2p^3-10\xi^2p^2-10\xi p^2+8\xi p-6p^2+11p-2 > 0.
		\end{equation}
		It is easy to check that
		\[p^2(4\xi^3p+5\xi ^2p-10\xi ^2-10\xi -6) > 0 \quad \text{and} \quad 8\xi p+11p-2 >0,\]
		so \eqref{eq:300} holds and this concludes the proof of the inequality $I_{3_2} > I_{4}$.
		
		We proceed with the second part of inequality \eqref{eq:case1equiv}, namely showing $I_{3_1} \geq I_{2}$. By Lemma \ref{lem:decreasingfctest} it suffices to show
		\begin{equation}\label{eq:1521}
			\begin{split}
				&\left(\frac{1}{(\xi+y+2/p-1/2)^2}-\frac{1}{(\xi+y+3/2-2/p)^2}\right)\\ &\times\left(\frac{2}{p}-\frac{1}{2}+\frac{2 \sin (p \pi/4) \cos(p \pi (y-1/4))}{\pi p}\right) \\
				+&\left(\frac{1}{(\xi+y+2/3)^2}-\frac{1}{(\xi+y-1/3+4/p)^2}\right) \\ &\times \left(\frac{4}{p}-\frac{4}{3}+\frac{2 \sin(2 p \pi /3)\cos(p \pi (y-2/3))}{\pi p}\right) > 0.
			\end{split}
		\end{equation}
		We observe that $\sin(2 p \pi /3)\cos(p \pi (y-2/3))$ is decreasing in $y$, so we can choose $y=1/3$ here. Applying Lemma \ref{lem:destforsin} and observing that \[2 \sin (p \pi/4) \cos(p \pi (y-1/4)) \geq 0\] we get that \eqref{eq:1521} holds if
		\begin{equation}\label{eq:31}
			\begin{split}
				&\left(\frac{1}{(\xi+y+2/p-1/2)^2}-\frac{1}{(\xi+y+3/2-2/p)^2}\right)\left(\frac{2}{p}-\frac{1}{2}\right) \\
				+&\left(\frac{1}{(\xi+y+2/3)^2}-\frac{1}{(\xi+y-1/3+4/p)^2}\right) \left(\frac{4}{p}-1-\frac{71}{150}\right) > 0.
			\end{split}
		\end{equation}
		By elementary operations and the substitution $x=\xi+y$ we see that \eqref{eq:31} is equivalent to
		\begin{equation*}
			\frac{32p^3(2x+1)(p-2)(4-p)}{(2px - p + 4)^2(2px + 3p - 4)^2 2p}+\frac{27(4-p)(6px+p+12)}{(3x + 2)^2 (3px - p + 12)^2}\left(\frac{4}{p}-1-\frac{71}{150}\right) > 0,
		\end{equation*} 
		for $x \geq 3/2-1/p$.
		This is in turn equivalent to 
		\begin{equation*}
			\begin{split}
				&16p^3(2x+1)(p-2)(3x + 2)^2(3px - p + 12)^2\\
				& > 27(6px+p+12)\left(p+\frac{71}{150}p-4\right)(2px-p+4)^2(2px+3p-4)^2,
			\end{split}
		\end{equation*}
		which by a rearrangement of terms can be expressed as 
		\begin{equation*}
			\begin{split}
				&(p-2)\left(px+\frac{p}{2}\right)\left(px+\frac{2}{3}p\right)^2\left(px-\frac{p}{3}+4\right)^2\\ &> \frac{221}{150}\left(p-\frac{600}{221}\right)\left(px-\frac{p}{2}+2\right)^2\left(px+\frac{p}{6}+2\right)\left(px+\frac{3}{2}p-2\right)^2.
			\end{split}
		\end{equation*}
		Observing that
		\[p-2 > \frac{221}{150}\left(p-\frac{600}{221}\right), \quad  px+\frac{p}{2}> px-\frac{p}{2}+2  \quad \text{and} \quad px-\frac{p}{3}+4 > px+\frac{p}{6}+2,\]
		it follows that we need only show
		\begin{equation*}
			\left(px+\frac{2}{3}p\right)^2\left(px-\frac{p}{3}+4\right) > \left(px-\frac{p}{2}+2\right)\left(px+\frac{3}{2}p-2\right)^2.
		\end{equation*}
		This is equivalent to
		\[\left(-\frac{3}{2}p+6\right)p^2x^2+\left(-\frac{3}{4}p^2+\frac{10}{3}p+4\right)px+\frac{211}{216}p^3-\frac{103}{18}p^2+14p-8 > 0,\]
		which holds since the polynomial in $x$ has only positive coefficients. This completes the proof that $I_{3_1} \geq I_2$, and thus \eqref{eq:appendixfinal} holds for $1/2-1/p \leq y \leq 1/3$.
	\end{proof}
	\subsection{The case \texorpdfstring{$1-2/p \leq y<y_{max}$}{1-2/p≤y<{ymax}}}	
	Recall that $y_{max}=5/6-1/p$ for $p \leq 3.6$ and $y_{max}=2/p$ for $p \geq 3.6$. Let $3<p<4$, $1-2/p \leq y<y_{max}$ and let $(\xi, \xi+2/p)$ be an interval at level $k$ such that $\xi \geq 1$. We then denote
	\begin{align*}
		I_1 &= \int_\xi^{\xi+y-1/3} \sin^2{\frac{p}{2}\pi (x-k)} \left(\frac{1}{(x+3-4/p)^2}-\frac{1}{(x+2)^2}\right) dx,\\
		I_2 &= \int_{\xi+y-1/3}^{\xi+y+4/p-4/3} \sin^2{\frac{p}{2} \pi (x-k)} \left(\frac{1}{(x+1)^2} - \frac{1}{x+3-4/p}\right)dx,\\
		I_3 &= \int_{\xi+y+4/p-4/3}^{\xi+y} \sin^2{\frac{p}{2}\pi(x-k)}\left(\frac{1}{(x+2-4/p)^2}-\frac{1}{(x+1)^2}\right)dx,\\
		I_4&=\int_{\xi+y}^{\xi+2/p} \sin^2{\frac{p}{2} \pi (x-k)} \left(\frac{1}{x^2}-\frac{1}{(x+2-4/p)^2}\right)dx.
	\end{align*}
	
	\begin{proof}[Proof of inequality \eqref{eq:appendixfinal} for $1-2/p \leq y < y_{max}$.]
		One can show that proving \eqref{eq:appendixfinal} for $1-2/p \leq y<y_{max}$ is equivalent to showing $I_2+I_4-I_1-I_3 > 0$. 
		In particular since $I_4 \geq 0$ it suffices to show 	$I_2-I_1-I_3 > 0.$
		Using Lemma \ref{lem:decreasingfctest} and observing that $\sin p \pi (y-1/3) \geq 0$ it follows that
		\begin{equation}\label{eq:case3I1}
			I_1 \leq \left(\frac{1}{(\xi+3-4/p)^2}-\frac{1}{(\xi+2)^2}\right)\left(y-\frac{1}{3}\right).
		\end{equation}
		Next by Lemma \ref{lem:decreasingfctest} and the fact that $-\sin p \pi (y-4/3)+\sin p \pi (y-1/3)\geq 0$ it follows that
		\begin{equation}\label{eq:case2I2}
			I_2 \geq \left(\frac{1}{(\xi+y-1/3+4/p)^2}-\frac{1}{(\xi+y+5/3)^2}\right)\left(\frac{4}{p}-1\right).
		\end{equation}
		Further we see that
		\[ \begin{split}
			\sin  p \pi \left(y-\frac{4}{3}\right)-\sin (p \pi y) &=-2\sin \frac{2 p \pi}{3} \cos \frac{p}{2} \pi \left(2y-\frac{4}{3}\right) \\ &\leq -2 \sin \frac{2 p \pi}{3} \cos \frac{p \pi}{3}  \leq \frac{7}{50} p \pi,
		\end{split}\]
		where the last inequality follows from Lemma \ref{lem:destforsin}. Thus by Lemma \ref{lem:decreasingfctest} we get 
		\begin{equation}\label{eq:case3I3}
			I_3 \leq \left(\frac{1}{(\xi+y+2/3)^2}-\frac{1}{(\xi+y-1/3+4/p)^2}\right) \left(\frac{4}{3}- \frac{4}{p}
			+\frac{7}{50}\right).
		\end{equation}
		Combining \eqref{eq:case3I1}, \eqref{eq:case2I2} and \eqref{eq:case3I3} we see that it suffices to show $F(\xi, p, y) \geq 0$, where 
		\begin{align*}
			F(\xi,p,y)=&\left(\frac{1}{(\xi+y+2/3)^2}-\frac{1}{(\xi+y+5/3)^2}\right)\left(\frac{4}{p}-1\right) \\
			-&\left(\frac{1}{(\xi+3-4/p)^2}-\frac{1}{(\xi+2)^2}\right)\left(y-\frac{1}{3}\right)\\
			-&\frac{71}{150} \left(\frac{1}{(\xi+y+2/3)^2}-\frac{1}{(\xi+y-1/3+4/p)^2}\right).
		\end{align*}
		Note that $F(\xi,4,y)=0$, so we are done if we can show that $F(\xi,p,y)$ is decreasing in $p$ for $3 < p < 4$. In particular since $1/3 \leq y<5/9$ for all $p$ we are going to show that $F(\xi,p,y)$ is decreasing as a function of $p$ for all fixed $1/3 \leq y<5/9$ and $\xi \geq 1$. Differentiating $F$ with respect to $p$, we get 
		\begin{align*}
			\frac{d}{dp}F(\xi,p,y)&=\frac{4}{p^2}\bigg(-\frac{1}{(\xi+y+2/3)^2}+\frac{1}{(\xi+y+5/3)^2}+(y-1/3)\frac{2}{(\xi+3-4/p)^3}\nonumber \\&+\frac{71}{75} \frac{1}{(\xi+y-1/3+4/p)^3}\bigg).
		\end{align*}
		Thus it will be enough to show $G(\xi,p,y) < 0$, where
		\[\begin{split}
			G(\xi,p,y)&=\frac{p^2}{4}\frac{d}{dp}F(\xi,p,y).
		\end{split}\]
		Differentiating $G$ with respect to $y$ and then using that $y \geq 1/3$ we see that
		\begin{align*}
			\frac{d}{dy}G(\xi,p,y)&=\frac{2}{(\xi+y+2/3)^3}-\frac{2}{(\xi+y+5/3)^3}+\frac{2}{(\xi+3-4/p)^3}\\&-\frac{71}{25} \frac{1}{(\xi+y-1/3+4/p)^4}\\
			&\geq \frac{2}{(\xi+y+2/3)^3}-\frac{71}{25} \frac{1}{(\xi+y-1/3+4/p)^4}\\
			&\geq \frac{2}{(\xi+y+2/3)^3}-\frac{71}{25} \frac{1}{(\xi+y+2/3)^4}\\
			&=\frac{2}{(\xi+y+2/3)^3}\left(2-\frac{71}{25} \frac{1}{(\xi+y+2/3)}\right) > 0,
		\end{align*}
		where the second inequality follows since the expression is decreasing in $p$. Thus since $G(\xi, p, y)$ is increasing for $1/3 \leq y \leq 5/9$ we aim to show $G(\xi,p,5/9) < 0$. Differentiating with respect to $p$ we get
		\begin{align*}
			\frac{d}{dp}G(\xi,p,5/9)=\frac{4}{p^2}\left(-\frac{4}{3}\frac{1}{(\xi+3-4/p)^4}+\frac{71}{25}\frac{1}{(\xi+2/9+4/p)^4}\right)>0.
		\end{align*}
		That is $G(\xi, p,5/9)$ is increasing in $p$, so the proof is complete if 
		\begin{equation}\label{eq:1434}
			\frac{1}{(\xi+20/9)^2}-\frac{1}{(\xi+11/9)^2}+\frac{4}{9}\frac{1}{(\xi+2)^3}+\frac{71}{75}\frac{1}{(\xi+11/9)^3}<0.
		\end{equation}
		By writing this as a polynomial equation one can easily check that \eqref{eq:1434} holds for all $\xi \geq 1$.
	\end{proof}	
	\subsection{The case \texorpdfstring{$1/3<y<1-2/p$}{1/3<y<1-2/p}}
	Now assume $3<p<4$ and $1/3<y<1-2/p$ and let $(\xi, \xi+2/p)$ be an interval at level $k$ such that $\xi \geq 1$. We will use a slightly different technique and split the interval $(1/3,1-2/p)$ into two subintervals. We start by assuming $y \geq 2/3-1/p$, and define the integrals
	\begin{align*}
		J_1&=\int_{2-4/p}^{1} \sin^2 \frac{p}{2}\pi(x-2)\left(\frac{1}{(\xi-x+2-2/p)^2}-\frac{1}{(\xi+x+1)^2}\right)\, dx\\
		J_2&=\int_{1}^{2-y-2/p}\sin \frac{-p}{2} \pi \sin {p \pi (x-3/2)}\left(\frac{1}{(\xi-x+2-2/p)^2}-\frac{1}{(\xi+x+1)^2}\right)\, dx\\
		J_3&=\int_{2-y-2/p}^{y+2/3}\sin \frac{-p}{2}\pi \sin {p \pi (x-3/2)}\left(\frac{1}{(\xi-x+3-2/p)^2}-\frac{1}{(\xi+x+1)^2}\right)\, dx\\
		J_4&=\int_{y+2/3}^{3/2-1/p}\sin \frac{-p}{2}\pi \sin {p \pi (x-3/2)}\left(\frac{1}{(\xi+x)^2}-\frac{1}{(\xi-x+3-2/p)^2}\right)\, dx
	\end{align*}
	\begin{proof}[Proof of inequality \eqref{eq:appendixfinal} for $2/3-1/p < y <1-2/p$.]
		One can show that proving \eqref{eq:appendixfinal} for $2/3-1/p<y<1-2/p$ is equivalent to proving \[J_1+J_2+J_3-J_4>0.\] It is easy to see that $J_1, J_2, J_3, J_4 \geq 0$, and that
		\[J_4=\int_{2-4/p}^{17/6-y-5/p}f(x)\, dx,\] where
		\[f(x,\xi,p)=-\sin \frac{p}{2}\pi \sin p \pi (x-2) \left(\frac{1}{(\xi-x+7/2-5/p)^2}-\frac{1}{(\xi+x-1/2+3/p)^2}\right).\]
		We split the interval of integration $I=(2-4/p,17/6-y-5/p)$ into the subintervals $I_1=(2-4/p,\min(2-2/p-y,17/6-5/p-y))$, $I_2=(2-2/p-y,\min(y+2/3,17/6-y-5/p))$ and $I_3=(y+2/3,17/6-y-5/p)$, where we consider an interval $(a,b)$ to be empty if $a>b$. We start by showing that 
		\begin{equation}\label{eq:0856}
			\int_{I_1} f(x,\xi,p)\, dx < J_1+J_2,
		\end{equation}
		by establishing the two pointwise estimates
		\begin{equation}\label{eq:22081}
			\begin{split}
				&\sin^2\frac{p}{2}\pi(x-2)\left(\frac{1}{(\xi-x+2-2/p)^2}-\frac{1}{(\xi+x+1)^2}\right)\\ \geq &\sin \frac{-p}{2}\pi \sin {p \pi (x-2)}\left(\frac{1}{(\xi-x+7/2-5/p)^2}-\frac{1}{(x+\xi+3/p-1/2)^2}\right) 
			\end{split}
		\end{equation}
		for $x$ in $(2-4/p,1)$, and
		\begin{equation}\label{eq:22082}
			\begin{split}
				&\sin {p \pi (x-3/2)}\left(\frac{1}{(\xi-x+2-2/p)^2}-\frac{1}{(\xi+x+1)^2}\right)\\ \geq &\sin {p \pi (x-2)}\left(\frac{1}{(\xi-x+7/2-5/p)^2}-\frac{1}{(\xi+x-1/2+3/p)^2}\right) 
			\end{split}
		\end{equation}
		for $x$ in $(1,4/3-1/p)$. (Note that $4/3-1/p=2-y-2/p$ for $y=2/3-1/p$.) Starting with the former estimate, we observe that 
		\[\sin^2 \frac{p}{2}\pi(x-2) \geq \sin p \pi (x-2) \sin(-p/2\pi) (x-2+4/p),\] so \eqref{eq:22081} follows if
		\[\begin{split}
			&(x-2+4/p)\left(\frac{1}{(\xi-x+2-2/p)^2}-\frac{1}{(x+\xi+1)^2}\right)\\&-\left(\frac{1}{(\xi-x+7/2-5/p)^2}-\frac{1}{(x+\xi+3/p-1/2)^2}\right) \geq 0.
		\end{split}\]
		This can be verified by proving that the expression is $0$ when $x=2-4/p$ and increasing in $x$. 
		
		Next, to show \eqref{eq:22082} we observe that \[2 \sin(p\pi(x-3/2)) \geq \sin(p \pi(x-2)),\] thus by letting $\eta=\xi-x$ we see that \eqref{eq:22082} follows if
		\begin{equation*}
			\frac{1}{(\eta+2-2/p)^2}-\frac{1}{(\eta+2x+1)^2} -\frac{2}{(\eta+7/2-5/p)^2}+\frac{2}{(\eta+2x-1/2+3/p)^2} \geq 0,
		\end{equation*}
		for $\eta  \geq 1/p-1/3$. This can be shown by observing that the expression is decreasing in $x$ so it suffices to consider $x=4/3-1/p$, and then rewriting the expression as a polynomial equation. Thus \eqref{eq:0856} follows. In particular we note that this concludes the proof for $p \leq 3.6$ since $I_2=I_3=\emptyset$ in this case. 
		
		Now assume $p>3.6$. We claim that 
		\begin{equation}\label{eq:1452}
			\int_{I_2} f(x,\xi,p)\, dx \leq J_3.
		\end{equation}
		Assume first that $y \geq 13/12-5/(2p)$, that is \[\min(y+2/3,17/6-y-5/p)=17/6-y-5/p.\] It suffices to show 
		\[\begin{split}
			&\int_{2-y-2/p}^{y+2/3} \sin p \pi (x-3/2) \left(\frac{1}{(\xi+1+y)^2}-\frac{1}{(\xi+3-y-2/p)^2}\right) \, dx\\ -&\int_{2-y-2/p}^{17/6-y-5/p} \sin p \pi (x-2) \left(\frac{1}{(\xi+2/3+y)^2}-\frac{1}{(\xi+7/3-y-2/p)^2}\right) \, dx \geq 0,
		\end{split}
		\]
		which is equivalent to showing 
		\begin{equation*}
			\begin{split}
				&(\cos p \pi  (1/2-y)-\cos p \pi(5/6-y))\left(\frac{1}{(\xi+1+y)^2}-\frac{1}{(\xi+3-y-2/p)^2}\right)\\ \geq &(\cos p \pi y+\cos p \pi(5/6-y))\left(\frac{1}{(\xi+2/3+y)^2}-\frac{1}{(\xi+7/3-y-2/p)^2}\right).
			\end{split}
		\end{equation*}
		Using that
		\[\cos p \pi (1/2-y)-\cos p \pi(5/6-y) \geq 0.91 (\cos p \pi y+\cos p \pi(5/6-y)), \]
		it suffices to show
		\[\frac{0.91}{(\xi+1+y)^2}-\frac{0.91}{(\xi+3-y-2/p)^2}\\-\frac{1}{(\xi+2/3+y)^2}+\frac{1}{(\xi+7/3-y-2/p)^2} \geq 0.\]
		This expression is increasing in $y$, so we insert $y=13/12-5/(2p)$, and conclude by writing the expression as a rational function. 
		
		Next assume $y \leq 13/12-5/(2p)$, that is $\min(y+2/3,17/6-y-5/p)=y+2/3$.
		In order to show \eqref{eq:1452} it then suffices to verify the pointwise estimate
		\begin{equation}\label{eq:1354}
			\begin{split}
				&\sin p \pi (x-3/2) \left(\frac{1}{(\xi-x+3-2/p)^2}-\frac{1}{(\xi+x+1)^2}\right)\\ \geq &\sin p \pi(x-2) \left(\frac{1}{(\xi-x+7/2-5/p)^2}-\frac{1}{(\xi+x-1/2+3/p)^2}\right),
		\end{split}\end{equation}
		for $x$ in $(11/12+1/(2p), 7/4-5/(2p))$.
		We observe that \[\frac{\sin p \pi (x-3/2)}{\sin p \pi (x-2)} \geq \frac{\sin p \pi (-7/12+1/(2p))}{\sin p \pi (-13/12+1/(2p))}, \]
		so \eqref{eq:1354} will follow if
		\[\begin{split}
			&\frac{\cos  7 p \pi/12}{\cos 13 p \pi /12} \left(\frac{1}{(\xi-x+3-2/p)^2}-\frac{1}{(\xi+x+1)^2}\right)\\ - & \left(\frac{1}{(\xi-x+7/2-5/p)^2}-\frac{1}{(\xi+x-1/2+3/p)^2}\right) \geq 0. \end{split}\]
		The expression above is decreasing in $x$ and it thus suffices to show
		\begin{equation}\label{eq:1106}
			\begin{split}
				&C\left(\frac{1}{(\xi+5/4+1/(2p))^2}-\frac{1}{(\xi+11/4-5/(2p))^2}\right)\\ - & \left(\frac{1}{(\xi-7/4-5/(2p))^2}-\frac{1}{(\xi+5/4+1/(2p))^2}\right) \geq 0, 
			\end{split}
		\end{equation}
		where $C \geq (\cos  7 p \pi/12)/(\cos 13 p \pi /12)$. For a fixed constant $C$ the expression in \eqref{eq:1106} is increasing in $p$ as each term is increasing in $p$. Thus we can split the range $3.6<p<4$ into subintervals and only consider the smallest $p$ in each interval, and then conclude by viewing \eqref{eq:1106} a polynomial in $\xi$. One such option of intervals and corresponding constant $C$ is to let $C=0.92$ for $p$ in $[3.6,3.65)$, $C=0.87$ for $p$ in $[3.65,3.7)$ and $C=0.83$ for $p$ in $[3.7,4)$.
		
		Finally we show that when $y \leq 13/12-3/(2p)$ then \[\int_{I_1 \cup I_3} f(x,\xi,p)\, dx \leq J_1+J_2.\]
		We observe that \[\begin{split}&\int_{I_3} f(x,\xi,p)\, dx  \leq 2  \sin\frac{-p}{2}\pi \\ \times \int_{11/12+1/(2p)}^{2-2/p-y} \sin p \pi(x-7/6) &\left(\frac{1}{(\xi-x+8/3-2/p)^2}-\frac{1}{(x+\xi+1/3)^2}\right) \, dx.\end{split}\]
		Using this and what we know from \eqref{eq:22081} and \eqref{eq:22082} it will be enough to show the pointwise estimate
		\begin{equation}\label{eq:220803}
			\begin{split}
				&\sin {p \pi (x-3/2)}\left(\frac{1}{(\xi-x+2-2/p)^2}-\frac{1}{(\xi+x+1)^2}\right)\\
				\geq &-2\sin p \pi(x-7/6) \left(\frac{1}{(\xi-x+8/3-2/p)^2}-\frac{1}{(x+\xi+1/3)^2}\right)\\
				&+\sin {p \pi (x-2)}\left(\frac{1}{(\xi-x+7/2-5/p)^2}-\frac{1}{(\xi+x-1/2+3/p)^2}\right),
			\end{split}
		\end{equation}
		for $x$ in $(11/12+1/(2p), 4/3-1/p)$.
		Since
		\[5 \sin p \pi(x-3/2) \geq 2 \sin p \pi(-x+7/6)+\sin p \pi(x-2),\]
		and
		\[\begin{split} &\frac{1}{(\xi-x+7/2-5/p)^2}-\frac{1}{(\xi+x-1/2+3/p)^2} \\ &\leq \frac{1}{(\xi-x+8/3-2/p)^2}-\frac{1}{(x+\xi+1/3)^2}, \end{split}\]
		inequality \eqref{eq:220803} holds if 
		\[\frac{1}{(\xi-x+2-2/p)^2}-\frac{1}{(\xi+x+1)^2} -5\left(\frac{1}{(\xi-x+8/3-2/p)^2}-\frac{1}{(x+\xi+1/3)^2}\right).\]
		This expression is decreasing in $x$ so we may conclude by inserting $x=4/3-1/p$ and rewriting the expression as a rational function. This completes the proof of \eqref{eq:appendixfinal} for $2/3 -1/p \leq y \leq 1-2/p$. 
	\end{proof}
	
	We now proceed with the case $1/3<y<2/3-1/p$, where we define 
	\begin{align*}
		J_1&=\int_{2-4/p}^{1} \sin^2 \frac{p}{2}\pi(x-2)\left(\frac{1}{(\xi-x+2-2/p)^2}-\frac{1}{(\xi+x+1)^2}\right)\, dx\\
		J_2&=\int_{1}^{y+2/3}\sin \frac{-p}{2} \pi \sin {p \pi (x-3/2)}\left(\frac{1}{(\xi-x+2-2/p)^2}-\frac{1}{(\xi+x+1)^2}\right)\, dx\\
		J_3&=\int_{y+2/3}^{2-y-2/p}\sin \frac{-p}{2}\pi \sin {p \pi (x-3/2)}\left(\frac{1}{(\xi-x+2-2/p)^2}-\frac{1}{(\xi+x)^2}\right)\, dx\\
		J_4&=\int_{2-y-2/p}^{3/2-1/p}\sin \frac{-p}{2}\pi \sin {p \pi (x-3/2)}\left(\frac{1}{(\xi+x)^2}-\frac{1}{(\xi-x+3-2/p)^2}\right)\, dx
	\end{align*}
	\begin{proof}[Proof of inequality \eqref{eq:appendixfinal} for $1/3<y \leq 2/3-1/p$.]
		One can show that proving \eqref{eq:appendixfinal} for $1/3 \leq y \leq 2/3-1/p$ is equivalent to proving \[J_1+J_2+J_3-J_4>0.\] We observe that $J_1, J_2, J_3, J_4 \geq 0$, and note that the expressions $J_1, J_2, J_3$ and  $J_4$ are very similar to the ones from the case $y \geq 3/2-1/p$. The differences in integral bounds is explained by the fact that when $y<2/3-1/p$ we have $y+2/3 \leq 2-y-2/p$. We see that 
		\[J_4=\int_{2-4/p}^{3/2+y-3/p}f(x,\xi,p) \, dx,\]
		where 
		\[f(x,\xi,p)=-\sin \frac{p}{2}\pi \sin p \pi (x-2)\left(\frac{1}{(\xi-x+7/2-5/p)^2}-\frac{1}{(\xi+x-1/2+3/p)^2}\right).\]
		To verify $J_1+J_2+J_3-J_4>0$ we will show that
		\begin{equation}\label{eq:221}
			\int_{2-4/p}^{\min(1,3/2+y-3/p)} f(x,\xi,p) \, dx< J_1,
		\end{equation}
		\begin{equation}\label{eq:222}
			\int_{1}^{\min(y+2/3,3/2+y-3/p)} f(x,\xi,p) \, dx< J_2,
		\end{equation}
		\begin{equation}\label{eq:223}
			\int_{y+2/3}^{\min(2-y-2/p,3/2+y-3/p)} f(x,\xi,p) \, dx\leq J_3,
		\end{equation}
		and 
		\begin{equation}\label{eq:224}
			\int_{1}^{\min(y+2/3, 3/2+y-3/p)} f(x,\xi,p) \, dx + \int_{2-y-2/p}^{3/2+y-3/p} f(x,\xi,p)  \, dx \leq J_2,
		\end{equation}
		where by a slight abuse of notation we consider the integral to be $0$ if the upper bound is smaller than the lower bound. We see that the inequalities \eqref{eq:221} and \eqref{eq:222} follow from the pointwise estimates \eqref{eq:22081} and \eqref{eq:22082}. Further for \eqref{eq:224} we can split the integral over $(2-y-2/p, 3/2+y-3/p)$ into two intervals, and then substitute to see that \eqref{eq:224} follows from \eqref{eq:220803}. It thus remains to show \eqref{eq:223}. In particular it is enough to show the pointwise estimate 
		\begin{equation}\label{eq:1553}
			\begin{split}
				&\frac{\sin p \pi(x-3/2)}{\sin p \pi(x-2)} \left(\frac{1}{(\xi-x+2-2/p)^2}-\frac{1}{(\xi+x)^2}\right)\\ - & \left(\frac{1}{(\xi-x+7/2-5/p)^2}-\frac{1}{(x+\xi-1/2+3/p)}\right) \geq 0 ,
			\end{split}
		\end{equation}
		for $x$ in $(1,5/3-2/p)$. Using that this is increasing in $x$ and the fact that $(\sin p \pi/2)/(\sin p \pi) \geq 1/2$ inequality \eqref{eq:1553} follows if
		\begin{equation}\label{eq:1554}
			\frac{1}{(\xi+1-2/p)^2}-\frac{1}{(\xi+1)^2}+ \frac{2}{(\xi+1/2+3/p)}-\frac{2}{(\xi+5/2-5/p)^2} \geq 0.
		\end{equation}
		One can verify that \eqref{eq:1554} holds by showing that this expression is increasing in $p$ or by rewriting it as a rational function. 
	\end{proof}
	\subsection{Proof of Lemma \ref{lem:pgeq36spes}}
	In this section we prove Lemma \ref{lem:pgeq36spes}. We start by observing that $\gamma$ is in $T_2$ since $\tau$ is. Thus what needs to be shown is \eqref{eq:appendixfinalpgeq36}. We define the following integrals.
	\begin{align*}
		I_1&=\int_{\xi}^{\xi-1/3+2/p}\sin^2{\frac{p}{2}\pi(x-k)}\left(\frac{1}{(x+2-4/p)^2}-\frac{1}{(x+1)^2}\right) dx,\\
		I_{2_1}&=\int_{\xi-1/3+2/p}^{\xi-7/12+3/p}\sin^2{\frac{p}{2}\pi(x-k)}\left(\frac{1}{x^2}-\frac{1}{(x+2-4/p)^2}\right) dx,\\
		I_{2_2}&=\int_{\xi-7/12+3/p}^{\xi-4/3+6/p}\sin^2{\frac{p}{2}\pi(x-k)}\left(\frac{1}{x^2}-\frac{1}{(x+2-4/p)^2}\right) dx, \\
		I_3&=\int_{\xi-4/3+6/p}^{\xi+2/p}\sin^2{\frac{p}{2}\pi(x-k)}\left(\frac{1}{(x+1-4/p)^2}-\frac{1}{x^2}\right) dx.
	\end{align*}
	\begin{proof}[Proof of inequality \eqref{eq:appendixfinalpgeq36}.]
		Proving \eqref{eq:appendixfinalpgeq36} is equivalent to showing $I_{2_1}+I_{2_2}-I_1-I_3>0$. In particular it suffices to show $I_{2_1} > I_1$ and $I_{2_2} > I_3$. Start with the former inequality, we observe that since $p > 3.6$ it follows that $-7/12+3/p \leq 1/p$ and hence $\sin^2(p/2 \pi(x-k))$ is an increasing function for $x$ in $(\xi, \xi-7/12+3/p)$. Thus $I_{2_1} > I_1$ will be implied by
		\begin{equation*}
			\int_{\xi}^{\xi-1/3+2/p}\frac{1}{(x+2-4/p)^2}-\frac{1}{(x+1)^2}dx < \int_{\xi-1/3+2/p}^{\xi-7/12+3/p}\frac{1}{x^2}-\frac{1}{(x+2-4/p)^2} dx.
		\end{equation*}
		Performing these integrals and eliminating equal terms we see that this is equivalent to
		\begin{equation*}
			\begin{split}
				&\frac{1}{\xi+17/12-1/p}-\frac{1}{\xi-7/12+3/p}+\frac{1}{\xi-1/3+2/p}\\ &-\frac{1}{\xi+2/3+2/p}+\frac{1}{\xi+1}-\frac{1}{\xi+2-4/p} > 0.
			\end{split}
		\end{equation*}
		Rearranging terms and dividing på $(4-p)$ we see that it suffices to show $f(\xi, p) >0$, where 
		\[\begin{split}
			f(\xi,p)&=1080p^3(p-3.6)\xi^3 +9p^2(191p^2-138p-1512)\xi^2\\&+3p(275p^3+1182p^2-4896p-2592)\xi \\ &+376p^4-1341p^3+7614p^2-21600p+7776.
		\end{split}\]
		Viewing $f(\xi,p)$ as a polynomial in $\xi$ we observe that all coefficients are positive so $f(\xi,p) > 0$, and thus $I_{2_1} > I_1$.
		
		We proceed to show  $I_{2_2} > I_3$, and begin by observing that for $x$ in the interval $(\xi-7/12+3/p, \xi-4/3+6/p)$ we have $\sin^2(p/2 \pi(x-k)) \geq 3/4$ . Combining this with the trivial estimate $\sin^2(p/2 \pi(x-k)) \leq 1$ we see that $I_{2_2} > I_3$ holds if
		\begin{equation}\label{eq:onepointeq2a}
			\frac{3}{4}\int_{\xi-7/12+3/p}^{\xi-4/3+6/p}\frac{1}{x^2}-\frac{1}{(x+2-4/p)^2} dx >  \int_{\xi-4/3+6/p}^{\xi+2/p}\frac{1}{(x+1-4/p)^2}-\frac{1}{x^2} dx.
		\end{equation}
		Evaluating these integrals we see that \eqref{eq:onepointeq2a} is equivalent to 
		\begin{align*}
			&\frac{3}{4}\left(\frac{1}{\xi+2/3+2/p}-\frac{1}{\xi-4/3+6/p}-\frac{1}{\xi+17/12-1/p}+\frac{1}{\xi-7/12+3/p}\right) \\
			& > \frac{1}{\xi+2/p}-\frac{1}{\xi+1-2/p}-\frac{1}{\xi-4/3+6/p}+\frac{1}{\xi-1/3+2/p}.
		\end{align*}
		Yet another rearrangement of terms and dividing by $(4-p)$ reveals that this follows if
		\[\begin{split}
			&27 (p-2)\left(\xi p -\frac{p}{3} + 2\right)(\xi p + p - 2)(\xi p + 2)\left(\xi p +\frac{p}{24}+\frac{5}{2}\right)\\
			&> 32 (p-3)\left(\xi p - \frac{7}{12}p  +3\right)\left(\xi p-\frac{p}{6}+2\right)\left(\xi p + \frac{17}{12}p-1\right)\left(\xi p + \frac{2}{3}p + 2\right).
		\end{split}\]
		Since $(\xi p +p/24+5/2) \geq (\xi p + 2p/3 + 2) $, we only need to show
		\[\begin{split}
			&27 (p-2)\left(\xi p - \frac{p}{3} + 2\right)(\xi p + p - 2)(\xi p + 2)\\
			&- 32 (p-3)\left(\xi p - \frac{7}{12}p  +3\right)\left(\xi p-\frac{p}{6}+2\right)\left(\xi p  + \frac{17}{12}p-1\right)> 0.
		\end{split}\]
		This is equivalent to 
		\[\begin{split}
			&(42-5p)p^3\xi ^3+\left(276-46p-\frac{10}{3} p^2\right)p^2\xi ^2+\left(312+236p-164p^2+\frac{197}{9} p^3\right)p\xi \\&-144+664p-\frac{1144}{3} p^2+\frac{665}{9} p^3-\frac{119}{27} p^4>0.
		\end{split}\]
		It is easy to check that this polynomial in $\xi$ has only positive coefficients, and thus the proof is concluded.
	\end{proof}
	\bibliographystyle{amsplain}
	\bibliography{pwp} 
\end{document}